\documentclass[]{article}
\usepackage{float}
\usepackage{soul}
\usepackage[margin=20mm,right=30mm]{geometry}
\usepackage{amsmath,amsthm,amssymb,mathrsfs}
\usepackage{enumitem} 
\setlist[description]
{noitemsep,labelindent=4ex}
\setlist{noitemsep}
\usepackage{graphicx}
\usepackage{import}
\usepackage{wrapfig}

\usepackage[font=small]{caption} 
\usepackage[rightcaption]{sidecap}

\usepackage{setspace}
\usepackage[backgroundcolor=white,bordercolor=white]{todonotes}
\graphicspath{{./figures/}}

\usepackage[colorlinks=true,linkcolor=blue,citecolor=blue]{hyperref}
\usepackage{tikz}
\usepackage{tikz-cd}

\numberwithin{equation}{section}
\newtheorem{counter}{Counter}[section]

\newtheorem{definition}[counter]{Definition}
\newtheorem{definition-proposition}[counter]{Definition-Proposition}
\newtheorem{theorem}[counter]{Theorem}
\newtheorem{lemma}[counter]{Lemma}
\newtheorem{corollary}[counter]{Corollary}
\newtheorem{example}[counter]{Example}
\newtheorem{proposition}[counter]{Proposition}
\newtheorem{remark}[counter]{Remark}

\newcommand{\idty}{{\mathrm{1}\mkern-4mu{\mathchoice{}{}{\mskip-0.5mu}{\mskip-1mu}}\mathrm{l}}}
\newcommand{\C}{\mathbb{C}}
\newcommand{\CC}{\mathbb{C}}
\newcommand{\DD}{\mathbb{D}}
\newcommand{\ZZ}{\mathbb{Z}}
\newcommand{\WW}{\mathbb{W}}

\DeclareMathOperator{\Hom}{Hom}
\DeclareMathOperator{\End}{End}
\DeclareMathOperator{\Fun}{Fun}
\DeclareMathOperator{\Rep}{Rep}
\DeclareMathOperator{\SL}{SL}

\newcommand{\K}{\mathcal{K}}
\DeclareMathOperator{\RT}{RT} 
\DeclareMathOperator{\Rib}{Rib}
\DeclareMathOperator{\Planar}{Planar} 

\newcommand{\G}{\mathcal{G}}
\renewcommand{\P}{\mathcal{P}}
\newcommand{\cB}{\mathcal{B}}
\newcommand{\cA}{\mathcal{A}}
\newcommand{\cC}{\mathcal{C}}
\renewcommand{\O}{\mathcal{O}}

\renewcommand{\mod}{\mathrm{mod}}
\newcommand{\Ch}{\operatorname{Ch}}

\DeclareMathOperator{\StratCob}{\mathbb{C}ob} 
\DeclareMathOperator{\Vect}{Vect}
\DeclareMathOperator{\Cat}{Cat}
\DeclareMathOperator{\Bimod}{Bimod}
\DeclareMathOperator{\QCoh}{QCoh}
\DeclareMathOperator{\Sk}{Sk}
\DeclareMathOperator{\SkMod}{SkMod}
\DeclareMathOperator{\SkCat}{SkCat}
\DeclareMathOperator{\SkAlg}{SkAlg}
\DeclareMathOperator{\BrTens}{BrTens}
\DeclareMathOperator{\Res}{Res} 
\DeclareMathOperator{\act}{act}
\DeclareMathOperator{\id}{id}

\DeclareMathOperator{\skel}{skel} 

\newcommand{\TetSurf}{\Sigma_{tet}} 
\newcommand{\Winv}{\mathbb{W}^{inv}_{\triangle}}
\newcommand{\Iloc}{\mathcal{I}_\triangle(K)}
\newcommand{\Mbulk}{M_\triangle}
\newcommand{\knotcomp}{S^3\setminus\overline{K}}

\newcommand{\cO}{\mathcal{O}}

\usepackage{xcolor}
\definecolor{our-orange}{HTML}{D86600}
\definecolor{our-blue}{HTML}{9BBBD6}
\definecolor{our-defect}{HTML}{77CCAA}
\definecolor{skein}{HTML}{8C0A60}

\usepackage{todonotes}

\title{Parabolic skein modules}
\author{Jennifer Brown and David Jordan}

\begin{document}

\maketitle

\begin{abstract}
We develop skein theory for 3-manifolds in the presence of codimension-one defects, focusing especially on defects arising from parabolic induction/restriction for quantum groups.  We use these defects as a model for the quantum decorated character stacks of \cite{JLSS2021}, thus extending them to 3-manifolds with surface defects.  As a special case we obtain knot invariants closely related to the ``quantum $A$-polynomial", and we give a concrete method for computation resembling the approach of Dimofte and collaborators based on ideal triangulations and gluing equations.
\end{abstract}


\tableofcontents

\section{Introduction}

The skein module of a 3-manifold $M$ is a vector space spanned by certain embedded and decorated 1-simplices in $M$, modulo locally defined linear relations modelled on the representation theory of a quantum group, or more generally a ribbon tensor category.  Skein modules famously give a model for Wilson loops and their quantization in Chern-Simons theory, and are at the heart of quantum topology, the branch of mathematics which applies techniques from quantum field theory to the study of low-dimensional topology.

In this paper, we initiate the study of skein modules of 3-manifolds in the presence of codimension-one defects.  We pay special attention to so-called parabolic defects, coming from functor of parabolic induction and restriction on quantum groups.  We apply parabolic defect skein modules to construct invariants of knots -- more generally of 3-manifolds with boundary -- and we provide an effective and elementary (cluster) algebraic formalism for computing with the resulting invariants.   Our main application is to give a new approach to the so-called ``quantum $A$-polynomial" invariant of knots, in the framework of extended TQFT with defects.

We give both an invariant definition and an effective method for doing computations relative to a triangulation of the manifold.  Our computational method recovers and refines a celebrated physical construction of Dimofte \cite{Dimofte2011} which inspired this work; our invariant construction also clarifies the apparent dependence on the choice of ideal triangulation in \textit{loc. cit.}  In our approach, there is an invariant skein-theoretic definition of a quantum $A$-ideal, and then the choice of an ideal triangulation gives a method for doing computations, with a tight control of the effect of the triangulation on the resulting computation.   While we believe our approach via parabolic defects to be new, let us mention a likely close relation to recent independent works \cite{Garoufalidis_Yu_2024} and \cite{Panitch_Park_2024}, which also compute skein modules of 3-manifolds relative to a triangulation. See Section \ref{sec:related} for more discussion of these and other related work.

While our primary focus for applications is to the quantum $A$-polynomial, we believe that the general notion of skein modules for 3-manifolds with surface defects developed here will have broad applications in many different settings.
Defect skein modules give a skein theoretic model for the well-known appearance of central tensor categories as domain walls between 4D Crane-Yetter type TQFT's determined by a braided tensor category, hence they are also closely related to domain walls between Witten-Reshetikhin-Turaev 3D theories.

\paragraph{The classical A-polynomial.}  The classical $A$-polynomial was constructed in \cite{Cooper_Culler_Gillet_Long_Shalen_1994}, and has been studied extensively (see for example \cite{Cooper_Long_1996, BoyerZhang1998,champanerkar2003polynomial, zickert2016ptolemy,HMP}). 
Given a knot $K$ in $S^3$ with knot complement $S^3\backslash K$ one studies the representation variety $\mathrm{Hom}(\pi_1(S^3\backslash K),\SL_2(\C))$, consisting of $\SL_2(\C)$-representations $\rho$ of the fundamental group of $S^3\backslash K$, and its subspace $\mathrm{Hom}^B(\pi_1(S^3\backslash K),\SL_2(\C))$, consisting of those $\rho$ whose value along paths in the boundary are constrained to lie in the standard Borel subgroup\footnote{The prominence of the representation variety, rather than the character variety, in the definition of the $A$-polynomial may be explained by the fact $m$ and $\ell$ can always be conjugated to lie in $B$,  uniquely so for generic $m$ and $\ell$.  Hence, requiring boundary monodromies to lie in the standard Borel is very close to quotienting by gauge equivalence in $G$, but without some of the baggage coming from a stacky $G$-action.}.  Such paths are generated by the meridian loop $m$ and longitude loop $\ell$ generating $\pi_1(T^2)$, where we identify $T^2=\partial(S^3\backslash K)$.

Denoting by $M$ and $L$, respectively, the upper left entries of the $2\times 2$-matrices $\rho(m), \rho(\ell)$, we obtain a restriction map,
\begin{align}\label{eqn:Res}
\operatorname{Res}:\mathrm{Hom}^B(\pi_1(S^3\backslash K),\SL_2(\C))&\to\mathbb{C}^\times\times \mathbb{C}^\times.\\
\rho\mapsto (M,L)\nonumber
\end{align}
The image of $\operatorname{Res}$ determines a one-dimensional sub-variety of $\C^\times \times \C^\times=\mathrm{Spec}(\C[M^{\pm 1},L^{\pm 1}])$, and its defining equation $A_K(M,L)=0$ is called the (classical) $A$-polynomial.  Despite its elementary definition, the classical $A$-polynomial is very difficult to compute exactly in practice; indeed the most substantial compendium of data on the classical $A$-polynomial was obtained by indirect sampling and interpolation methods \cite{Culler_A-polys}.

\paragraph{The quantum $A$-polynomial.} From several influential papers by research groups spanning both the mathematics and physics communities emerged a proposal for a certain ``quantum A-polynomial" knot invariant -- an element
\[A^q_K(M,L)\in \C_q[M^{\pm 1},L^{\pm 1}] := \CC(q)\langle M^{\pm 1}, L^{\pm 1} \mid LM = q ML \rangle,\]
which would $q$-deform the classical $A$-polynomial.  Due to its importance in the field of quantum topology, the quantum $A$-polynomial has by now far surpassed the classical $A$-polynomial as a focus of study -- for instance the vast majority of the 500+ papers citing \cite{Cooper_Culler_Gillet_Long_Shalen_1994} concern the quantum $A$-polynomial, and QFT-related techniques for its computation.

Among the proposals put forward to quantize the $A$-polynomial we highlight on the mathematical side (a) an approach via skein theory (see for instance \cite{Frohman_Gelca_LoFaro_2001,Gelca_2001}, more recently \cite{Panitch_Park_2024,Garoufalidis_Yu_2024}), (b) a related approach involving difference equations satisfied by the coloured Jones polynomial (sometimes called ``the AJ-conjecture") \cite{Garoufalidis_2004, Le2006, GarLe2005}.  On the physics side, we highlight (c) an approach using perturbative Lagrangian quantum field theory \cite{Gukov2005, Aganagic_Vafa_2012,Gukov_Sulkowski_2012}, another based on topological recursion \cite{dijkgraaf2009volume, borot2015all, Eynard_Garcia-Failde_Marchal_Orantin_2024}, and finally (e) an approach (also motivated by Lagrangian QFT, specifically quantum Chern-Simons theory) involving the combinatorics of ideal triangulations and the resulting ``gluing equations" \cite{Dimofte2011,DGG2016}. 
A number of generalisations of the quantum $A$-polynomial based on string/M-theoretic interpretations and BPS counting have been proposed \cite{Gukov_Sulkowski_Awata_Fuji_2012,Aganagic_Vafa_2012,Fuji_Gukov_Sulkowski_2013,Garoufalidis_Lauda_Le_2018,Ekholm_Ng_2020}

Our aim in this paper is to give a new approach synthesizing (a) and (e), using ideas from the theory of character stacks and extended topological field theory, specifically the papers \cite{JLSS2021}, \cite{GJS2021}.  We propose a manifestly well-defined topological invariant, and we give a method to compute it.  In future work, we intend to return to the relation of our work to approach (b).

\bigskip
The remainder of the introduction is organised as follows.  We first reformulate the classical $A$-polynomial via a certain moduli space of decorated local systems called the decorated character stack.  We then motivate a completion of the moduli decorated character stack called the redecorated character stack, and we discuss a system of coordinates on the latter generalising Wilson loops.  Motivated by these coordinates, we introduce the notion of parabolic skeins, and in this framework we define the quantum $A$-ideal which is our main object of study.  We further introduce a certain localisation of the quantum $A$-ideal depending on the choice of an ideal triangulation of the knot.  Finally, we explain an algorithm for computing the localised quantum $A$-ideal, and give a number of examples.

\subsection{Classical geometry}

\paragraph{Decorated local systems.}
A more geometric formulation of the $A$-polynomial may be given in terms of so-called decorated local systems.  We follow the framework of \cite{JLSS2021}, which is a stacky refinement of the well-known and closely related framework of Fock and Goncharov \cite{Fock_Goncharov_2009a,Fock_Goncharov_2009b} (see \cite[Remark 1.20]{JLSS2021} for a more careful comparison of the two).

Let $G$ be a reductive group.  By a $G$-local system on a manifold $M$, we mean a principal $G$-bundle $E$ over $M$ equipped with a flat connection, giving isomorphisms $E_x\to E_y$ for any homotopy class of paths connecting $x$ to $y$, compatibly with compositions of paths. The $G$-character stack $\operatorname{Ch}_G(M)$ of a manifold $M$ is the moduli stack of $G$-local systems on $M$; its GIT quotient variety is called the character variety.

Recall that a framing of $E$ at a point $x$ is an identification $E_x\cong G$ of $G$-torsors.  The representation variety, also known as the framed character variety, is the moduli stack\footnote{We note that when $M$ is connected the framing precludes the existence of automorphisms, ensuring that the moduli stack is indeed a plain variety.} of $G$-local systems equipped with a framing at a fixed basepoint of $M$.  By definition the character stack and the character variety are the stack quotient (respectively, GIT quotient) of the representation variety, by the (typically, non-free) $G$-action of changing framing.

Fix a Borel subgroup $B$ of $G$, and let $T=B/(B,B)$ denote the universal Cartan.  Further, fix a bipartition of $M$ into open regions $M_G$ and $M_T$, separated by a codimension one submanifold $M_B$.  By a \emph{decorated local system} we mean a triple $(E_G,E_T,E_B)$ consisting of a $G$-local system $E_G$ over $M_G$, a $T$-local system $E_T$ over $M_T$, and a sub-$B$-local system of $E_G\times E_T$ over $M_B$.  The data $E_B$ is equivalent to that of a section $\sigma_B$ of the $(G\times T)/B \simeq G/N$-bundle associated to $E_G\times E_T$ over $M_B$, so we may equally consider a triple $(E_G,E_T,\sigma_B)$. The decorated character stack is the moduli stack of such triples; its GIT quotient variety is the decorated character variety.  The decorated representation variety, also known as the framed decorated character variety, is the moduli stack we obtain by adding the data of a framing of $E_G$ or $E_T$ at some fixed basepoint in each connected open region of $M$.

Given a knot $K$ with tubular neighbourhood $\overline{K}$, and knot complement denoted $\knotcomp$, let us fix a torus $T^2$ contracting onto the boundary torus.  By construction this separates two open components, one being a homeomorphic copy of $S^3\backslash K$, the other being a neighbourhood $T^2\times I$ of the boundary.  We declare a bipartition of $\knotcomp$ by labelling the former region with $G$, the latter region with $T$, and hence the interface $T^2$ them with $B$.  We choose a pair of a $G$- and $T$-basepoint on either side of the defect, and consider the decorated representation variety.

We have a natural ``forgetful" map to the $T$-representation variety of $T^2$, which only remembers $E_T$ with its framing.  The $T$-representation variety of $T^2\times I$ is $\CC^\times \times  \CC^\times$, with coordinates $M$ and $L$ recording the holonomy of the meridian and longitude respectively. The affine closure of the image of the decorated representation variety coincides by definition with the image of the restriction map \eqref{eqn:Res}.  Hence we obtain an equivalent formulation of the classical $A$-polynomial, as the annihilator ideal of the constant function $1$ on the decorated representation variety, under the action by $\C[M^{\pm 1},L^{\pm 1}]$.

\paragraph{Redecorated local systems}
The Borel subgroup $B$ features in the definition of decorated local systems as a means to couple the $G$- and $T$-monodromies.  However, the representation theory of $B$ and its associated geometry presents us with three related obstacles:
\begin{itemize}
    \item \underline{Geometrically},  the moduli space $(G\times T)/B \simeq G/N$ of $B$-torsors in $G\times T$ is not affine. 
    \item \underline{Algebraically}, the group $B$ admits no non-zero compact projective representations.  In particular the category $\Rep B$ does not have enough compact projective objects.
   
    \item \underline{Topologically}, $\Rep B$ is not a 2-dualizable tensor category in the sense of \cite{douglas2020dualizable, brandenburg2014reflexivity, lurie2008classification}.
\end{itemize}

The geometric obstacle means that -- in contrast to ordinary representation varieties -- we cannot hope to directly construct decorated representation varieties and their quantizations explicitly using rings and modules; the passage to global sections will fail both conservativity and exactness.  The algebraic obstacle complicates any attempt at a skein-theoretic construction, because for non-semisimple categories such as $\Rep B$, skeins are typically coloured with compact-projective objects \cite{Brown_Haioun_2024}\footnote{To quote, ``compact objects are as necessary in this subject as air to breathe." \cite{neeman2001triangulated,BZFN}}.  Echoing this, the topological obstacle precludes defining any invariant of surfaces based on $\Rep B$ via the cobordism hypothesis.

Happily all three of obstacles have a common resolution, and one that is familiar in geometric representation theory.  Observe that we have an isomorphisms,
$\mathrm{pt} / B = G\backslash(G/N) /T
$ of stacks, and
hence we obtain standard equivalences of categories,
\[
\Rep B \simeq QC(\mathrm{pt}/B) \simeq  QC_{G\times T}(G/N)
\]

Throughout this paper, we let $\widetilde{G/N}=\mathrm{Spec} (\cO(G/N))$, and we let $\widetilde{\Rep}B$ denote the category of $G\times T$-equivariant $\cO(G/N)$-modules.  Then we have:
\begin{itemize}
     \item \underline{Geometrically}, the variety $\widetilde{G/N}$ is affine; indeed, by construction it is the affinisation of $G/N$.
    \item \underline{Algebraically}, the category $\widetilde{\Rep} B$ has enough compact-projectives; indeed, it is generated under colimits by the regular module $\mathcal{O}(G/N)$ and its twists by representations of $G$ and $T$.
    \item \underline{Topologically}, the tensor category $\widetilde{\Rep}B$ is 2-dualizable as a cp-rigid tensor category, as a 1-morphism in $\BrTens$ \cite{BJS2021}, and in particular as a domain wall between $\Rep G$ and $\Rep T$ (in the sense of \cite{Haioun_2024}).
\end{itemize}
Observe that $\widetilde{G/N}$ contains $G/N$ as a $G\times T$-equivariant open sub-variety; consequently $\widetilde{\Rep}B$ contains $\Rep B$ as an open subcategory, namely of those $G\times T$-equivariant sheaves supported on $G/N$.  Hence we may regard $\widetilde{\Rep} B$ as a sort of categorical completion of $\Rep B$, which features better geometric, algebraic and topological properties.

With these cheerful facts in mind, we define a \emph{redecorated local system} to be a triple $(E_G,E_T,\widetilde{\sigma_B})$, where $E_G$ and $E_T$ are as before, but where $\widetilde{\sigma_B}$ is instead a section of the associated $\widetilde{G/N}$-bundle to $E_G\times E_T$ along $M_B$. The redecorated character stack $\widetilde{\Ch^{dec}_G}(M)$ is the moduli stack of redecorated local systems, and we likewise define redecorated representation varieties to include the additional data of framings at basepoints.

\begin{example}
As an important motivating example, let us consider the case $G=\operatorname{SL}_2$.  In this case $G/N=\C^2\setminus \{0\}$, and hence $\widetilde{G/N}=\C^2$.  Once we fix $E_G$ and $E_T$, the data of $\sigma_B$ in a decorated local system is simply that of a nowhere vanishing section of the associated $\mathbb{C}^2$-bundle.  In a redecorated local system we allow also the zero section (noting that a flat section vanishing at one point of $M_B$ vanishes identically along $M_B$). When $\widetilde{\sigma_B}=0$, the $G$- and $T$-monodromies completely decouple, hence we have a decomposition:
\begin{equation}\label{eqn:redec-decomp}
\widetilde{\Ch_G^{dec}}(M) = \Ch_G^{dec}(M) \sqcup \left(\Ch_G(M_G)\times \Ch_T(M_T)\right),
\end{equation}
where on the left hand side we mean the redecorated character stack, the moduli stack of redecorated local systems.  This decomposition allows us to pass between decorated and redecorated character character stacks (similarly decorated and redecorated representation varieties) simply by taking care about the vanishing of $\widetilde{\sigma_B}\in \C^2$.\end{example}

\begin{remark} Nevertheless, some care must be taken when defining the $A$-polynomial after redecorating.  As before we have a ``forgetful" map,
\[
\widetilde{\Res}:\widetilde{\Ch_G^{dec}}(\knotcomp)\to \C^\times \times \C^\times. 
\]
However, due to the copy of $\Ch_T(M_T) = \CC^\times \times \CC^\times$ in the second component in the decomposition \eqref{eqn:redec-decomp}, this map is in fact \emph{surjective}, hence the associated annihilator ideal in $\C[L^{\pm 1},M^{\pm 1}]$ is the zero ideal. To get something meaningful, we must impose the non-vanishing of $\widetilde{\sigma_B}$.
\end{remark}

\begin{remark} We note that redecorated character stacks that we study here may be regarded as topological analogs of the so-called ``spaces of quasi-maps" introduced by Drinfeld in his study of parabolic local systems on projective curves.  We are grateful to David Ben-Zvi for pointing out this parallel, and encouraging us to embrace redecoration when formulating parabolic skeins.\end{remark}

\paragraph{Wilson loop coordinates}
Convenient coordinates on redecorated representation varieties may be obtained by generalising Wilson loop observables as follows (for simplicity we focus on $G=\SL_2$).  Let us consider a pair of basepoints $x_0^G$ and $x_0^T$, connected by a small arc $\delta$ passing through the $B$-defect.  Recall that a \emph{framing} of $E_G$ at $x_0^G$ (respectively, $E_T$ at $x_0^T$) is a trivialisation, $E_{x_0^G}\cong G$ (respectively $E_{x_0^T}\cong T$).  Associated to various curves in $M$ we have the following functions on the redecorated representation variety.
\vspace{12pt}

\begin{figure}[H]
\includegraphics[width=.45\linewidth]{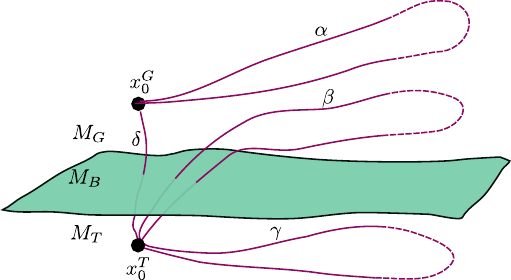}\hspace{1em} \includegraphics[width=.5\linewidth]{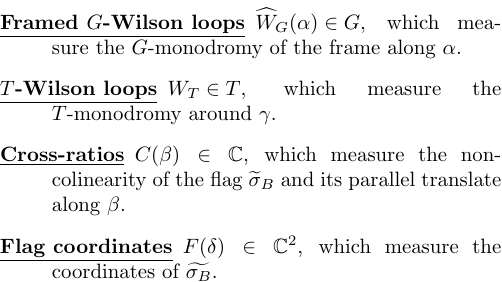}
\caption{Coordinates defined by some typical loops in $M$.}\label{fig:functions-on-Ch}
\end{figure}

Taking the trace of $\widehat{W}_G(\alpha)$ yields the ordinary Wilson loop observable $W_G(\alpha)\in\CC$, while in the case of $T=\C^\times$ the framed and ordinary Wilson loops coincide.  The flag $F(\delta)$ has two coordinates $x$ and $y$.  By requiring either $x$ or $y$ is non-zero we ensure we are on the decorated locus within the redecorated character stack.  We will typically set $y=0$ and $x\neq 0$, which corresponds to fixing the standard Borel.

\subsection{Quantum algebra}

With the preceding reformulation of the classical $A$-ideal in hand, we now turn to an overview of our quantization prescription via parabolic defect skein modules.

\paragraph{Parabolic skeins.}
Recall that the skein module $\SkMod(M)$ of a 3-manifold is a vector space spanned by certain labelled graphs embedded in the 3-manifold, modulo local ``skein relations", modelled on the graphical calculus of some ribbon tensor category.   Two typical examples of ribbon tensor categories of interest are the categories $\Rep_qG$ and $\Rep_qT$, the categories of integral representations of the quantum groups $U_q\mathfrak{g}$ and $U_q\mathfrak{t}$ respectively; the resulting skein theories are very well-studied, especially in the case $G=\SL_2$, and $T=\C^\times$.

A novel component in this paper is the introduction of what we call \emph{defect skein modules} (see \cite{jordan2022langlands, Meusburger2023} for other recent appearances of defects in skein theory). 
Fix a bipartite 3-manifold $M$ as above.
The defect skein module is spanned by labelled graphs in the 3-manifold, where the labels and skein relations are drawn from $\Rep_qG$ and $\Rep_qT$ over $M_G$ and $M_T$, and from $\Rep_qB$ over $M_B$.  Skeins descend from the open regions $M_G$ and $M_T$ into the defect $M_B$ via the parabolic restriction or induction functors, respectively.  We will typically impose a constraint that skeins intersect the $B$-defect transversely (see Figure \ref{fig:passing-through-defect} for an illustration, and see Section \ref{sec:defects-in-skein-theory} for precise definitions).
In \ref{sec:repqB-renormalisation} we show that the transversality condition very naturally reconstructs the category $\widetilde{\Rep_q}B$ of $U_q(\mathfrak{g})\otimes U_q(\mathfrak{t})$-equivariant $\O_q(G/N)$-modules.

\begin{figure}[h]
    \begin{center}
      \includegraphics[width=0.25\linewidth]{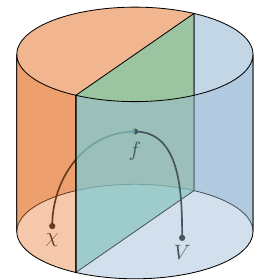}\hspace{20pt}
      \includegraphics[width=0.6\linewidth]{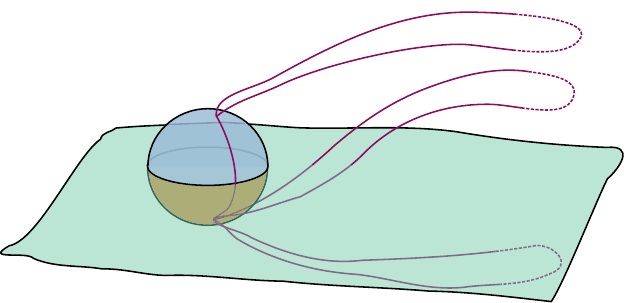}
      \caption{
        At left, a parabolic defect skein:  
        here $V$ and $\chi$ are objects of $\Rep_q G$ and $\Rep_q T$, respectively, and $f\in \Hom_{\Rep_qB}(\operatorname{Ind}_T^B(\chi),\Res_B^G(V))$.  At right, skeins in $M^\circ$ which quantize the functions from Figure \ref{fig:functions-on-Ch}.}\label{fig:passing-through-defect}
    \end{center}
  \end{figure}

We dub the category $\widetilde{\Rep_q}B$ the ``parabolic defect" because it gives a skein-theoretic implementation of parabolic induction/restriction functors in representation theory.  It follows from \cite{Cooke2019,Brown_Haioun_2024} that the \emph{defect skein categories} we construct agree with factorization homology.   The results of \cite{BZFN} therefore apply to give that our parabolic defect skein categories of surfaces recover the decorated quantum character stacks of \cite{JLSS2021} (we remark that the notion of redecoration was only implicit {\it a posteriori} in that work).  Furthermore on 3-manifolds, the results of \cite{BZFN} together with \cite{AFT2017} may be expected to give that our construction recovers the functions on the redecorated character stack of the 3-manifold.

In Lemma \ref{lemma:kauffman-presentation} we give a presentation of the parabolic defect skein module for $\SL_2$ in terms of concrete Kauffman/M\"uller type relations, which we expect will be comprehensible and familiar to skein theorists.
In this presentation, parabolic skeins consist of ribbon tangles, which are unoriented in the $G$-region, and oriented in the $T$-region.
In the $G$-region they satisfy the $\SL_2$ Kauffman bracket skein relations, and in the $T$-region they satisfy simpler $q$-crossing skein relations (see Figure \ref{fig:kauffman-crossings}).  At defects, they satisfy the relations in Figure \ref{fig:skein-defect-relations}.

\begin{figure}[h]\begin{center}
\includegraphics{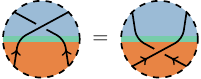}\hspace{20pt}\includegraphics[]{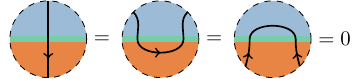}\end{center}
\caption{Additional skein relations satisfied at the parabolic defect.  The first relation is a consequence of stratified isotopy invariance, while the latter three relations are easy computations with parabolic induction and restriction.}
\label{fig:skein-defect-relations}
\end{figure}

The role of framings at basepoints in the geometric construction is played by ``internal skeins" ending at ``gates" in the skein module.
We recall from \cite{GJS2021} the notion of an internal skein module:  given a marked disk -- a ``gate" -- on the boundary of a 3-manifold, we obtain an ``internal skein object" which functorially describes skeins in the 3-manifold ending on the boundary, weighted by a vector in the representation labelling the endpoint.
In this sense internal skeins should be understood as a functorial version of the ``stated skeins" construction of Costantino and Lê \cite{Costantino_Le_2022} (see \cite{Haioun_2022} for a precise comparison of the two formalisms in the case $G=\SL_2$); we also remark a close relation between the relations in Figure \ref{fig:skein-defect-relations} and the ``reduced stated skein" relations.  A key difference however is that we allow the $T$-region to have topology, whereas only the case of contractible $T$-region would recover reduced stated skeins. 
The internal skein algebras and modules by construction carry a compatible action of the quantum group, and this action quantizes the action by changes of framing on the corresponding representation variety.
Whereas in \cite{GJS2021}, framings and internal skein algebras/modules were an intermediate tool to studying ordinary character varieties and their quantizations, in the present paper, as with \cite{JLSS2021}, the ``internal" objects are the primary focus, for the reason that the classical $A$-polynomial is defined via the representation variety rather than the character variety.

\paragraph{The quantum $A$-ideal}
Given a knot $K$, we first construct the knot complement $M=\knotcomp$ by removing a tubular neighbourhood $\overline{K}$ from $S^3$. 
As above, we fix $M_B$ to be any embedded $T^2$ which is disjoint from, but contracts onto, the boundary.  We define $M_T$ to be the resulting $T^2\times I$ trapped between $M_B$ and the boundary, and we define the ``bulk" manifold $M_G$ to be the remainder.
Hence, $M_G$ is simply a retract of $M$ away from its boundary.

The parabolic skein algebra attached to $M$ is naturally a module for $\C_q[M^{\pm 1},L^{\pm 1}] = \SkAlg_T(T^2)$, acting by inserting skeins at the boundary $T^2$.
The annihilator ideal of the empty skein would be the analogue of the classical $A$-ideal, except that this recovers the image of the redecorated character variety, and as we have remarked, this gives the $0$-ideal.  We need to implement, skein theoretically and in the $q$-deformed setting the requirement that $\sigma_B\neq 0$.

The solution comes by judicious use of framings. 
Let us now remove a small ball straddling the $B$-defect, hence intersecting $M_B$ at a 2-disk, and each of $M_G$ and $M_T$ in solid hemispheres (this will play an analogous role to the basepoint in the classical construction).  We denote the resulting manifold $M^\circ$.  We now introduce a $G$ and a $T$-gate in each respective hemisphere, and consider the internal skein module $\SkMod^{int}(M^\circ)$ (see Figure \ref{fig:passing-through-defect}).  The internal skein module is now a module for an algebra $\O_q(G/N)\otimes \C_q[M,L]$, where the $\O_q(G/N)=\C_q[x,y]$ acts by inserting internal skeins on the spherical boundary component.  We set $\widehat{\SkMod^{int}}(M^\circ)$ to be the module obtained by localising $x$ and setting $y=0$.  This is the quantum analogue of the requirement that monodromies lie in the \emph{standard} Borel subgroup as opposed to some conjugate.

\begin{definition}
The quantum $A$-ideal $\mathcal{I}(K)$ is the annihilator in $\C_q[M,L]$ of the module $\widehat{\SkMod^{int}}(M^\circ)$.
\end{definition}

We remark that, by construction, the quantum $A$-ideal at $q=1$ is precisely the classical $A$-ideal.

\paragraph{The localised quantum $A$-ideal}
Recall that any 3-manifold $M$ with boundary admits an ideal triangulation:  this is an expression of $M$ as a union of truncated tetrahedra, such that the triangles at the ``tips" of the tetrahedra form a triangulation of the boundary surface.
In what follows we refer to six ``long edges" and twelve ``short edges" of the truncated tetrahedra.  By a ``long edge skein" we mean a skein which has its ends in the $T$-region gate and passes through the $G$ region following the long edge of the triangulation.  By ``short edge skeins" we refer to skeins sitting entirely in the $T$-region which follow the short edge.  We may label the long-edge by an arbitrary simple object $V(m)$ of $\Rep_q \SL_2$.

\begin{definition}
Fix an ideal triangulation $\triangle$ of $\knotcomp$.    The localised quantum $A$-ideal $\mathcal{I}_\triangle(K)$ is the sum of the annihilator ideals of the long edge skeins with arbitrary labels.
\end{definition}

We make some remarks about this definition.  Firstly, by definition we have a containment $\mathcal{I}(K)\subset \mathcal{I}_\triangle(K)$. Secondly, when $q=1$, the long edge skeins represent functions which compute the cross-ratio of the monodromy along the edges.

Classically, we could simply invert these functions and consider the annihilator ideal of $1$ in the localised skein module.  When we quantise, however, we have only the skein module and not a skein algebra for 3-manifolds:  skeins which pass through the bulk of the 3-manifold cannot be multiplied and hence cannot be inverted.  Thus we cannot directly talk about inverting long element skeins.  In the above definition we bypass this by focusing on the annihilator ideal for the boundary action. However, in Section \ref{sec:quantum-A-polynomial} (and briefly detailed in the next subsection), we address the issue head-on by first removing a tubular neighbourhood of each long edge, and working with the internal skein algebra of the resulting surface $\Sigma_\triangle$, in which the long edge skeins are algebra elements and can be inverted formally.

In Section \ref{sec:quantum-A-polynomial}, we give an alternative presentation of $\Sigma_\triangle$ as glued from several four-punctured spheres.  In this formulation, $\mathcal{I}_{tet}$ is the left ideal generated by one Kauffman-type relation filling in the interior of each punctured sphere to give an ideal tetrahedron.  The two perspectives on $\Sigma_\triangle$ are indicated graphically in Figure \ref{fig:the-big-surface}.

\subsection{Main result}

Our main result is the explicit determination of the localised quantum $A$-ideal $\Iloc$ associated to any ideal triangulation of the knot complement $\knotcomp$.  Let $\triangle$ be an ideal triangulation involving $n$ tetrahedra.  A standard argument with Euler characteristics implies that $\triangle$ also has exactly $n$ edges.
In Section \ref{sec:localised-quantum-A-ideal}, we define a quantum torus $\Winv$ extending $\C_q[M,L]$ by $n-1$ edge-ratio generators $Y_1,\ldots Y_{n-1}$, together with $2n$ $T$-monodromy generators $r_1,\ldots r_{2n}$.
The quantum torus $\Winv$ arises as the $T$-invariant subalgebra of a localisation of the internal skein algebra of $\Sigma_\triangle$.
We define a left ideal $I_{bulk}$ and a right ideal $I_{thr}$ of $\Winv$.
Roughly speaking, $I_{bulk}$ fills in the interior of each tetrahedron via Kauffman type skein relations, while $I_{thr}$ fills in the tubular neighbourhood of the edges of the triangulation, and gets its name from the ``thread monodromies'' that generate it, see Lemma \ref{lemma:thread-monodromies}.
We have the following:

\begin{theorem}[Theorem \ref{thm:its-the-quantum-A-polynomial}]\label{them:main-result-intro}
We have an equality of left ideals,
\[
\Iloc = \C_q[M^{\pm1},L^{\pm1}] \cap \left(\mathcal{I}_{thr} + \mathcal{I}_{bulk}\right),
\]
between the quantum $A$-ideal and the joint elimination ideal of the tetrahedron relations and thread relations.  
\end{theorem}

\begin{remark} In more elementary terms, the intersection and reduction in Theorem 1.6 simply amounts to ``eliminating additional variables".  There is a complication, apparently a very interesting one, which is that we are eliminating modulo a sum of a left and a right ideal.  We note that the polynomial $\hat{A}_K$ can be computed in finite time using standard noncommutative Groebner basis methods.
\end{remark}

Our proof of Theorem \ref{them:main-result-intro} proceeds by constructing an intermediate 3-manifold tuned to $\triangle$: we remove from $S^3$ not only the tubular neighbourhood of the knot $K$ but also the tubular neighbourhood of the skeleton of the triangulation: now instead of a $T$-labelled torus, we have a $T$- and $G$-labelled surface $\Sigma_\triangle$ of genus $n+1$.  We form a composition of cobordisms,
\[
\knotcomp =  \underbrace{(D^2\times I)^{\sqcup 2n}_{T}}_{M_{thr}} \bigcup_{\mathrm{Ann}^{\sqcup 2n}} \underbrace{\left(\Sigma_\triangle\times I\right) \,\,\bigcup_{\Sigma_\triangle} \,\,(M_{tet}^{\sqcup n}/\varphi_{\triangle}^{})}_{\Mbulk}.
\]
Where the gluing map $\varphi_{\triangle} : \DD_3^{\sqcup 2n} \to \DD_3^{\sqcup 2n}$ dictates which faces $\DD_3 \subset M_{tet}$ are identified in the triangulation.

Applying the defect skein TFT, we have that \begin{equation}\label{eq:excision}
\begin{aligned}
    \SkMod(\knotcomp) &\cong \SkMod(M_{thr}) \bigotimes_{\SkCat(\mathrm{Ann}^{\sqcup 2n})}\SkMod \left(\Sigma_\triangle\times I\right) \bigotimes_{\SkCat(\Sigma_\triangle)} \SkMod(\Mbulk).
\end{aligned}
\end{equation}

The skeins running parallel to the long edges of the triangulation define an Ore set in the internal skein algebra $\SkAlg^{int}_T(\Sigma_\triangle)$. 
We define $\Winv$ to be the localisation at this Ore set. 
We define
\begin{equation}\label{eq:excision-loc}
\begin{aligned}
    \SkMod^{loc}(\knotcomp) &:= \SkMod(M_{thr}) \bigotimes_{\WW_{thr}}\Winv \bigotimes_{\Winv} \SkMod(\Mbulk).
\end{aligned}
\end{equation}
Here $\WW_{thr} := \C[r_1^{\pm 1},\ldots,r_{2n}^{\pm 1}]\subseteq \Winv$.

By construction we have a canonical homomorphism of $\C_q[M,L]$-modules $\SkMod(\knotcomp)\to \SkMod^{loc}(\knotcomp)$, and we identify the ideal $\mathcal{I}_\triangle(K)$ with the annihilator of the empty skein in $\SkMod^{loc}(\knotcomp)$.

\begin{figure}
    \centering
    \includegraphics[width=0.9\linewidth]{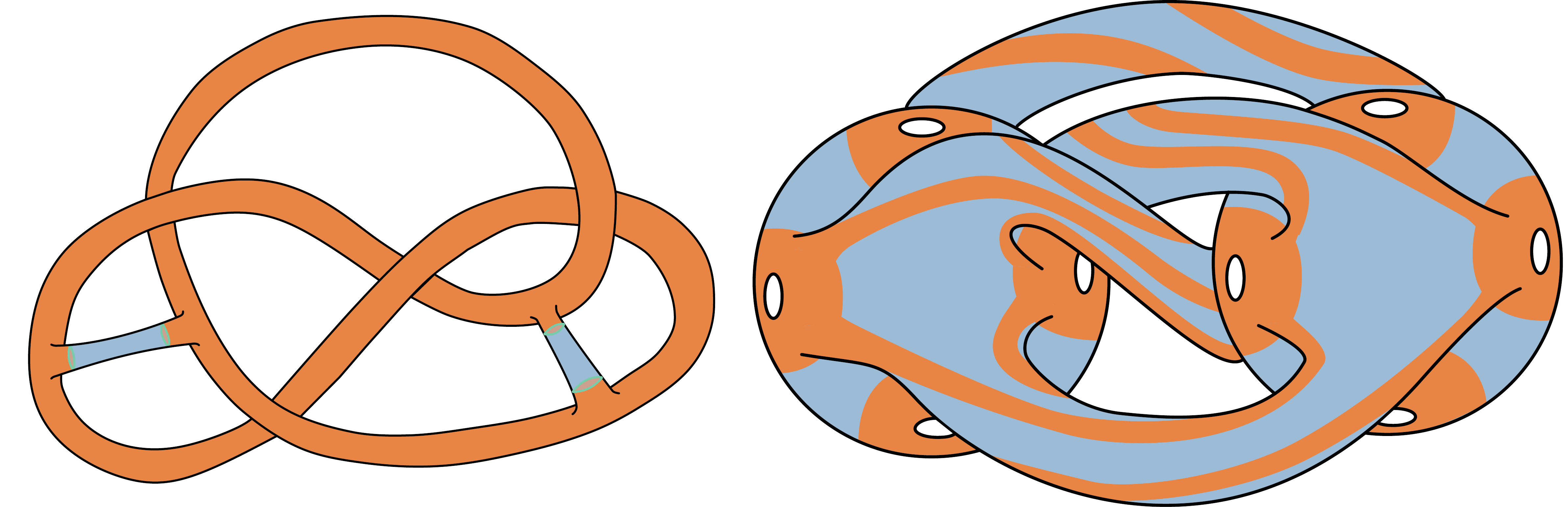}
    \caption{Two depictions of the same decorated surface $\Sigma_{\triangle}$ associated to the figure-eight knot and an ideal triangulation. The surface $\Sigma_\triangle$ consists of a $T$-region torus, onto which one $G$-region handle is attached for each long edge of the triangulation.  On the \textbf{left}, $\Sigma_\triangle$ is drawn emphasizing the shape of the knot as it sits in $S^3$.
    On the \textbf{right}, the same $\Sigma_\triangle$ is drawn to emphasize the role of the ideal triangulation. The eight punctures indicate the location of the tetrahedra's vertices, while the connecting cylinders show identified faces. }
    \label{fig:the-big-surface}
\end{figure}

 \begin{figure}
     \centering
     \includegraphics{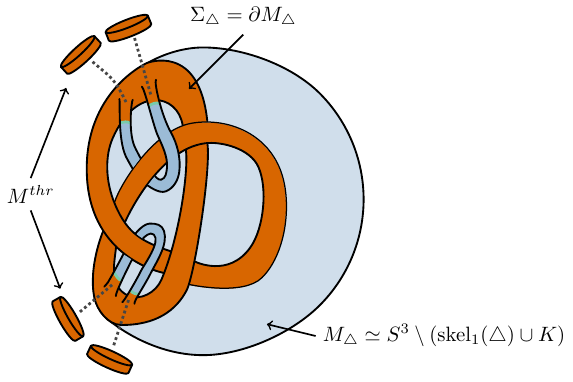}
     \caption{Attaching $T$-coloured disks to the boundary of $M^{bulk}$.}
     \label{fig:knot-complement-decomposition}
 \end{figure}

\subsection{Examples}
We close the introduction with a few illustrative examples.  More computations are done in Section \ref{sec:worked-examples}, and we have developed code in Sage and Singular which automates the process \cite{OurCode}.
Running our program with $q^{1/2} = -1$ produces the classical non-Abelian A-ideal. 
In the quantum version, the order in which we specialise thread variables and eliminate skeletal variables matters because the relevant variables do not commute. This is the source of the extra factors we see below in our computation for the trefoil knot.

In both the quantum and classical setting, the computationally expensive step of comes at the very end, in the form of non-commutative variable elimination.
Our limiting factor is the number of tetrahedron involved, which correlates roughly to the minimum number of crossings in the knot diagram.

In the current version of the code, the ideals $\mathcal{I}_{bulk}$ and $\mathcal{I}_{thr}$ are computed matter of minutes on a non-specialized computer for triangulations with as many as 11 tetrahedra. 
Performance improvements and large-scale computations will be the focus of future updates.

\begin{example}[The trefoil]
In Section \ref{sec:worked-examples} we compute that $\WW^{inv}_{\triangle}$ is a quantum torus with generators $L,M,Y,r$ subject to commutation relations:
\begin{equation}\begin{aligned}
  L  M = q  M  L,\qquad
 Y  L = q^{-1}  L  Y,\qquad
 Y  r = q  r  Y,
 \end{aligned}
 \end{equation}
The thread ideal $\mathcal{I}_{thr}$ is the right ideal generated by $(r-1)$.  The bulk ideal $I_{bulk}$ is the left ideal generated by
\begin{equation}
  \begin{aligned}
    B_1 &= q^{5} M^{-4} L r^{-1} + q^{7} r^{-1} Y + q^{3/2} \\
    B_2 &= q^{-9/2} M^6 L^{-1} Y^{-1} + q^{-5} M^2 r Y^{-1} + q^{3/2}
  \end{aligned}
\end{equation}
Eliminating $Y$ and $r$ is a straightforward computation yielding $\mathcal{I}_\triangle(3_1) = \langle\hat{A}_{3_1} \rangle$, where

\begin{equation}
    \hat{A}_{3_1} = q^7M^{10} + q^{15/2}M^6L -q^{11/2}M^4L - L^2 = (q^{11/2}M^4+L)(q^{3/2}M^6-L).
\end{equation}
\end{example}
\begin{remark}The factor $(q^{11/2}M^4+L)$ is purely quantum: if we set $q^{1/2}=-1$, then $\mathcal{I}_{tet}$ is a two-sided ideal, giving us more elimination moves because we can swap the order of the two factors of $\hat{A}_{3_1}$, and we recover the ideal $M^6+L$ --- the classical $A$-polynomial, less its Abelian factor $(L-1)$.
\end{remark}

\begin{example}[The figure-eight]
As an illustration of triangulation dependence, we consider in Section \ref{sec:worked-examples} two triangulations of $S^3\backslash \overline{4_1}$ related by a 2-3 Pachner move.  The first triangulation $\triangle_1$ features two tetrahedra, the second $\triangle_2$ features three tetrahedra.

In the smaller triangulation, the left ideal $I_{bulk}$ is generated by
 \begin{equation}\begin{aligned}
   B_1 =  q^{23/2} M^{-2} L^{-1} r^{-3} Y^{-1}+ q M^2 r Y  + q^{3/2},\hspace{1cm}&&
   B_2 = q^{1/2} L r Y + q r^{-1} Y^{-1} + q^{3/2}
 \end{aligned}\end{equation}
Eliminating the skeletal variable $Y$ and then specializing the thread monodromy $r$, we arrive at a non-principal ideal generated by
\begin{equation}
    \begin{aligned}
       g_1 &= q^{22}  M^8  L^3 -q^{17}  M^8  L -q^{16}  M^6  L^3 + (q^{37/2} - q^{29/2})  M^6  L^2 \\
       &+ q^{15/2} M^4  L^4 + q^{15}  M^6  L -(q^{12} + q^{10})  M^4  L^3 + (q^{17} + q^{15})  M^4  L \\
       &-q^{4}  M^2  L^3 + -q^{35/2}  M^4 + (-q^{25/2} + q^{17/2})  M^2  L^2 + q^{15}  M^2  L + L^3 -q^{15}  L\\
       \\
        g_2 &= q^{25/2}  M^{10}  L  -q^{28/2}  M^8  L^2  -q^{21/2}  M^8  L + (q^{10} + q^{6})  M^6  L^2 + (-q^{25/2} - q^{21/2})  M^6  L \\
        &-q^{7/2}  M^4  L^3 + q^{13} M^6 + (q^{8} + q^{6})  M^4  L^2 + (-q^{21/2} - q^{13/2})  M^4  L \\
        &+ q^{2}  M^2  L^2 + q^{21/2}  M^2  L - L^2.
    \end{aligned}
\end{equation}

Compare these with the classical $A$-polynomial $M^4+ L(-1 + M^2 + 2M^4 + M^6 - M^8)+ L^2M^4$.

\end{example}

\begin{example}[The unknot]
Already the trivial example of the unknot is illustrative: we see that when divisors arise with multiplicity in the $A$-ideal, these may need to be factored into the quantization.  We also see the detection of the Abelian component when computing with parabolic skeins.

With our conventions for colouring, the 3-manifold $S^3\backslash\overline{\mathrm{unknot}}$ is a solid torus, coloured with $G$ in its bulk region, and with $T$ on the boundary $T^2$, with a $B$-defect in the middle.  Let us consider a fibre of the standard fibration over the circle -- a disk with contractible $G$ bulk region, and an annular $T$-region at the boundary.

Recall that $\pi_1(S^3\backslash \mathrm{unknot})$ is simply $\mathbb{Z}$, generated by the meridian.  In particular, we have that $L=1$.  We will now proceed to show that $\mathcal{I}(\mathrm{unknot})=\langle(L-1)^2\rangle$, when $q=1$, and that for more general $q$ we have $\mathcal{I}(\mathrm{unknot})=\langle(L-q)(L-1/q)\rangle$.

\begin{figure}
    \centering
    \includegraphics{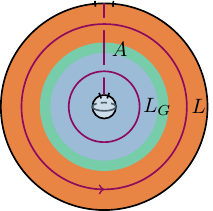}
    \caption{The skeins $A, L, L_G$ used in the computation of the quantum A-polynomial for the unknot, shown here in a cross section of the solid bipartite torus. Note that the Kauffman relations in the $G$ region imply that $L_GA = (L+L^{-1})A$.}
    \label{fig:unknot-example}
\end{figure}
The longitude in the $G$ region $L_G$ equals $q + q^{-1}$, we can write it in terms of the $T$-region longitude $L$ as follows.
We excise a small ball on the bulk, allowing ourselves to open a $G$-gate and a $T$-gate. Let $A$ denote the skein passing between these gates, labelled by the fundamental representation in the $G$ region and the one-dimensional weight +1 representation in the $T$ region, as shown in Figure \ref{fig:unknot-example}.
Then applying the crossing relation Figure \ref{fig:kauffman-crossings} then an isotopy, we get that $ L_GA = (L+L^{-1})A = (q+q^{-1})A$.
Unlike $L$, the skein $A$ is not naturally invertible. We get that $L+L^{-1} = q + q^{-1/2}$ in the localisation of the skein module of the unknot-complement. 
Multiplying through by $L$, we get $(L-q)(L-q^{-1}) = 0$.

\end{example}

\subsection{Relation to other works}\label{sec:related}
Our work is of course closely related to \cite{Dimofte2011,DGG2016,Garoufalidis_2004, Dimofte_Garoufalidis_2013}; indeed, recovering and refining their results was a major motivation.  We note that the surface $\Sigma_\triangle$ is reminiscent of the approach to abelianisation by adding additional handles to the boundary, which appeared in \cite{dimofte2014spectral}.

We draw considerable inspiration from the works \cite{gaiotto2013spectral,Hollands_Neitzke_2016,Gabella_2017,Neitzke_Yan_2020} and many others pertaining to abelianisation and non-abelianisation of skein algebras and skein modules using ideas from physics.  We hope to explore precise relations/applications of our work to abelianisation and spectral networks in a future work.

On the strictly mathematical side, we also mention the recent works \cite{Panitch_Park_2024, Garoufalidis_Yu_2024} which appeared during the preparation of the present manuscript appear to be closely related to this work.  It would be very interesting to understand precise relations between our work and theirs.

We also remark that at $q=1$ (more precisely at $q^{1/2}=-1$) we expect our localised quantum $A$-polynomial to be closely related -- perhaps equal to, up to normalisation -- that obtained in \cite{champanerkar2003polynomial, zickert2016ptolemy,HMP}, using considerations in hyperbolic geometry.

Beyond the relations to the $A$-polynomial, the idea to study skein modules with defects is certainly implicit in many works about stated skeins (e.g. \cite{Costantino_Le_2022,Mandel_Qin_2023}, indeed going back to \cite{Muller_2016} and the works \cite{Bonahon_Wong_2011a,Bonahon_Wong_2011b,Bonahon_Wong_2016,Bonahon_Wong_2017} of Bonahon--Wong on quantum trace).
However, to our knowledge these structures were not developed to the extent and detail that they are the present work.
In particular, the appearance of non-trivial topology in the $T$-region appears to be a new feature of our work, whereas a $T$-gate in a contractible $T$ region can be understood via the reduced stated skein algebras.

Codimension-two defects labelled by braided module categories have been proposed in \cite{jordan2022langlands}.

\subsection{Outline of paper}
In Section \ref{sec:defects-in-skein-theory} we introduce codimension-one defects in skein theory. Section \ref{sec:stratified-spaces} introduces the bipartite surfaces and three manifolds encountered in this work, while Section \ref{sec:local-coeffs} describes the algebraic data needed to define a defect skein theory and the resulting defect skein relations.

Defect skein categories and modules are introduced in Section \ref{sec:defect-skein-category}, then a brief review of some relevant categorical constructions in Section \ref{sec:various-cats} allows us to formulate the important gluing properties of skein categories and modules in Section \ref{sec:gluing-skein-cats}.
This is followed by the introduction of defect internal skein modules and algebras in \ref{sec:internal-constructions}, with some discussion changing gates and on their geometric interpretation.

Section \ref{sec:parabolic-defect} focuses on the parabolic defect powering our A-polynomial work. It opens with a description of the local coefficients associated to $\Rep_qB$ then gives a Kauffman-\"uller style presentation of parabolic skeins for $G=\SL_2\CC$.
Monadic reconstruction for the non-affine defect $\Rep_qB$ is described in Section \ref{sec:monadic-reconstruction}, and Section \ref{sec:repqB-renormalisation} describes the affinisation procedure discussed above.
Section \ref{sec:quantum-tori} introduces quantum tori, quantum cluster charts, and a useful Lemma for taking certain quotients in an efficient manner.
This part concludes with Section \ref{sec:building-blocks}, where the defect internal skein algebras of four basic bipartite surfaces are given presentations. 

Section \ref{sec:quantum-A-polynomial} focuses on the computation of the quantum A-polynomial which motivated the current work.
In Section \ref{sec:quantum-cluster-charts} we construct quantum cluster charts for $\SkAlg^{int}(\Sigma_\triangle)$ and $\SkMod^{int}(\Mbulk)$, then in Section \ref{sec:localised-quantum-A-ideal} we obtain the localised quantum $A$-ideal by restricting to an invariant sub-algebra and passing to a certain quotient.
Finally, Section \ref{sec:worked-examples} works through the computations in concrete detail for a handful of knots.

\subsection{Acknowledgements}
We are grateful to David Ben-Zvi, Tudor Dimofte, Sam Gunningham, Andy Neitzke, Sunghyuk Park, Pavel Safronov, Gus Schrader, Christoph Schweigert, and Sasha Shapiro for helpful discussions about this work.

The work of JB was partially funded by NSF grants DMS-2202753 and DMS-1753077. The work of DJ was supported by the EPSRC Open Fellowship ``Complex Quantum Topology", grant number EP/Y008812/1.  The work of both authors was supported by the Simons
Foundation award 888988 as part of the Simons Collaboration on Global Categorical Symmetry.

We are grateful to the Banff International Research Station, the International Centre for Mathematical Sciences, Beijing Institute of Mathematical Sciences and Applications, and Simons-Laufer Mathematical Sciences Institute, where the authors presented this work in 2023 and 2024, and to the Atlantic Algebra Centre at Memorial University and the SwissMap Research Station where JB has given lecture series about this work in 2024 and 2025.

\section{Defect skein theory}\label{sec:defects-in-skein-theory}

A central part of this work is to define skein theory of a bipartite surfaces and three-manifolds, as a (3,2)-TFT with defects, building on the defect-free case as developed in \cite{Walker_2006} and \cite{Johnson-Freyd_2021}, following \cite{RT1990}.

\subsection{Stratified spaces}\label{sec:stratified-spaces}
We follow the conventions for stratified spaces used in \cite[\S 2]{AFT2017local}.  A \emph{stratified space} is a paracompact Hausdorff space $X$ together with a continuous map $\varphi: X \to P$ to a poset $P$, where $P$ is endowed with a topology whose closed sets are generated by $P_{\leq p}$ for $p \in P$.
We use the notation $X_p := \varphi^{-1}(p)$ for the individual \emph{strata}.
\begin{remark}
    When $\dim (X_p) < \dim(X)$, $X_p$ is often referred to as a ``defect of codimension $\dim(X)-\dim (X_p)$", as a ``surface defect" (when $\dim(X_p)=2$) or ``line defect" (when $\dim(X_p)=1$) or ``point defect" (when $\dim(X_p)=0)$.  When $\dim(X)-\dim(X_p)=1$, $X_p$ is sometimes called an ``interface" or ``domain wall".
\end{remark}

A map between stratified spaces $X\to P$ and $Y\to Q$ is a pair $F_1 : X \to Y$ and $F_2 :P \to Q$ of continuous maps such that the following square commutes:
\begin{equation}
  \begin{tikzcd}
    X \ar[r,"F_1"] \ar[d] & Y \ar[d]  \\
    P \ar[r,"F_2"]        & Q
  \end{tikzcd}
\end{equation}
We call such a map an embedding if both $F_1$ and each $F_1|_{p}^{} : X_p \to Y_{F_2(p)}$ is an embedding.
An isotopy between two embeddings $F,G$ of $X \to P$ into $Y \to Q$ is a pair of continuous maps $H_1, H_2$ such that
\begin{equation}
  \begin{tikzcd}
    X \times [0,1] \ar[r,"H_1"] \ar[d] & Y \ar[d] \\
    P \ar[r,"H_2"] & Q
  \end{tikzcd}
\end{equation}
where $X\times [0,1] \to P$ depends only on the first factor and $H_1(\cdot,0) = F_1$, $H_1(\cdot,1) = G_1$.
In particular, this implies that $H_2 = F_2 = G_2$.  We say that $F$ and $G$ are isotopic if there exists an isotopy between them.

In this paper we will consider surfaces and 3-manifolds stratified over the 3-element poset on letters $A,B,C$ with $A \leq B \geq C$ (so that $A$ and $C$ are not comparable in the poset).  We are moreover in the favorable condition that the strata $X_A,X_B,X_C$ are each smooth submanifolds of co-dimension $0$, $1$, $0$, respectively.

We often denote such spaces as unions of such regions glued along defects, i.e. $X = X_a \cup_{X_b} X_c$ is shorthand for the stratified space $X \overset{p}{\to} \{B \leq A, B \leq C\}$ with $p^{-1}(A) = X_a, p^{-1}(C) = X_{c}$ both codimension zero and $p^{-1}(B) = X_b$ a (possibly disconnected) embedded surface.  Such a defect $X_b$ is sometimes called an \emph{interface} or \emph{domain wall} in the physics literature.

\begin{figure}
    \centering
    \includegraphics[]{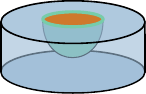}
    \caption{A thickened disk with the structure of a bipartite three manifold. Note that the defect (teal) meets the upper boundary transversely.}
    \label{fig:bipartite}
\end{figure}

\begin{definition}\label{def:bipartite-3-manifold} A \emph{\bf bipartite 3-manifold} consists of the data of a 3-manifold $M$, with a continuous map $\phi:M\to P=(A \leq B \geq C)$, with the property that $M_B := \phi^{-1}(B)$ is a (possibly disconnected) embedded oriented surface, whose outward normal always points towards an $A$ region and inward normal towards a $C$ region. We allow 3-manifolds with boundary, and allow the $M_B$-strata to meet the boundary transversely.
\end{definition}

\begin{definition} A \emph{\bf bipartite surface} consists of the data of a surface $\Sigma$, with a continuous map to $\phi:\Sigma\to P=(A \leq B \geq C)$, with the property that $\Sigma_B$ is a (possibly disconnected) smoothly embedded curve whose outward normal always points towards an $A$ region and inward normal towards a $C$ region. We allow surfaces to have boundary, and allow the $\Sigma_B$-strata to meet the boundary transversely.
\end{definition}

\begin{definition}\label{def:bipartite-graph} An \emph{\bf an embedded bipartite ribbon graph} in a bipartite 3-manifold $M$ consists of a graph $\Gamma$, together with:
\begin{itemize}
    \item A stratified embedding $\iota:(\Gamma,P_{graph}) \hookrightarrow (M,P)$, where $\Gamma \to P_{graph}$ is the standard graph stratification $\Gamma$ for which $v\leq e$ whenever $v$ is an end vertex of an edge $e$.
    \item A stratified framing on $\iota(\Gamma) \subset M$,
    i.e. a choice of non-vanishing section of the restricted normal bundle $N\iota(e)\cap T M_p$ along each edge $e$ of $\Gamma$, where $M_p$ is the strata containing the edge $e$.
    \item A cyclic ordering on the half-edges adjacent to each vertex.
\end{itemize}

\end{definition}
\begin{remark}
    Note that the data of the stratified embedding $\iota$ includes a labelling $\Gamma \to \{ A \geq B \leq C\}$ such that if an edge $e$ is $B$-labelled then so are its vertices, but $A$-labelled (resp. $C$-labelled) edges can have vertices in either $B$ or $A$ (resp. $C$).
\end{remark}

\subsection{Local coefficients}\label{sec:local-coeffs}
  
  Our constructions will involve a pair of ribbon tensor categories $\cA$ and $\cC$ assigned to the open $A$- and $C$- regions, respectively, together with the data of a pivotal tensor category $\cB$ assigned to the defect.  There is a further necessary structure relating $\cA$, $\cB$ and $\cC$, captured in the following definition.

    \begin{definition}
    Fix ribbon tensor categories $\cA$ and $\cC$.  A pivotal $(\cA,\cC)$-central tensor category is a pivotal tensor category $\cB$ together with a pivotal braided tensor functor $(H,J,c):\cA\boxtimes \overline{\cC}\to\mathcal{Z}(B)$, where $H : \cA \boxtimes \cC \to \cB$ is the underlying functor, $J : H(-\otimes -) \to H(-)\otimes H(-)$ is the tensor structure, and $c_{x,y,z} : H(x\boxtimes y) \otimes z \to z \boxtimes H(x\boxtimes y)$ is the half-braiding.
    \end{definition}
    Here $\mathcal{Z}(\cB)$ denotes the Drinfeld center of $\cB$ equipped with its canonical pivotal braided tensor structure, while $\overline{\cC}$ denotes the tensor category $\cC$ equipped with the opposite braiding and inverse ribbon structure.
We will abbreviate simply by $\cB$ the data
\begin{equation}\label{eq:ABC-def}
  \cB := \big(\cA, \cB,\cC, (H,J,c) : \cA\boxtimes \overline{\cC} \to Z(\cB)\big).
\end{equation}
Following \cite[Figure 2]{BJS2021} and \cite{FSV2013}, we will use the central structure to determine how skeins are `pushed' from the bulk into the defect.
We package this information as a functor analogous to the aforementioned braided tensor functors $\RT_\cA : \Rib_\cA \to \cA$ and $\RT_\cC : \Rib_\cC \to \cC$ introduced in \cite{RT1990}.

When $\cA = \cC = \Vect$ and the central structure is given by $V\boxtimes W \mapsto (V\otimes W)\otimes \idty_\cB$, the defect theory will reduce to the two-dimensional skein theory of the pivotal tensor category $\cB$.
We therefore start by reviewing the two-dimensional theory, following the approach laid out in \cite[\S 3.1]{GPV2012}.

The algebraic input is in general a pivotal cp-rigid tensor category $\cB$, which is not necessarily braided.  Recall that cp-rigid means that the compact-projective objects of $\cB$ are left- and right-dualizable, and that we have fixed a monoidal isomorphism between the functors of taking left and right double duals.

\begin{definition}\label{def:planar-graph-cat}
  Let $\cB$ be a pivotal category. $\Planar_\cB$ is the monoidal category of planar graphs colored by $\cB$. It has:
  \begin{description}
    \item \emph{\underline{objects}}
      Finite collections $\{(x_i,V_i,\epsilon_i)\}$ of signed (denoted $\epsilon_i\in\{\pm 1\}$), framed points (denoted $x_i\in I$), each one labelled with an object $V_i$ of $\cB$.
    \item \emph{\underline{morphisms:}}
      A morphism from $\{(x_i,V_i,\epsilon_i)\}$ to $\{(y_j,W_j,\epsilon_j)\}$ is an oriented planar graph in $\mathbb{R}\times[0,1]$,
      with edges colored by objects in $\cB$ and coupons (i.e. vertices) labeled by morphisms from the tensor product of incoming edges to that of outgoing edges.
      The orientation and color of edges at the boundary $\mathbb{R} \times \{0,1\}$ must agree with the colors and signs of the source and target objects.
      These graphs are considered up to isotopy.

      We will depict morphisms as going upward, from $\mathbb{R}\times \{0\}$ to $\mathbb{R} \times \{1\}$.
      Composition is given by vertical stacking followed by a rescaling $\mathbb{R} \times [0,2] \overset{\cong}{\to} \mathbb{R} \times [0,1]$.
    \item \emph{\underline{product:}}
      The monoidal structure is given by disjoint union, so that $T_1 \otimes T_2$ is $T_1$ placed to the left of $T_2$.   The empty graph is the unit.

    \item \emph{\underline{duality:}}
      The dual of $\eta = ((V_1, \varepsilon_1),\ldots,(V_n,\varepsilon_n))$, is $\eta^* := ((V_n, -\varepsilon_n),\ldots,(V_1,-\varepsilon_1))$.
      Coevaluation and evaluation are given by oriented cups and caps.
      We will draw the pivotal isomorphism $(V,+)^{**} = (V,-)^* = (V,+)$ as the identity.
  \end{description}
\end{definition}
\begin{remark}
This is essentially a version of $\Rib_\cA$ (called $\mathrm{HCDR}(\cA)$ in \cite[\S 4.5]{RT1990}) without the extra dimension necessary for braidings or differences in framing to occur.
\end{remark}

In \ref{def:defect-RT} we will use the planar diagrammatics of pivotal categories \cite{Selinger2011} to define skein relations near a planar defect.

    \begin{definition}\label{def:rel-defect-labels}
    A \emph{\bf$\cB$-labelling} of a bipartite surface is a finite collection $\{(x_i,V_i,\epsilon_i)\}$ of signed (denoted $\epsilon_i\in\{\pm 1\}$), framed points $x_i$, each one labelled with an object $V_i$, of $\cA$, $\cB$, or $\cC$, according to which region the point $x_i$ occupies. 
    \end{definition}

    \begin{definition}\label{def:b-labelling}
    A \emph{\bf $\cB$-coloring} on an embedded bipartite ribbon graph (definition \ref{def:bipartite-graph}) is an assignment of an object $V_e$ of $\cA$, $\cB$, or $\cC$ for each edge $e \in E(\Gamma)$, always chosen from the category associated to $e$, and a morphism $f_v$ for each vertex $v \in V(\Gamma)\}$. The source and target for $f_v$ is determined as follows.

    For $v \notin M_B$ not in the defect, we let $h(v)$ denote the cyclically ordered set of half edges incident to $v$, and we require a morphism
    \begin{equation*}
        f_v : \bigotimes_{e \in h(v)}V_e^{\epsilon_e} \to \idty
    \end{equation*}
    (in $\cA$ or $\cC$ according to the region) where $V_e^{\epsilon_e}  = V_e$ for incoming half-edges and $V_e^*$ for outgoing half-edges, and $\idty$ is the tensor unit in $\cA$ or $\cC$, according to the region.
    
    For $v \in M_B$ we let $h_A(v)$ denote the half-edges in $M_A$, $h_B(v)$ those in $M_B$, and $h_C(v)$ those in $M_C$, and we require:
    \begin{equation*}
        f_v : H\left( \bigotimes_{e \in h_A(v)}\hspace{-1ex} V_e^{\epsilon_e} \boxtimes \bigotimes_{e \in h_C(v)}\hspace{-1ex} V_e^{\epsilon_e}\right)\otimes \bigotimes_{e \in h_B(v)}\hspace{-1ex}V_e^{\epsilon_e} \to \idty_\cB.
    \end{equation*}
    \end{definition}

Here $H : \cA\boxtimes \cC \to \cB$ is the tensor functor of Definition \ref{eq:ABC-def}.

  \begin{definition}\label{def:defect-ribbons}
    The  category $\Rib_\cB$ of $\cB$-coloured embedded bipartite ribbon graphs \emph{\bf near the defect} has:
    \begin{description}
      \item \emph{\underline{Objects}:} 
    $\cB$-labelings of $\DD_B$.
        See Figure \ref{fig:RT-B-on-objects}.

      \item\emph{\underline{Morphisms:}} $\cB$-coloured bipartite ribbon graphs embedded in $\DD_B \times [0,1]$, compatibly with given $\cB$-labelling of $\DD_B\times \{0,1\}$.
        See Figure \ref{fig:passing-through-defect}.
      \item \emph{\underline{Monoidal Structure}}
        induced by a embedding $\DD_B \sqcup \DD_B \hookrightarrow \DD_B$.
        The unit $\idty$ is the empty object.
      \item \emph{\underline{Duality:}}
        Fix an object $X$ in $\Rib_\cB$.
        Its dual $X^*$ is obtained by reversing the orientation, i.e. reversing the signs on all marked points.
        Coevaluation $\idty \to X\otimes X^*$ is given by connecting each marked point with its partner by an arc that remains in a single stratum.
        We can ensure the arcs are not entangled, e.g. by representing them by flat half-circles.
        Evaluation $X^*\otimes X \to \idty$ is described similarly.
\end{description}
  \end{definition}
For the coefficient system considered in Section \ref{sec:parabolic-defect}, the pivotal category assigned to the defect is generated under co-limits by the image of the central structure.
    This means the hom-spaces in our skein categories will be generated by diagrams that are at most transverse to the defect.  This motivates the following definition.

\begin{definition}
    The subcategory $\Rib^\pitchfork_\cB \subset \Rib_\cB$ has as its objects those $\cB$-labelings none of whose points $x_i$ lie on the $B$-defect, and as its morphisms only those ribbon graphs whose edges intersect the $B$ defect only at their end-vertices.
\end{definition}

\begin{remark}
    As a caution we note that stratified isotopy is a weaker equivalence relation than isotopy relative to the defect, since the former allows the location where a strand meets a defect to change.
    For example, the tangles in Figure \ref{fig:crossing-before-or-after-defect} are stratified-isotopic but not isotopic relative to the defect.
    A coupon in the defect may have edges transverse to the surface, see Figure \ref{fig:passing-through-defect}.
    In this case we implicitly pre-compose with the central structure and think of the underlying morphism as being between objects in $\cB$.
\end{remark}    

    \begin{remark}\label{rmk:embeddings-make-functors}
      Any bipartite embedding $\sqcup_{i} \DD_{R_i} \hookrightarrow \sqcup_{j} \DD_{S_j}$ between finite collections of disks induces a functor between the Deligne-Kelly tensor products of the associated categories of ribbon graphs.
      This is because coloured bipartite ribbon graphs are themselves defined by bipartite embeddings and so the pullback of a morphism along a disk embedding is again a morphism.
      Similarly, a (bipartite) isotopy between two bipartite embeddings induces a natural transformation in the category of ribbon graphs.
      We will make frequent use of such induced functors and natural transformations throughout the paper.
    \end{remark}

    \begin{definition-proposition}\label{lemma:defect-ribbon-properties}
      $\Rib_\cB$ carries a pivotal $(\Rib_\cA,\Rib_\cC)$-central algebra, defined as follows:
    \end{definition-proposition}
    \begin{itemize}
    \item  The dual of an object $X \in \Rib_\cB$ is obtained by switching the sign on all points.
    
    \item  The pivotal structure is given by the canonical (identity) isomorphism $i_X : X^{**} \overset{\cong}{\longrightarrow} X$.

      \item The pivotal central structure consists of a pivotal tensor functor, 
      \[H^{rib} :\Rib_\cA \boxtimes \overline{\Rib}_\cC \to \Rib_\cB, \qquad J^{rib} : H^{rib}(-\otimes-) \to H^{rib}(-)\otimes H^{rib}(-)\] together with a half braiding $c^{rib}$ lifting this to the Drinfeld centre.

      Following \cite[Fig. 2]{BJS2021}, this is constructed by lifting the functor induced by the embedding $\DD_A \sqcup \overline{\DD}_C \hookrightarrow \DD_B$ to the Drinfeld centre.
      The half-braiding is induced by an isotopy between the two embeddings $\DD_A \sqcup \DD_B \sqcup \DD_C \hookrightarrow \DD_C$ shown in Figure \ref{fig:induced-half-braiding}.
      
      \item Compatibility with the pivotal structure on $\Rib_\cA$ and $\Rib_\cC$ follows from the fact that the pivotal structure on $\Rib_\cA$ and $\Rib_\cC$ is constructed in the same way as above.
      \end{itemize}

        \begin{figure}
      \begin{center}
        \includegraphics[scale=.7]{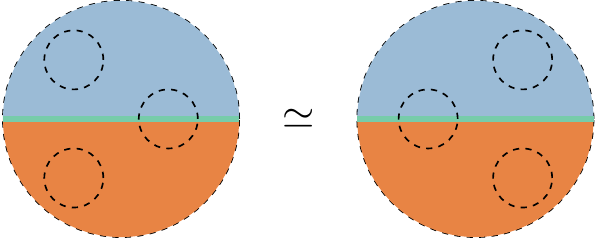}
        \caption{An isotopy induces the half braiding that gives $\Rib_\cB$ a $(\Rib_\cA,\Rib_\cC)$-central structure.
        }\label{fig:induced-half-braiding}
      \end{center}
        \end{figure}

  \begin{figure}
    \begin{center}
      \includegraphics{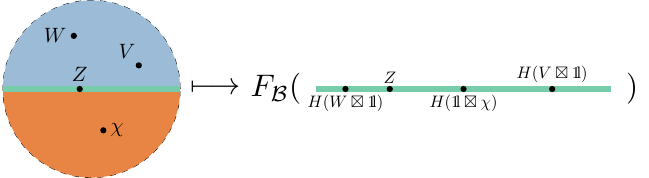}
      \caption{The functor $\RT_\cB$ on objects of $\Rib_\cB$.
      Here $H$ comes from the central structure $\cA \boxtimes \overline{\cC} \to Z(\cB)$ and $F_\cB$ is the evaluation functor for the planar diagrammatics of $\cB$.}
        \label{fig:RT-B-on-objects}
    \end{center}
  \end{figure}

  \begin{definition-proposition}\label{def:defect-RT}
 The following procedure determines an essentially surjective, full, pivotal monoidal functor $\RT_\cB : \Rib_\cB \to \cB$:

  Given a defect ribbon graph $T$ in $\DD_B\times [0,1]$, first pick a representative in generic position with respect to the orthogonal projection onto the defect, so that all crossings are transverse and coupons are not identified by the projection.
  Project this representative onto the defect, replacing all crossings with a coupon labelled by the appropriate half-crossing in $Z(\cB)$.
    Extend linearly to finite sums of defect tangles in $\DD_B\times [0,1]$.

    Use the planar graphical calculus of the pivotal category $\cB$ to evaluate the resulting (finite sums of) graphs on the defect as morphisms in $\cB$.
    \end{definition-proposition}
See \cite[\S 3.1]{GPV2012} and \cite[\S 4.4.2]{Selinger2011} for the planar diagrammatics of pivotal categories.

      \begin{figure}
        \begin{center}
          \includegraphics{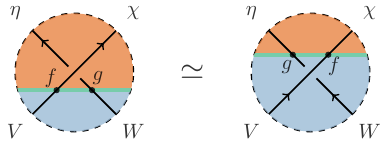}
          \caption{A birds-eye view of isotopic tangles, demonstrating that crossings can move through defects.
          The coupons can move past each other because the isotopy takes place in a thickening of the disk.}\label{fig:crossing-before-or-after-defect}
        \end{center}
      \end{figure}
    \begin{proof}
      We start by showing $\RT_\cB$ is well defined, i.e. it respects isotopies of the underlying stratified ribbon graphs and doesn't depend on the projection of the bulk ribbon graphs into the defect.
      It suffices to show that it is invariant under stratified versions of the Reidemeister moves.
      Let $H: \cA \boxtimes \overline{\cC} \to \cB$ 
      be the functor associated with the central structure. We use $\sigma^\cA, \sigma^\cC$, etc. to denote the braiding.

      \begin{description}
        \item[Crossings can pass through defects:]

          We show that the isotopic graphs in Figure \ref{fig:crossing-before-or-after-defect} are sent to the same morphism in $\cB$.
          Let $f : H(V\boxtimes \idty_\cC) \to H(\idty_\cA\boxtimes \chi)$ and $g : H(W \boxtimes \idty_\cC) \to H(\idty_\cA \boxtimes \eta)$ be the labels for the two coupons in the tangles shown in Figure \ref{fig:crossing-before-or-after-defect}.
          Then $\RT_\cB$ maps the left tangle of Figure \ref{fig:crossing-before-or-after-defect} to the morphism
          \begin{equation}\label{eq:cross-after-defect}
            \begin{tikzcd}
              H(V \!\boxtimes\! \idty) \!\otimes\! H(W\!\boxtimes\! \idty) \ar[r,"f\otimes g"]
              & H(\idty \!\boxtimes\! \chi)\! \otimes\! H(\idty \!\boxtimes\! \eta)
              \cong H(\idty \!\boxtimes\! \chi \!\otimes\!\eta) \ar[r,"H\left(\mathrm{id}\boxtimes \sigma^{\overline{\cC}}\right)"]
              &[7mm] H(\idty\!\boxtimes\! \eta\!\otimes\!\chi)
            \end{tikzcd}
          \end{equation}
          While the tangle on the right is sent to
          \begin{equation}\label{eq:cross-before-defect}
            \begin{tikzcd}
              H(V\!\otimes\! W \!\boxtimes\! \idty) \ar[r,"H(\sigma^\cA \boxtimes \mathrm{id})"]
              &[7mm] H(W\!\otimes\! V \!\boxtimes\! \idty)
              \cong H(W \!\boxtimes\! \idty) \!\otimes\! H(V\!\boxtimes\! \idty) \ar[r,"g\otimes f"]
              & H(\idty \!\boxtimes\! \eta)\! \otimes\! H(\idty \!\boxtimes\! \chi)
            \end{tikzcd}
          \end{equation}

          By assumption $H$ respects the braided tensor structure, so $\sigma^{Z(\cB)} = H\left(\sigma^{\cA\boxtimes \overline{\cC}}\right) = H\left(\sigma^\cA \boxtimes \sigma^{\overline{\cC}}\right)$.
          It follows that
          \begin{equation}\label{eq:defect-braiding}
           \sigma^{Z(\cB)}|_{H(\idty\boxtimes -)}^{} = H\left(\mathrm{id} \boxtimes \sigma^{\overline{\cC}}\right) \quad\text{ and }\quad \sigma^{Z(\cB)}|_{H(-\boxtimes \idty)}^{} = H(\sigma^\cA\boxtimes \mathrm{id}).
          \end{equation}
          Since the braiding is a natural transformation, we conclude
          \begin{align*}
            H\left(\mathrm{id}\boxtimes \sigma^{\overline{\cC}}\right)\circ (f\otimes g)
            & = \sigma^{Z(\cB)} \circ (f \otimes g) \\
            &= (g\otimes f) \circ \sigma^{Z(\cB)} \\
            &= (g\otimes f) \circ H\left(\sigma^\cA\boxtimes \mathrm{id}\right).
          \end{align*}
          It follows that \eqref{eq:cross-after-defect} and \eqref{eq:cross-before-defect} are the same morphism up to composition with the isomorphisms $H(-\otimes -) \cong H(-)\otimes H(-)$.

        \item[Stratified Reidemeister 3 holds:]
          Our previous result means we only need to consider tangles where no strand moves between regions.
          Even so, the three strands can be in some combination of the $\cA,\cB,$ and $\cC$ regions.

          The functor $\RT_\cB$ sends any crossing to the appropriate braiding in $Z(\cB)$, so the various stratified versions of the Yang-Baxter equation all simplify to those of $Z(\cB)$, where some of the objects happen to be in the image of the functor $H$.

        \item[Stratified Reidemeister 2 holds:]
          By the same argument as above, this reduces to and hence follows from the analogous statement about the invertibility of the braiding in $Z(\cB)$.
      \end{description}

      Both fullness and essential surjectivity follow from the fact that pivotal categories have coherent planar diagrammatics \cite{Selinger2011}.
    \end{proof}

    \begin{lemma}
    Skein relations are compatible with embeddings.
    Let $\iota_A : \DD_A \hookrightarrow \DD_B$ and $\iota_C : \DD_C \hookrightarrow \DD_B$ be stratified embeddings. If two $\cA$ (resp. $\cC$) colored ribbon graphs are equivalent as $\cA$ (resp. $\cC$) skeins, then their push forwards along $\iota_A$ (resp. $\iota_C$) are equivalent as defect skeins.
    \end{lemma}
\begin{proof}
        The claim follows from commutativity of the following diagrams:
        \begin{equation}
            \begin{tikzcd}
            \Rib_\cA \ar[r,"(\iota_A)_*"] \ar[d,"\RT_\cA"] & \Rib_\cB \ar[d,"\RT_\cB"] \\
            \cA \cong \cA\boxtimes \idty \ar[r,"H|_{\cA\boxtimes \idty}"] & \cB
            \end{tikzcd}
            \hspace{2cm}\begin{tikzcd}
            \Rib_\cC \ar[r,"(\iota_C)_*"] \ar[d,"\RT_\cC"] & \Rib_\cB \ar[d,"\RT_\cB"] \\
            \cC \cong \idty\boxtimes \cC \ar[r,"H|_{\idty \boxtimes \cC}"] & \cB
            \end{tikzcd}
        \end{equation}
        which itself follows from the definition of $\RT_\cB$ on bulk skeins.
\end{proof}

  \subsection{Defect skein modules and categories}\label{sec:defect-skein-category}
  We now define the defect skein category and associated constructions analogously to the unstratified skein theory, using our defect Reshetikhin-Turaev evaluation functors as a local model near the domain walls.  Fix a compact oriented bipartite 3-manifold, possibly with boundary, and a $\cB$-labeling $X$ on $\partial M$.

    \begin{definition}\label{def:rel-defect-skein-module}
    
    The \emph{\bf relative defect skein module} $\Sk_\cB(M,X)$ is the $\K$-module spanned by isotopy classes of stratified ribbon graphs in $M$ compatible with $X$.
    These are taken modulo the relations induced by $\RT_\cA, \RT_\cC,$ and $\RT_\cB$ from the bipartite embedding of any cylinder $\DD_\mathcal{R} \times [0,1] \hookrightarrow M$. Here $\mathcal{R} = \cA,\cB,\cC$.
  \end{definition}
In the special case where $X=\emptyset$, we abbreviate by $\Sk_\cB(M)$ and refer to this as the \emph{\bf defect skein module}.

  \begin{definition}\label{def:defect-skein-category}
    The \emph{\bf defect skein category} $\SkCat_{\cB}(\Sigma)$ has
    \begin{description}
      \item \emph{\underline{objects}:} $\cB$-labellings: Finite collections of oriented framed marked points in $\Sigma$, colored by objects of the category associated to their region.
      \item \emph{\underline{morphisms:}} The homomorphism space from $X$ to $Y$ is the relative defect skein module
        \begin{equation}
          \Hom(X,Y) := \Sk_{\cB}(\Sigma \times [0,1], \overline{X}\sqcup Y),
        \end{equation}
        where $\overline{X}$ is on $\Sigma \times \{0\}$ and $Y$ is on $\Sigma \times \{1\}$.
        Composition is given by stacking copies of the thickened surface and a smoothing at boundary points.
    \end{description}
  \end{definition}

\begin{definition}
The transverse defect skein category $\SkCat^\pitchfork_\cB(\Sigma)$ is the subcategory of $\SkCat_\cB(\Sigma)$ with 
\begin{description}
    \item \emph{\underline{objects}:} $\cB$ labellings of $\Sigma$ with marked points only in the $\cA$ and $\cC$ regions.

    \item \emph{\underline{morphisms}:} The homomorphism space from $X$ to $Y$ is the subspace of $\Sk_\cB(\Sigma\times[0,1], \overline{X}\sqcup Y)$ spanned by equivalence classes of $\cB$-colored stratified ribbon graphs which intersect the defect transversely or not at all.
\end{description}
\end{definition}
\begin{remark} We note that the inclusion $\SkCat^\pitchfork_\cB(\Sigma) \to \SkCat_\cB(\Sigma)$ is neither full nor essentially surjective  in general.
\end{remark}

An important feature of skein theory is that we have an equivalence $\SkCat_\cB(\DD_A) \cong \cA$ as ribbon categories. (Similarly for $\cC$.)
This follows from the fact that the evaluation functor $\RT_\cA : \Rib_\cA \to \cA$ is full, essentially surjective and gives a well defined graphical calculus.
Now we show a similar result for $\SkCat_\cB(\DD_B)$.

\begin{lemma}\label{lemma:defect-disk-gives-B} We have an equivalence
  $\SkCat_\cB(\DD_B) \cong \cB$ as pivotal $(\cA,\cC)$-central algebras.
\end{lemma}

\begin{proof}
  The pivotal functor $\RT_\cB : \Rib_\cB \to \cB$ of Definition \ref{def:defect-RT} is full and essentially surjective.
  Since skein relations are exactly the kernel of this map, it follows that $\RT_\cB$ induces an equivalence of pivotal categories $\SkCat_\cB(\DD_B) \cong \cB$.

  By Remark \ref{rmk:embeddings-make-functors}, the $(\Rib_A,\Rib_C)$-central structure on $\Rib_\cB$ descends to the skein category.
  Let $(H^\text{Sk}, J^{Sk}, c^{Sk})$ be the $(\SkCat(\DD_A),\SkCat(\DD_C)$-central structure on $\SkCat_\cB(\DD_B)$.
  Equivalence of the central structures can be expressed by this commutative diagram:
  \begin{equation*}
    \begin{tikzcd}
      \SkCat_\cB(\DD_A) \boxtimes \SkCat_\cB(\overline{\DD}_C) \ar[d,"\RT_\cA \boxtimes \RT_\mathcal{\overline{C}}"]\ar[r,"{H^\mathrm{Sk},c^{Sk}}"]
      &[5mm] Z(\SkCat_\cB(\DD_B)) \ar[d,"Z(\RT_\cB)"] \\
      \cA \boxtimes \overline{\cC} \ar[r,"{H,c}"]
      & Z(\cB)
    \end{tikzcd}
  \end{equation*}
  To obtain the right-most vertical arrow in the above diagram we've used that $\RT_\cB$ is a categorical equivalence.
  Tangles entirely in the bulk regions of $\DD_B$ are exactly those in the image of $H^\mathrm{Sk}$.
  In the construction of $\RT_\cB$ such tangles are projected onto the defect where $H$ gives the appropriate objects and morphisms of $\cB$.
  This implies that $\RT_\cB \circ H^\mathrm{Sk} = H \circ (\RT_\cA \boxtimes \RT_\cC)$.

  When projection onto the defect introduces crossings, they are replaced by coupons labelled by the half braiding $c$.
  The half braiding $c^\mathrm{Sk}$ is induced by the isotopy in Figure \ref{fig:induced-half-braiding}, which introduces crossings between strands.

  Commutativity of the diagram follows from the fact that $\RT_\cB$ sends a coupon to the labelling morphism.
\end{proof}

\subsection{Categories and coends}\label{sec:various-cats}
Here we recall a few categorical concepts and definitions we will require, see \cite{Adamek_Rosicky_1994} for complete details.

\begin{definition}\label{def:cat-cat}
  The 2-category $\Cat$ of \emph{\bf $\K$-linear categories and functors} has:
  \begin{description}
    \item \emph{\underline{objects:}} Small $\K$-linear categories.
    \item \emph{\underline{1-morphisms:}} $\K$-linear functors.
    \item \emph{\underline{2-morphisms:}} natural transformations
\end{description}
\end{definition}

Given a $\K$-linear bifunctor $F:\cC\times \mathcal{C}^{op}\to \mathcal{E}$, where $\cC$ is a linear categories, and $\mathcal{E}$ contains all colimits, the coend of $F$ is the object of $\mathcal{E}$ defined by,
\begin{equation}\label{eq:coend-colimit}
\operatornamewithlimits{colim}_{\overset{c,d \in \cC}{f:c\to d}}
\int^{c\in\cC} F = \left(
    \begin{tikzcd}[column sep=large,row sep=scriptsize]
        &F(c,c)\arrow[d, "F(f{,}\id_{c})"] \\
        F(d,d) \arrow[r,"F(\id_{d}\!{,}f)"] &F(d,c)
    \end{tikzcd}\right)
\end{equation}

\begin{definition}\label{def:bimod-cat}
  The 2-category $\Bimod$ of \emph{\bf $\K$-linear categories and categorical bimodules} has:
  \begin{description}
    \item \emph{\underline{Objects:}} Small $\K$-linear categories.
    \item \emph{\underline{1-morphisms:}} From $\cC$ to $\mathcal{D}$, these are $\K$-linear functors $F : \cC \times \mathcal{D}^{op} \to \Vect$.
      \footnote{Such functors are sometimes called \emph{bimodules} or \emph{profunctors}.}
    
      \item \emph{\underline{Composition:}} The composition of two bimodules $F : \cC \times \mathcal{D}^{op} \to \Vect$ and $G: \mathcal{D} \times \mathcal{E}^{op} \to \Vect$ is given by the coend \cite[IX.6]{MacLane1978}:
      \begin{equation}\label{eq:profunctor-composition}
        (F \circ G) (c,e) := \int^{d\in\mathcal{D}} F(c,d)\otimes G(d,e)
      \end{equation}
    \item \emph{\underline{2-morphisms:}} Natural transformations.
  \end{description}
\end{definition}

\begin{definition}\label{def:pr}
  The 2-category of \emph{\bf locally presentable categories $\Pr^L$} has:
  \begin{description}
    \item \emph{\underline{Objects:}} Locally presentable $\K$-linear categories.
    \item \emph{\underline{1-morphisms:}} cocontinuous functors.
    \item \emph{\underline{2-morphisms:}} natural transformations.
  \end{description}
\end{definition}
We have natural embeddings of 2-categories,
$\Cat \subset \Bimod \subset {\Pr}^L,$ constructed as follows.  Firstly, we have a fully faithful 2-functor $\Cat \to \Bimod$, which is the identity on objects, and which sends a functor $F : \cC \to \mathcal{D}$ to the bimodule
\begin{equation}\label{eq:functor-to-bimodule}
  \begin{aligned}
    \cC \times \mathcal{D}^{op} &\to \Vect \\
    (c,d) &\mapsto \Hom(d,F(c)),
  \end{aligned}
\end{equation}
and which sends a natural transformation to a bimodule homomorphism in the obvious way.

Any $\cC\in\Bimod$ determines a locally presentable category $\widehat{\cC}\in\Pr^L$, as follows. Let $\Vect\in\Pr^L$ denote the category of vector spaces (whose basis may be of arbitrary cardinality).  Given $\cC\in\Cat$, we will denote by $\widehat{\cC} := \Fun(\cC^{op},\Vect)\in\Pr^L$ the pre-sheaf category of $\cC$. 

We recall the Yoneda embedding $\cC\to\widehat{C}$, sending $c \mapsto \Hom(-,c)$.  Any $\K$-linear functor from $\cC$ to a cocomplete category $\mathcal{D}$ extends canonically through the Yoneda embedding to a cocontinuous functor $\widehat{\cC}\to\mathcal{D}$; for this reason $\widehat{C}$ is also called the \emph{free cocompletion} of $\cC$.

The free cocompletion extends to a 2-functor $\widehat{\cdot}: \Bimod \to \Pr^L$, as follows. The free cocompletion $\widehat{N}$ of a bimodule $N: \cC \times \mathcal{D}^{op} \to \Vect$ is the cocontinuous functor $\widehat{N} : \widehat{\cC} \to \widehat{\mathcal{D}}$,
\begin{equation}\label{eq:cocomp-on-a-bimodule}
  \widehat{N}(S) : d \mapsto \int^{c \in \cC} N(c,d) \otimes S(c).
\end{equation}
In the special case that $N$ is obtained from a functor $F$, the free cocompletion $\widehat{F}$ of the functor $F : \cC \to \mathcal{D}$ is given by the coend 
\begin{equation}\label{eq:yoneda-on-functors}
  \hat{F}(S)(d) = \int^{c} \Hom(d, F(c)) \otimes S(c). 
\end{equation}

A functor between categories in $\Pr^L$ with enough compact projectives lies in the image of the cocompletion functor if, and only if, it preserves compact projectives objects, or equivalently if, and only if, its right adjoint is itself cocontinuous.

\subsection{Cobordisms and gluing}\label{sec:gluing-skein-cats} 
Defect skein modules define a categorified oriented 2+1 TQFT for cobordisms with smoothly embedded codimension-one defects, which we now survey.
For our main application we will compute with stratified analogues of one and two handle attachments and three dimensions.

  \begin{definition}\label{def:skein-bimodule-functor}
  Let $M$ be a stratified 3-manifold as in Definition \ref{def:rel-defect-skein-module}, with the additional data of a decomposition of its boundary $\partial M \cong (\overline{\Sigma}_{in} \sqcup \Sigma_{out}) \cup R.$
  We will write 
  \begin{equation}
\underline{\Sk}(M) : \SkCat(\Sigma_{in})^{op} \times \SkCat(\Sigma_{out}) \to \Vect
  \end{equation}
    for the bimodule defined by the relative defect skein module $X_{in},X_{out} \mapsto \Sk_\cB(M,\overline{X}_{in}\sqcup X_{out})$, and refer to this as the \emph{\bf defect skein functor}.
  \end{definition}

  The following result is a mild generalization of \cite[Theorem 2.5]{GJS2021} and \cite[Theorem 4.4.2]{Walker_2006} to bipartite 3-manifolds.
  The main claim is that collar shift isotopies and stratified isotopies supported in balls disjoint from $\Sigma_{gl} \sqcup \overline{\Sigma}_{gl}$ generate stratified isotopies on $M_{gl}$.  
    \begin{lemma}\label{lemma:skfun-glues}
        Let $M$ be an oriented stratified 3-manifold with a stratification $M \cong M_A \cup_{S_B} M_C$.
        Fix a decomposition $\partial M \cong \Sigma_{gl} \cup \overline{\Sigma}_{gl} \cup R$ and let $M_{gl}$ denote $M$ with the two copies of $\Sigma_{gl}$ identified.
        There is an equivalence of functors $\SkCat(R)^{op} \to \Vect$
        \begin{equation}\label{eq:skfun-gluing}
            \underline{\Sk}(M_{gl})(-) \cong \int^{X \in \SkCat(\Sigma_{gl})} \hspace{-1em} \underline{\Sk}(M)(\overline{X} \sqcup X\sqcup -).
        \end{equation}
    \end{lemma}

The coend relations in the formula above express how a coupon on one side of a gluing surface can pass through to become a coupon on the other side of the surface.
The appearance of a coend causes similar behaviour for coupons near gates, see Figure \ref{fig:internal-skein-details}.
Comparing the coend formulas for gluing \eqref{eq:skfun-gluing} and for composition of profunctors \eqref{eq:profunctor-composition}, we get the following corollaries:

\begin{corollary}
    Let $\Sigma_1 \overset{M_{12}}{\to} \Sigma_2 \overset{M_{23}}{\to} \Sigma_3$ be a pair of composable stratified cobordisms.  Then we have:
    \begin{equation}
      \Sk(M_{12}\cup_{\Sigma_{2}} M_{23}) \cong \Sk(M_{12})\circ \Sk(M_{23}).
  \end{equation}

\end{corollary}
\begin{corollary}\label{cor:SkCat-TFT}
  Together $\SkCat_\cB(-)$ and $\underline{\Sk}(-)$ define a contravariant functor from the category of oriented 2+1 collared cobordisms with smoothly embedded codimension-one defects to $\Bimod$.
  We will denote this categorified 2+1 TQFT by
  \begin{equation}
      \underline{\Sk}_{\cB} : \StratCob_{2+1}^{or} \to \Bimod
  \end{equation}
\end{corollary}
We rely heavily on Lemma \ref{lemma:skfun-glues} in our computation of the quantum A-polynomial, where we use the gluing properties of defect skein categories to decompose 3-manifolds into triangulations.

\subsection{Defect internal skein modules and algebras}\label{sec:internal-constructions}

We assume that our surfaces have non-empty boundary.
  
  \begin{definition}\label{def:disk-insertion}
    Let $\Sigma\cong \Sigma_A \cup_{\Gamma_B} \Sigma_C$ denote a stratified surface and $\G$ be a collection of disjoint intervals (called \emph{\bf gates}) along $\partial \Sigma$, also disjoint from the interface $\Gamma_B$.
    The inclusion $\mathcal{G} \times [0,1] \hookrightarrow \Sigma$ induces the \emph{\textbf{disk insertion functor}}:
    \begin{equation}
        \mathcal{P} : \SkCat(\mathcal{G} \times [0,1]) \to \SkCat(\Sigma)
    \end{equation}
    Suppose there are $n$ gates in the $A$ region, and $m$ in the $C$ region. Then $\SkCat(\mathcal{G} \times [0,1]) \cong \cA^{\boxtimes n} \boxtimes \cC^{\boxtimes m}$; we frequently consider $\mathcal{P}$ as a functor from the latter . Abusing notation slightly, we will use the shorthand $\mathcal{G} : = \cA^{\boxtimes n}\boxtimes \cC^{\boxtimes m}$.
  \end{definition}

    Note that $\SkCat(\mathcal{G}\times[0,1])$ carries a natural monoidal structure via concatenation in the $[0,1]$ direction; likewise the choice of gates equips $\SkCat(\Sigma)$ with the structure of a $\mathcal{G}$-module category.  Upon free cocompletion, the functor $\widehat{\mathcal{P}}$ induced by disk insertion has a right adjoint:
  \begin{equation}
    \widehat{\mathcal{P}}^R : \widehat{\SkCat}(\Sigma) \to \widehat{\SkCat}(\mathcal{G} \times [0,1]).
  \end{equation}
  The monad $\widehat{\mathcal{P}}^R \widehat{\mathcal{P}}$ of this adjuction is a $\mathcal{G}$-module endofunctor of $\widehat{\SkCat}(\mathcal{G}\times[0,1]) \cong \text{Fun}\big(\cA^{\boxtimes n}\boxtimes \cC^{\boxtimes m},\Vect\big)$. Evaluating the monad induced by disk insertion at the unit $\widehat{\idty} = \Hom(\idty,-)$ gives an algebra object in the free cocompletion of $\mathcal{G}=\cA^{\boxtimes n} \boxtimes \cC^{\boxtimes m}$.
  \begin{definition}\label{def:internal-skein-algebra}
   
   The \emph{\textbf{defect internal skein algebra}} is the algebra object:
\begin{equation}
 \SkAlg^{int}_\G (\Sigma) := \widehat{\mathcal{P}}^R\widehat{\mathcal{P}}(\widehat{\idty})\in\widehat{\mathcal{G}}.
\end{equation}
  \end{definition}

We note that in \cite{GJS2021}, the internal skein algebra was described by the equivalent data of a lax monoidal functor
  \begin{equation}\label{eq:int-skalg-as-functor}
  \begin{aligned}
            \SkAlg^\mathrm{int}_\G(\Sigma) : \G^{op} &\to \Vect\\
      V &\mapsto \Sk\big(\Sigma\times [0,1];\mathcal{P}(V),\varnothing\big).
  \end{aligned}
    \end{equation}
    The equivalence of these data follows from a standard identification between algebra objects in $\widehat{\G}$ and lax monoidal functors $\G^{op} \to \Vect$, and the short computation:
  \begin{equation}\label{eq:int-skein-evaluation}
      \begin{aligned}
          \SkAlg^{int}_\G(\Sigma)(V) &\cong \Hom_{\widehat{\G}}\left(\widehat{V}, \widehat{\mathcal{P}}^R\widehat{\mathcal{P}}(\widehat{\idty})\right) \cong \Hom_{\widehat{\SkCat}(\Sigma)}\left(\widehat{\mathcal{P}(V)},\widehat{\mathcal{P}(\idty)}\right) \\
          &\cong \Hom_{\SkCat(\Sigma)}\left(\mathcal{P}(V),\mathcal{P}(\idty)\right) = \Sk(\Sigma\times[0,1];\mathcal{P}(V)).
      \end{aligned}
  \end{equation}
  Here the first and third equivalence use the Yoneda lemma ($\Hom(\hat{c},F)\simeq F(c)$) while the second follows from the definition of an adjoint together with $\widehat{\mathcal{P}}(\widehat{W}) \cong \widehat{\mathcal{P}(W)}$.

    \begin{definition}\label{def:internal-skein-bimodule}
    Let $M = M_A \cup_{\Sigma_B} M_C$ be a stratified 3-manifold. Fix an identification $\partial M = \Sigma \cup_{\Gamma} R$. For the purpose of disk insertion we treat $\Gamma$ as the boundary of $\partial M$.
    Fix gates $\mathcal{G}\times [0,1] \hookrightarrow \Sigma$ with associated disk insertion functor $\P$.
      Recall the shorthand $\G := \SkCat(\mathcal{G}\times[0,1])$.
      The \emph{\textbf{defect internal skein module }}         $\Sk^{\text{int}}(M) \in \widehat{ \G}$
    given by
    \begin{equation}
      V \mapsto \Sk\big(M;\mathcal{P}(V)\big).
    \end{equation}
    \end{definition}
    
    This is a $\SkAlg^{int}(\Sigma)$ module internal to $\widehat{\G}$, meaning we have an associative multiplication morphism $\SkAlg^{int}(\Sigma) \otimes \Sk^{int}(M) \to \Sk^{int}(M)$ in $\widehat{\G}$.
    By \eqref{eq:int-skein-evaluation}, we have an isomorphism $\SkAlg^{int}(\Sigma) \cong \Sk^{int}(\Sigma\times [0,1])$.

    Internal skein modules and algebras have good gluing properties, which will power our computations in Section \ref{sec:quantum-A-polynomial}. Defects play a minor role in this story, which appears in \cite{GJS2021}.

  \begin{lemma}\label{lemma:gluing-for-int-skein-modules}
      Let $M$ be an oriented stratified 3-manifold with a stratification $M \cong M_A \cup_{S_B} M_C$.
      Fix a decomposition $\partial M \cong \Sigma_{gl} \cup \overline{\Sigma}_{gl} \cup R$ and let $M_{gl}$ denote $M$ with the two copies of $\Sigma_{gl}$ identified.
      Let $\G_{gl}$ denote the gate category on $\Sigma_{gl}$ and $\G_R$ that on $R$.
      The internal skein modules of $M$ and $M_{gl}$ are related as follows:
      \begin{equation}
          \Sk^{int}(M_{gl})(-) \cong \int^{V \in \G_{gl}}\hspace{-2em} \Sk^{int}(M)(V\boxtimes V^* \boxtimes -)
      \end{equation}
  \end{lemma}
 
  \begin{proof}
      Let $\P_R : \G_R \to \SkCat(R)$, $\P_{gl} : \G_{gl} \to \SkCat(\Sigma_{gl})$ and $\P : \G \to \SkCat(\partial M)$ denote the disk insertion functors.
      Note that $\Sk^{int}(N) = \Sk(N) \circ \P_{\partial N}$ as functors. 
      Then
      \begin{equation}
          \begin{aligned}
              \Sk^{int}(M_{gl}) &= \Sk(M_{gl})\circ \P_{R} \cong \int^{X\in \SkCat(\Sigma_{gl})}\hspace{-5em} \Sk(M)(X \sqcup \overline{X} \sqcup \P_{R}(-)) \\
              &\cong \int^{V\in \G_{gl}}\hspace{-2em} \Sk(M)(\P_{gl}(V)\sqcup \overline{\P_{gl}}(V) \sqcup \P_R(-)) \\
              &= \int^{V\in \G_{gl}}\hspace{-2em} \Sk(M)(\P(V\boxtimes V^* \boxtimes -) = \int^{V\in \G_{gl}}\hspace{-2em} \Sk(M)^{int}(V\boxtimes V^* \boxtimes -).
          \end{aligned}
      \end{equation}
      The coend on the first line is a result of Lemma \ref{lemma:skfun-glues}, while the equivalence in the second line follows from the fact that every $X$ in $\SkCat(\Sigma_{gl})$ is isomorphic to some $\P_{gl}(V)$.
  \end{proof}
  
\subsection{Gates}
    Gates are the crucial feature of internal skein algebras and modules.
In the closely related stated skein model \cite{Costantino_Le_2022} skeins are allowed to end on certain fixed boundary components, where they are labelled (by \emph{states}) and subject to extra skein relations.
This same behaviour appears for internal skein algebras and modules.
To see how, we will need the following standard lemma:
\begin{lemma}[co-Yoneda]\label{lemma:co-yoneda}
    Fix a functor $F : \cC^{op} \to \Vect$. Then
    \begin{equation}
        F \simeq \int^{c\in\cC} \Hom_\cC(-,c)\otimes F(c).
    \end{equation}
\end{lemma}
Applying Lemma \ref{lemma:co-yoneda} to the internal skein algebra, we get
\begin{equation}\label{eq:internal-coend-expression}
    \Sk^{int}(M) \simeq \int^{V \in \G} \hat{V} \otimes \Sk^{int}(M)(V) = \int^{V \in \G} \hat{V} \otimes \Sk(M; \P(V)).
\end{equation}
Where the second equivalence is by definition.
Unpacking the definition of the coend, this formula states that the internal skein module is spanned by skeins in $M$ which may end in $\partial M$ at a gate with some label $\P(V)$, stated by an $v \in \hat{V}$. 

These are taken up to the coend equivalence relation \eqref{eq:coend-colimit}, 
which implies that a skein with a coupon $f : V \to W$ near a gate stated by some element $v\in V$ is equivalent to the otherwise-unchanged skein without that coupon and with state $f(v)\in V$, see Figure \ref{fig:internal-skein-details}.

\begin{figure}
    \centering
        \raisebox{3mm}{\includegraphics{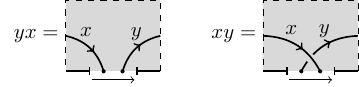}}\hspace{20mm}\includegraphics[]{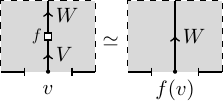}
    \caption{\textbf{Left:} Products of skein are read right to left -- the leftmost skein is the lowest and is first according to the orientation at the gate. \textbf{Right:} Internal skeins are stated where they meet a gate, and coupons can be absorbed into gates due to the coend expression \eqref{eq:internal-coend-expression} for internal skeins.}
    \label{fig:internal-skein-details}
\end{figure}

Because labellings of the form $\P(V)$ have marked points only immediately adjacent to gates, we will draw skeins as starting and ending on the gates themselves.
When multiple edges of a skein meet at a single gate in a thickened surface, their relative heights are determined according to the orientation of the boundary component.
See Figure \ref{fig:internal-skein-details}.
    
    \paragraph{Changing gates and closing punctures}
  For computations it is sometimes practical to change the number of gates used in the construction of internal skein algebras and modules.
  By \emph{closing a gate} we will mean taking invariants of the associated disk insertion action.
  On the level of the internal skein algebra, this is restriction to an invariant sub-algebra generated by those skeins which do not have an endpoint at the specified gate.

  By \emph{opening a gate} we will mean passing to a larger algebra in which skeins are allowed to end at an additional gate.
  In practice we will often set up our marked surfaces to have sufficient gates for gluing operations, then we will close all gates near the end of a computation to arrive at the (non-internal) skein algebra.

Going further, it is sometimes necessary to puncture a surface so that it has sufficient boundary components for monadic reconstruction.
In the language of \cite[Sec. 1.3]{JLSS2021}, this is the need for a $\G$-chart.
  Let $\Sigma^* := \Sigma \setminus \DD$ be a stratified punctured surface.
  Let $\G_p$ denote the gates on the boundary component associated to the puncture and $\G_r$ the gates on the rest of the boundary.
  \emph{Quantum Hamiltonian reduction} allows us to pass from $\SkAlg^{int}_{\G_p \cup \G_r}(\Sigma^*)$ to $\SkAlg^{int}_{\G_r}(\Sigma)$, or if $\G_r = \varnothing$, to $\SkAlg(\Sigma)$.
  This two step process is used extensively in \cite{GJS2021}, where the unstratified surfaces need at most one puncture.
    In the first step we restrict to the invariant sub-algebra generated by skeins which do not meet the gates at the puncture:
    \begin{equation}
        \SkAlg^{int}_{\G_p\cup \G_r}(\Sigma^*)^{\G_p}\subset \SkAlg^{int}_{\G_p\cup \G_r}(\Sigma^*)
    \end{equation}
    In the second step we quotient by the relation that fixes the monodromy around the puncture.
    To be precise, this means specializing the skein parallel and immediately adjacent to the puncture's boundary component equal to the appropriate quantum dimension.

\paragraph{Gates and framing}
Gates are the algebraic analogue of trivialising a local system at a marked point. This trivialisation is often called a \emph{framing}.
As we'll now see, this analogy becomes precise when we use the classical $\Rep G$ as our ribbon category.

Recall that ``opening gates'' on a surface $\Sigma$ means specifying an embedding $\DD^{\sqcup n} \sqcup \Sigma \hookrightarrow \Sigma$ and using the induced module category structure
\begin{equation}\label{eq:disk-embedding}
      \SkCat(\DD)^n \times \SkCat(\Sigma) \to \SkCat(\Sigma)
\end{equation}
 to define an algebra object $\SkAlg^{int}(\Sigma)$ in the monoidal category $\SkCat(\DD)^n$.
 Note that such an embedding only exists if $\partial\Sigma \neq \varnothing$.
 
 Assume for exposition's sake that $n=1$ and that our gate is in a $G$-region.
In the classical setting, an embedding $\DD \hookrightarrow \Sigma$ induces a restriction map $\rho: \Ch^{dec}(\Sigma) \to \Ch_G(\DD)$ whose pullback functor \begin{equation}
    \rho^* : \QCoh(\Ch_G(\DD)) \to \QCoh (\Ch^{dec}(\Sigma))
\end{equation}
is the geometric incarnation of \eqref{eq:disk-embedding}.
Note that the $G$-character stack of a disk is $\bullet/G$, so that $\SkCat(\DD_G) \simeq \QCoh\left(\bullet/G \right) \simeq \Rep G$.
(Similarly $\SkCat(\DD_T) \simeq \QCoh\left(\bullet/T \right) \simeq \Rep T$.)

Since the restriction map is quasi-compact and quasi-separated (both follow from its codomain being $\bullet/G$), the right adjoint of $\rho^*$ is the pushforward functor $\rho_* : \QCoh(\Ch^{dec}(\Sigma)) \to \Rep G$. (See, e.g. \cite[Prop. 14.5.7]{Vakil}.)

We therefore construct the internal skein module by evaluating the monad $\rho_*\rho^* : \Rep G \to \Rep G$ on the monoidal unit, i.e. the structure sheaf on $\mathcal{O}_{\bullet/G}$:
\begin{equation}
    \SkAlg^{int}_{\Rep G} (\Sigma) := \rho_*\rho^*(\O_{\bullet/G}) = \rho_*\O_{\Ch^{dec}(\Sigma)}
\end{equation}
For $\mathscr{F}$ a quasi-coherent sheaf on $\Ch^{dec}(\Sigma)$, the pushforward $\rho_*\mathscr{F} \in \Rep G$ is the vector space of global sections $\mathscr{F}(\Ch^{dec}(\Sigma))$ endowed with the structure of a $G$-representation.
We conclude that the internal skein algebra is the ring of regular functions $\mathcal{O}(\Ch^{dec}(\Sigma))$ endowed with an extra $G$-action, i.e. a framing.

  \section{The parabolic defect}\label{sec:parabolic-defect}

This section is dedicated to constructing our main example of a local coefficient system, built from the representation theory of the quantum group, its Borel subalgebra, and the universal Cartan sub-quotient.  While our detailed computations in this paper are for $\SL_2$, the basic definitions make sense for any reductive group, and so we recall them in that generality.  We refrain from giving detailed presentations of various quantum groups we consider, for which we refer the reader to standard textbooks such as \cite{Kassel_1995}, \cite{Chari_1994}, \cite{Jantzen_1996}, or \cite{Klimyk_Schmudgen_1997}.

Let $G$ be a reductive group, with fixed Borel subgroup $i : B \hookrightarrow G$, and its universal Cartan quotient $\pi : B \to T := B /(B,B)$.\footnote{%
  It is important that $T$ is regarded as a quotient of $B$ rather than a subgroup of $G$.
}  We let $\Lambda=\Hom(T,\mathbb{C}^\times)$ denote the weight lattice, and let $E_\lambda,F_\lambda,K_\lambda$ denote the Serre generators of $U_q\mathfrak{g}$.  Given fundamental weights $\lambda_1,\ldots,\lambda_r$, we'll use the shorthand $X_{\lambda_i} = X_i$.
We identify $U_q \mathfrak{b} \subset U_q \mathfrak{g}$ as the subalgebra generated by the $E_\lambda$ and $K_\lambda$.
The projection $\pi : U_q\mathfrak{b} \to U_q\mathfrak{t}$ is given by $E_\lambda \mapsto 0$, $K_\lambda \mapsto K_\lambda$.

We denote by $\Rep_qG$ be the ribbon category of finite dimensional representations of $U_q\mathfrak{g}$. 
For generic values of $q$, this is a semisimple category with simple objects $V(\lambda)$ indexed by dominant integral weights $\lambda\in\Lambda^+$.

We denote by $\Rep_qT$ be the category of locally finite dimensional $U_q\mathfrak{t}$-modules spanned by weight vectors $v_\mu$ on which the $K_i$ act by the $i$th fundamental weight $K_i \cdot v_j = q^{\langle \lambda_i,\mu\rangle}v_j.$  For generic $q$ this is a semisimple category with simple objects indexed by characters $\mu\in\Lambda$.

Finally, we denote by $\Rep_qB$ the category of \emph{locally finite} $U_q\mathfrak{b}$-modules spanned by weight vectors (for all generators $K_i$) as above.  For generic $q$, the category is not semi-simple.  Simple objects of $\Rep_qB$ are all obtained by pulling back one-dimensional representations via the homomorphism $\pi: U_q(\mathfrak{b})\to U_q(\mathfrak{t})$ which sends each $E_\lambda\mapsto 0$.  In particular, the simple objects are indexed by $\mu\in\Lambda$.  A typical example of an indecomposable object is $i^*(V)$, for $V$ a simple object of $\Rep_qG$.  Here $i$ denotes the inclusion $i: U_q(\mathfrak{b})\to U_q(\mathfrak{g})$.

The category $\Rep_qB$ carries a canonical structure of a $(\Rep_qG, \Rep_gT)$-central tensor category, via the functor:
\begin{align*}
  \Rep_q G \otimes \Rep_q T^{op} & \to Z(\Rep_q B)\\
  V\boxtimes \chi & \mapsto \big(i^*(V) \otimes \pi^*(\chi), c_{(V\boxtimes\chi,\cdot)}\big).
\end{align*}
The half braiding $c_{(V\boxtimes\chi,\cdot)}$ for the central structure is defined as follows.
Let $\mathcal{R}^G, \mathcal{R}^T$ denote the $\mathcal{R}$-matrices defining the braiding for $\Rep_qG$ and $\Rep_qT$, then the half-braiding is given by:
\begin{equation*}
  \begin{tikzcd}
    c_{(V\boxtimes \chi,W)}^{}: \big(i^*(V) \otimes \pi^*(\chi)\big) \otimes W \ar[r,|->,"\sigma_{(123)} \mathcal{R}^G_{13} (\mathcal{R}^T_{32})^{-1}"] &[2cm] W\otimes \big(i^*(V) \otimes \pi^*(\chi)\big).
  \end{tikzcd}
\end{equation*}

The following lemma follows from a direct computation and is responsible for the relations shown in Figure \ref{fig:kauffman-dips}.

\begin{lemma}\label{lemma:defect-skein-unique}
    Fix a dominant weight $\lambda\in\Lambda^+$, and a character $\mu\in\Lambda$, let $V(\lambda)$ and $\mu$ denote the corresponding representations in $\Rep_qG$ and $\Rep_qT$, respectively.
    Let $X_{\lambda}$ (resp. $X_{\mu}$) be the $\Rep_qB$ labelling of $\DD_B$ with a single positively oriented point in the $\Rep_qG$ (resp. $\Rep_qT$) region labelled by $V(\lambda)$ (resp. $\mu$).  Then we have:
    \[\Hom_{\SkCat_{\Rep_q(B)}(\DD_B)}(X_\lambda,X_\mu) = \Hom_{\SkCat^\pitchfork_{\Rep_q(B)}(\DD_B)}(X_\lambda,X_\mu) = \left\{ \begin{array}{ll} \K,& \textrm{if $\lambda=\mu$,}\\ 0,& \textrm {else.}\end{array}
    \right.
    \]
\end{lemma}

The significance of the lemma is that locally, if we restrict the labels of our skeins to simple objects of $\Rep_qT$ and $\Rep_q G$ in the $T$- and $G$- regions of the surface -- which we may do with out loss of generality, by standard arguments -- then the coupons labelling where the skein passes through the defect are unique up to a scalar. 
In the case of $G=\SL_2$, this leads to the following Kauffman-style relations.

\begin{lemma}[Kauffmann--M\"uller type presentation of parabolic defect skeins for $\SL_2$]\label{lemma:kauffman-presentation}
The parabolic defect skein module of a bipartite 3-manifold is the $\K$-vector space spanned by bipartite ribbon tangles in $M$.
Edges are unoriented in the $G$ region and oriented in the $T$ region.

At $G$-gates the skeins have a state $+$ or $-$, at $T$-gates they're unstated.

These are and subject to the local skein relations shown in Figures \ref{fig:kauffman-crossings} -- \ref{fig:kauffman-gates}.
\begin{figure}
    \centering
    \includegraphics[]{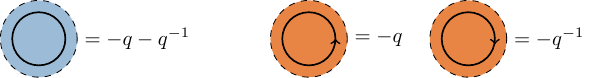}\hspace{8mm}\includegraphics[]{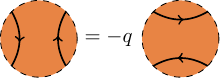}\\[1em]
    \includegraphics[]{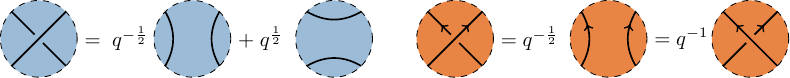}
      \caption{Relations for the parabolic defect skein algebra which don't involve gates or the defect. The blue disk \textbf{(left)} shows the $\SL_2\C$ skein relations, using the canonical skew-symmetric self-duality of the fundamental $\SL_2\CC$ representation. The orange disk \textbf{(right)} shows the $\CC^*$ skein relations.}\label{fig:kauffman-crossings}\medskip

    \includegraphics[]{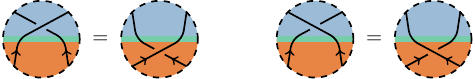}
      \caption{Crossings can be passed through the defect via (stratified) isotopy. This holds for general defect skein theories, not just the parabolic defect.}\label{fig:kauffman-defect-crossing}\medskip

    \includegraphics[]{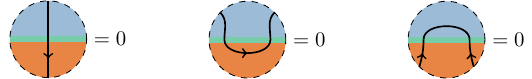}
    \caption{\textbf{Left}: A $T$-skein pointing out of a defect is trivial. \textbf{Middle, right:} Trivial ``dips'' -- these follow from Lemma \ref{lemma:defect-skein-unique}. Note that this does not imply triviality of general skeins crossing the defect, see Figure \ref{fig:kauffman-nontrivial}.}
    \label{fig:kauffman-dips}\medskip

     \includegraphics[]{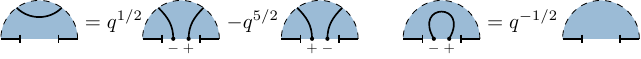}\hspace{8mm}\includegraphics[]{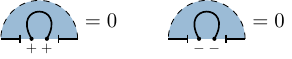}
    \caption{Relations for $G$-region skeins near a gate. Here we follow the ``stated skeins" convention, denoting by $+$ and $-$ the normalised highest and lowest weight vectors, respectively, of the fundamental representation $V$ of $\SL_2$. }\label{fig:kauffman-G-gates}
    \smallskip
    
    \includegraphics[]{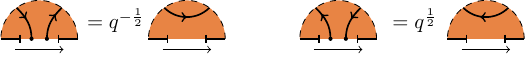}
    \caption{Relations for $T$-region skeins near a gate. The arrow below the gate indicates the orientation of the boundary component. See Figure \ref{fig:T-region-commutation-relations} for additional relations near a gate.}
    \label{fig:kauffman-gates}
\end{figure}
\end{lemma}

\begin{proof}
    Away from any defects and gates, this is a standard result about Kauffman bracket skein algebras and $\CC^*$ skein algebras.
    We follow standard conventions, with the state $+$ corresponding to the highest weight vector in the fundamental representation and $-$ to the lowest weight vector. 
    Near the defect, we use Lemma \ref{lemma:defect-skein-unique}, and near a gate we use the coend relation of an internal skein algebra.
\end{proof}

\begin{figure}
    \centering
    \includegraphics[]{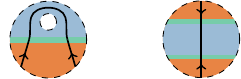}
    \caption{The vanishing dips in Figure \ref{fig:kauffman-dips} \emph{do not} imply that all skeins passing between regions vanish. For instance, the two skeins shown do not vanish in general. The skein on the right plays a major role in Section \ref{sec:quantum-A-polynomial}, where a collection of such skeins is localised (i.e. made invertible) to obtain a quantum cluster chart.}
    \label{fig:kauffman-nontrivial}
\end{figure}

\subsection{Monadic reconstruction}\label{sec:monadic-reconstruction}

Let us denote by $\act_m$ the $\Rep_qG \boxtimes \Rep_qT^{bop}$-action functor applied to some object $m \in \Rep_qB$:
\begin{equation}
  \begin{aligned}
    \act_m : \Rep_qG \boxtimes \Rep_qT^{bop} &\to \Rep_qB \\
    V \boxtimes \chi & \mapsto \iota^*V \otimes \pi^* \chi \otimes m.
\end{aligned}
\end{equation}
This is a colimit preserving functor between presentable categories and therefore has a right adjoint, $\act_m^R$.
We have an associated monad 
\[(A_m:= \act^R_m\act_m, \mu : A_mA_m \Rightarrow A_m, \epsilon : \id \Rightarrow A_m).\]

The right adjoint $\act^R_m$ is often denoted $\underline{\Hom}(m,-)$ and called the internal hom functor.
Similarly, the algebra object $A_m(\idty)=\underline{\Hom}(m,m)$ is commonly denoted $\underline{\End}(m)$ and called the internal endomorphism algebra of $m$.  We will use these notations henceforth.

Next we compare the Eilenberg-Moore category of this monad with $\Rep_qB$ itself.  We have an adjoint pair of comparison functors:
\begin{equation}
  \begin{tikzcd}
    \Rep_qB \ar[r, bend left,"\tilde{L}",end anchor={[xshift=-3em]},<-] \ar[r, bend right,"\tilde{R}",end anchor={[xshift=-3em]}] & \underline{\End}(m) -\mod_{\Rep_qG\boxtimes \Rep_qT^{bop}} 
  \end{tikzcd}
\end{equation}

In good situations, the comparison functor $\tilde{R}$ defines a reflexive embedding, or even an equivalence.  The following theorem is explained in \cite{JLSS2021,BBJ2018Integrating}, however it is a straightforward modification of a theorem of Ostrik \cite{Ostrik}, which is itself a special case of the Beck monadicity theorem \cite{Beck}.
\begin{theorem}
  Fix an action of $\cA$ on $ \mathcal{M} \to \mathcal{M}$, where $\cA$ is a cp-rigid tensor category and $\mathcal{M}$ is abelian.
  Let $m$ denote an object of $\mathcal{M}$, $\act_m : \cA \to \mathcal{M}$ the action on $m$, and $\underline{\End}(m)$ the monad of the adjunction $(\act_m, \act_m^R)$. 
  \begin{enumerate}
    \item Suppose that $\underline{\Hom}(m,-)$ is conservative (i.e. reflects isomorphisms.) Then we have a reflexive embedding of $\cA$-module categories,
      \begin{equation}
        \widetilde{R}:\mathcal{M} \hookrightarrow \underline{\End}(m)-\mod_{\cA}
      \end{equation}

    \item Suppose that $\underline{\Hom}(m,-)$ is both conservative and colimit-preserving. Then the embedding $\widetilde{R}$ is an equivalence of $\cA$-module categories,
      \begin{equation}
        \mathcal{M} \cong \underline{\End}(m)-\mod_{\cA}
      \end{equation}
  \end{enumerate}
\end{theorem}

\subsection[The redecoration of RepqB]{The redecoration $\widetilde{\Rep}_qB$ of $\Rep_q B$}\label{sec:repqB-renormalisation}
Geometrically speaking, it is clear that $\Rep B$ and its quantization $\Rep_qB$ are the natural systems of local coefficients for a skein theory modelling the classical and quantum $A$-polynomial, respectively.  Indeed, following the framework of \cite{BZFN} and of \cite{JLSS2021}, one may identify the stratified factorization homology of the triple $(\Rep G,\Rep B, \Rep T)$ on a bipartite 3-manifold with the decorated character stack on the same bipartite 3-manifold.

From a skein-theoretic perspective, however, there is a problem:  in the setting of non-semisimple categories \cite{Brown_Haioun_2024} such as $\Rep_qB$, skeins are labelled by compact-projective objects, whereas the categories $\Rep B$ and $\Rep_qB$ have no compact-projective objects whatsoever.  In this section we apply monadic construction to give a resolution to this dilemma.

It is easy to compute the algebra $\underline{\End}_{G\times T}(\idty_B)$ from the definition in the case $\mathcal{M}=\Rep_qB$, using Lemma \ref{lemma:defect-skein-unique}.  We have an isomorphism:
\[
\underline{\End}_{G\times T}(\idty_B)\cong \bigoplus_{\lambda\in \Lambda^+}  V(\lambda)^* \boxtimes \lambda \quad \in \Rep_q G \boxtimes  \Rep_q T^{bop}.
\]

The algebra structure coincides with that of the $U_q(\mathfrak{n})$-invariant subalgebra of the FRT algebra $\cO_q(G)$, hence this algebra is often denoted $\cO_q(G/N)$ in the literature.  We note that
$\underline{\End}_{G\times T}(\idty_B)$ is a commutative algebra object -- this is a general feature of the monadic reconstruction of a central tensor category. 

\begin{definition}
We define $\widetilde{\Rep}_qB$ to be the $(\Rep_q G,\Rep_qT)$-central algebra,
\[\widetilde{\Rep}_qB = \underline{\End}_{G\times T}(\idty_B)-\mod_{\Rep_qG\boxtimes \Rep_qT^{bop}}.\]
\end{definition}

\begin{lemma}[{\cite[Prop. 3.41]{JLSS2021}}]\label{lem:monogon}
  The internal global sections functor
  \[
  \underline{\End}_{G\times T}(\idty_B,-) : \Rep_qB \to \widetilde{\Rep}_qB,\]
  is conservative but not colimit preserving.
\end{lemma}

  From Lemma \ref{lem:monogon} we obtain a reflexive embedding.
\begin{equation}
  \Rep_qB \hookrightarrow \widetilde\Rep_qB = \cO_q(G/N)-\mod_{\Rep_q G\boxtimes \Rep_qT^{bop}}
\end{equation}
of $\Rep_qG\boxtimes \Rep_q T^{bop
}$-module categories.

Skein-theoretically, $\widetilde{\Rep_q}B$ is very natural.  The subcategory of $\Rep_qB$ generated by $\idty_B$ is precisely the free co-completion of the category $\SkCat^\pitchfork_\cB(\DD_\cB)$.  Hence, upon passing to skein categories of arbitrary surfaces, we obtain the following:
\begin{corollary}
We have an equivalence of categories
\[\widehat{\SkCat}^\pitchfork_{\Rep_qB}(\Sigma)\simeq \widehat{\SkCat}_{(\widetilde{\Rep}_qB)^{cp}}(\Sigma)
\]
\end{corollary}

\subsection{Quantum tori}\label{sec:quantum-tori}
Our approximation of the quantum A-polynomial (Section \ref{sec:quantum-A-polynomial}) relies heavily on quantum cluster charts of internal decorated skein algebras.

    \begin{definition}
        Let $\Lambda$ be a lattice with a skew symmetric pairing $\omega : \Lambda \times \Lambda \to \mathbb{Z}$. The \emph{quantum torus $\WW_\Lambda$} is the $\C_q:= \C[q^{\frac12},q^{-\frac12}]$ algebra generated by $X^v$, $v\in \Lambda$ with multiplication given by
        \begin{equation}
            X^v X^w = q^{\frac12\omega( v, w)} X^{v+w}
        \end{equation}
    \end{definition}

    Given a basis $\{e_i\}$ for $\Lambda$, we let $\langle-,-\rangle : \Lambda \times \Lambda \to \ZZ$ denote the inner product making $\{e_i\}$ an orthonormal basis.  Then we may write $\omega(-,-) := \langle-,\Omega-\rangle$ for a square, skew-symmetric matrix $\Omega$.

The use of quantum tori to study skein algebras goes back to the quantum trace map of \cite{Bonahon_Wong_2011a,Bonahon_Wong_2011b}, and is tightly related to the construction of cluster coordinates on character varieties \cite{Fock_Goncharov_2009a,Fock_Goncharov_2009b}.
\begin{definition}
Let $\WW$ be a quantum torus with a $T^n$-action, and that $\Sigma$ is a decorated surface with $n$ $T$-region gates.
    A \emph{quantum cluster chart} of a decorated internal skein algebra $\SkAlg^{int}(\Sigma)$ is a isomorphism of $T^n$ representations $\varphi : \SkAlg^{int}(\Sigma)[S^{-1}] \to \WW$, where $S$ is some Ore set.
\end{definition}
It is convenient to describe $T$-actions on quantum tori in terms of a map $s : \Lambda \to \ZZ^{n}$ describing the integral weights in terms of the underlying lattice elements.
We are primarily interested in quantum cluster charts as computation tools.
The underlying lattice is amenable to computer algebra systems, so whenever possible we frame the constructions of Section \ref{sec:quantum-A-polynomial} in terms of the lattice.
The following lemma is a useful part of this program.

A sublattice $\Gamma \subset \Lambda$ is called \emph{isotropic} if the pairing restricted to $\Gamma$ vanishes identically; this is equivalent to asking that the subalgebra generated by $X^v$ for $v\in\Gamma$ is commutative.  A much stronger condition on $\Gamma$ is that it lie in the kernel of the pairing; this is equivalent to asking that the subalgebra generated by $X^v$ for $v\in\Gamma$ is central.  Finally, we say that $\Gamma$ is \emph{unimodular} if the quotient $\Lambda/\Gamma$ is torsion-free.

We require the following elementary lemma, whose proof follows from the fact that $\WW_\Lambda$ is free over $\WW_\Gamma$ with basis given by monomials $X^v$ for $v\in\Lambda/\Gamma$.

\begin{lemma}\label{lemma:central-quotient} Let $\Gamma\subset \Lambda$ be an isotropic sublattice, and fix a character $\chi:\Gamma\to\C$.  Then a basis for the induced module
$\WW_\Lambda \otimes_{\WW_\Gamma} \C_\chi$ is given by $\{X^v, \textrm{ for }v\in \Lambda/\Gamma\}$.  In particular, suppose that $\Gamma$ is unimodular (i.e. the quotient is torsion free) and lies in the kernel of the pairing. Then $\WW_\Lambda \otimes_{\WW_\Gamma} \C_\chi$ is naturally an algebra, and we have an isomorphism,
\[
\WW_\Lambda \otimes_{\WW_\Gamma} \C_\chi \cong \WW_{\Lambda/\Gamma}.
\]
\end{lemma}

\begin{remark}
The relative tensor product $\WW_\Lambda \otimes_{\WW_\Gamma} \C_\chi$ may be presented as the quotient by the left ideal $I=\langle X^{v_1} - \chi(v_1),\ldots,X^{v_k}-\chi(v_k) \rangle$, for any spanning set $v_1,\ldots, v_k$ of $\Gamma$.  This is how it typically arises.    
\end{remark}

\begin{figure}
    \centering
    \includegraphics{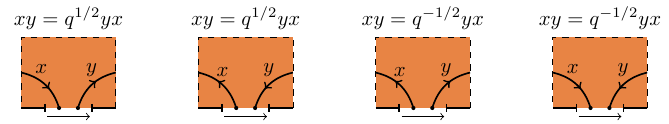}
    \caption{$T$-region commutation relations for skeins meeting only at a single gate.}
    \label{fig:T-region-commutation-relations}
\end{figure}

\subsection{Building block examples}\label{sec:building-blocks}

We will compute quantum cluster charts for the internal decorated skein algebras of a few bipartite surfaces.
In this section we will always work with the local coefficient system associated to the central structure $\Rep_qG\boxtimes \Rep_qT^{bop}\to Z_{Dr}(\widetilde{\Rep}_qB)$ for $G = \SL_2\C$ and $B$ the upper triangular Borel.

\begin{figure}
    \centering
    \includegraphics[width=.9\linewidth]{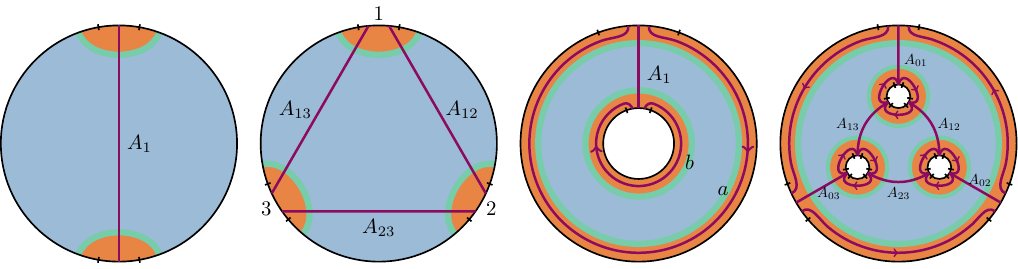}
    \caption{From left to right: The digon $\DD_2$, the triangle $\DD_3$, the annulus $\widetilde{Ann}$, and the four-punctured sphere $\TetSurf$. In each case the generators of a quantum cluster chart computed in Section \ref{sec:building-blocks} are shown in purple, the $T$-regions in orange, and the $G$-regions in blue.}
    \label{fig:basic-stratified-surfaces}
\end{figure}

\subsubsection{The digon}
By Lemma \ref{lemma:defect-disk-gives-B} and excision, we have
\[\SkCat(\DD_2) \cong \Rep_qB\underset{\Rep_qG}{\otimes}\Rep_qB\] Here $\DD_2$ is the digon shown in Figure \ref{fig:basic-stratified-surfaces}.  We make this a $\Rep_q G\boxtimes \Rep_qT^2$-module category by adding one gate in each open region of the boundary. (Note that the figure only shows the $T$-gates.)

By \cite[Theorem 3.49]{JLSS2021}, we have an isomorphism in $\Rep_q G \boxtimes  \Rep_q T \boxtimes  \Rep_q T$:
\[
\SkAlg^{int}_{G\times T^2}(\DD_2) \cong \SkAlg^{int}_{G\times T}(\DD_B)\widetilde{\otimes} \SkAlg^{int}_{G\times T}(\DD_B) \cong \bigoplus_{\lambda,\mu\in \Lambda^+}  \left(V(\lambda)\otimes V(\mu)\right)^* \boxtimes \,\lambda \boxtimes\, \mu
\]
In applying the result from \cite{JLSS2021}, we've used that $\SkAlg^{int}_{\mathcal{G}}(\Sigma) \cong \underline{\End}_\mathcal{G}(\mathrm{Dist}_\Sigma)$.
Similar to Lemma \ref{lem:monogon}, the right adjoint to disk insertion,
\[\underline{\Hom}(\idty_B\boxtimes \idty_B,-) : \SkCat(\DD_2) \to \Rep_qG\boxtimes \Rep_qT\boxtimes \Rep_qT,\]
  is conservative but not colimit preserving.
We can however still apply monadic reconstruction to a localized subcategory.
  
Let $A_{\lambda} \in \Sk(\DD_2\times[0,1]; \P(\idty_{G}\boxtimes\lambda\boxtimes\lambda))$ be the skein passing between the two $T$-regions, coloured by $\lambda \in \Rep_qT$ in the $T$-regions and $V(\lambda)\in \Rep_q G$ in the $G$-region.
Here $V(\lambda)$ is the unique finite dimensional simple $U_q\mathfrak{g}$-representation whose highest weight space is a copy of $\lambda$.
(Recall that not every $\lambda$ has an associated $V(\lambda)$ -- the weight must be in $\Lambda^+$.)

\begin{lemma}[The digon {\cite[Theorem 3.62]{JLSS2021}}]\label{lem:digon}
Fix $G = \SL_2\C$ and let $\SkAlg^{int}_{G\times T^2}(\DD_2)[A^{-1}_\lambda]-\mod_{G\times T^2}$ be the open subcategory on which the $A_{\lambda}$ are invertible.
The restriction of the comparison functor to this localized subcategory is both conservative and colimit preserving. 
Moreover, the functor of taking $G$-invariants is also conservative upon restriction and we obtain an equivalence
\[\SkAlg^{int}_{G\times T^2}(\DD_2)[A^{-1}_{\lambda}]^{U_q\mathfrak{g}}-\mod_{G\times T^2} \cong \SkAlg^{int}_{T^2}(\DD_2)[A^{-1}_{\lambda}]-\mod_{T^2}.\]
\end{lemma}
By $A_\lambda A_\mu = A_{\lambda+\mu}$ and Lemma \ref{lemma:defect-skein-unique}, we have for $G = \SL_2\CC$:
\begin{equation}
    \SkAlg^{int}_{T^2}(\DD_2)[A^{-1}_{\lambda}] \simeq \K\langle A_{1}^{\pm 1}\rangle.
\end{equation}
The generators $A_1^{\pm 1}$ have weight $\pm 1$ at the two gates.

\subsubsection{The triangle.}
We now give a version of Lemma \ref{lem:digon} for the triangle $\DD_3$ shown in Figure \ref{fig:basic-stratified-surfaces} and its skein category:
\[
\SkCat(\DD_3) \cong \Rep_qB\underset{\Rep_qG}{\otimes}\Rep_qB\underset{\Rep_qG}{\otimes}\Rep_qB.
\]

Label the $T$-regions $i=1,2,3$ in clockwise order and let $A_{(\lambda_1,\lambda_2,\lambda_3)} \in \SkAlg^{int}_{G\times T^3}(\DD_3)$ denote the skein coloured by the simple $U_q\mathfrak{t}$-representation $\lambda_i$ in the $i$th $T$-region.
By Lemma \ref{lemma:defect-skein-unique} and $A_{\underline{\lambda}}A_{\underline{\nu}} = q^{\frac{1}{2}\omega\left({\underline{\lambda},\underline{\nu}}\right)} A_{\underline{\lambda}+\underline{\nu}}$, the sub-algebra spanned by the $A_{\underline{\lambda}}$ is generated by $A_{12} := A_{(1,1,0)}$, $A_{13} := A_{(1,0,1)}$, and $A_{23} := A_{(0,1,1)}$.

\begin{lemma}[The triangle {\cite[Lemma 4.5]{JLSS2021}}]
The comparison functor 
\begin{equation}
    \SkAlg^{int}_{G\times T^3}(\DD_3)-\mod_{G\times T^3} \to \SkCat(\DD_3)
\end{equation}
becomes conservative and colimit preserving upon restriction to the open subcategory 
where the $A_{ij}$ actions are invertible. 
Moreover, the functor of taking $G$-invariants is also conservative on restriction to this open subcategory, and we obtain an equivalence
\[\SkAlg^{int}_{G\times T^3}(\DD_3)[A_{12}^{-1},A_{13}^{-1},A_{23}^{-1}]^{U_q\mathfrak{g}}-\mod_{G\times T^3} \cong \SkAlg^{int}_{T^3}(\DD_3)[A_{12}^{-1},A_{13}^{-1},A_{23}^{-1}]-\mod_{T^3}.\]

\end{lemma}
The localisation $\SkAlg^{int}_{T^3}(\DD_3)[A_{12}^{-1},A_{13}^{-1},A_{23}^{-1}]$ is a rank three quantum torus generated by the $A^{\pm 1}_{ij}$ with commutation relations
\begin{equation}
    \begin{aligned}
        A_{12} A_{23} &= q^{1/2} A_{23}A_{12},& A_{23} A_{13} &= q^{1/2} A_{13}A_{23}, & A_{13} A_{12} &= q^{1/2} A_{12}A_{13}
    \end{aligned}
\end{equation}

The $\Rep_qT^3$ action is determined by $A_{ij}^{\pm 1}$ having weight $\pm 1$ at the $i$th and $j$th gate and weight zero elsewhere.

\subsubsection{The annulus}
The decorated annulus $\widetilde{Ann}$ shown in Figure \ref{fig:basic-stratified-surfaces} is not triangulated, which distinguishes it from the surfaces considered in \cite{JLSS2021}.
Its internal skein algebra nonetheless admits a quantum cluster chart, which we will now compute.
First, its skein category can be presented via excision as the Hochschild homology of the digon:
\[
\SkCat(\widetilde{Ann}) = \SkCat(\DD_2)\underset{\SkCat(\DD_2)\boxtimes \SkCat(\DD_2)^{op}}{\boxtimes} \SkCat(\DD_2).
\]
Let $\SkAlg^{int}_{T^2}(\widetilde{Ann})$ be the internal skein algebra obtained by opening a gate on each boundary component of the annulus.
We write $a,b$ for the two $T$-monodromies (see Figure \ref{fig:basic-stratified-surfaces} for their orientations) and $A_\lambda$ for the skein passing directly between the two $T$-regions, which is coloured by the simple $U_q\mathfrak{t}$ representation $\lambda$ in each $T$-region. (Again, recall that this implies $\lambda \in \Lambda^+$.)

\begin{lemma}\label{lemma:the-decorated-annulus}
The comparison functor is conservative and colimit preserving upon restricting to the open subcategory obtained by localising $A_\lambda$, giving a open embedding
\[
\SkAlg^{int}_{T^2}(\widetilde{Ann})[A_\lambda^{-1}]-\mod_{T^2}\hookrightarrow \SkCat(\widetilde{Ann}).
\]
The localisation $\SkAlg^{int}_{T^2}(\widetilde{Ann})[A_\lambda^{-1}]$ is a rank three quantum torus generated by $a^{\pm 1}, b^{\pm 1}, A_{1}^{\pm 1}$, with commutation relations
\begin{equation}
\begin{aligned}
       A_1 a &= q a A_1, & A_1 b &= qb A_1, & ab &= ba.
\end{aligned}
\end{equation}
Both $a$ and $b$ are $T$-invariant, while $A_1$ has weight $+1$ at both gates.

\end{lemma}

\begin{proof}
We use the fact that this bipartite annulus is the digon $\DD_2$ times a circle.
Taking the factorisation homology of the localised skein category of the digon, applying Lemma \ref{lem:digon} and braided commutativity of the algebra $\SkAlg^{int}_{T^2}(\DD_2)$, we compute:
\begin{equation}\label{eq:annulus-category}
    \int_{S^1} \left( \SkAlg^{int}_{T^2}(\DD_2)[A_\lambda^{-1}]-\mod_{\Rep_qT\boxtimes \Rep_q T} \right) \simeq \left( \SkAlg^{int}_{T^2}(\DD_2)[A_\lambda^{-1}] \right)-\mod_{\int_{S^1}(\Rep_qT\boxtimes \Rep_q T)}.
\end{equation}
Applying monadic reconstruction to $\int_{S^1} \Rep_q T \boxtimes \Rep_q T$, we get
\begin{equation}
\begin{aligned}
    \int_{S^1} \Rep_q T \boxtimes \Rep_q T &\simeq \SkAlg^{int}_{T}(Ann)\boxtimes \SkAlg^{int}_T(Ann) -\mod_{\Rep_q T \boxtimes \Rep_qT}, \\
    & \simeq \K [ a^{\pm1},b^{\pm1}] -\mod_{\Rep_q T\boxtimes \Rep_qT}
\end{aligned}
\end{equation}
where here $Ann$ is the $T$-coloured annulus with no defects.
It follows from the $T$-region crossing relations (Figure \ref{fig:kauffman-crossings}) that $A_1 a = q a A_1$ and $A_1 b = q b A_1$.
An object in \eqref{eq:annulus-category} must therefore have compatible $T^2$-equivariant actions of both $\K[A_\lambda^{\pm 1}]$ and $\K[a^{\pm 1}, b^{\pm 1}]$.
We conclude that
\begin{equation}
\SkAlg^{int}_{T^2}(\widetilde{Ann})[A_\lambda^{-1}]-\mod_{T^2} \simeq \K\langle a^{\pm 1} , b^{\pm 1}, A_1^{\pm 1}\mid A_1 a = q a A_1, A_1 b = q b A_1\rangle-\mod_{T^2}
\end{equation}
\end{proof}

\subsubsection{The four-punctured sphere}\label{sec:ideal-tet-boundary}
Building on the digon and the triangle, one the main results of \cite{JLSS2021} is a construction of a quantum cluster chart for any simple decorated surface equipped with a triangulation.
We recall the four-punctures sphere as a special case, which will feature prominently in Section \ref{sec:quantum-A-polynomial}.

Let $M_{tet}$  denote the closed ball $B^3 \subset \mathbb{R}^3$ stratified and labelled as follows.
Pick four distinct points on the boundary and label a contractible 3-dimensional neighbourhood of each by $T$.
Label the rest of $M$ by $G$.
This is our model for a $G$-labelled tetrahedron with a small $T$-region surrounding each vertex.
Let $\TetSurf:= \left(\partial M_{tet}\right) \setminus (\DD_T^{\sqcup 4})$ be the four-punctured sphere with an annular $T$-region surrounding each puncture, as shown in Figure \ref{fig:basic-stratified-surfaces}.

We equip each boundary component with three gates.
This is two gates per puncture \emph{more} than is required for Lemma \ref{lemma:muller-applies}, but facilitates gluing tetrahedra along faces.
The 18 edges of the truncated tetrahedron determine a distinguished set $\triangle_{tet}$ of elements in $\SkAlg^{int}_{\Rep_qT^{\boxtimes 12}}(\TetSurf)$.

Next, let $\WW_{tet}$ be the quantum tori with underlying lattice $\Lambda_{tet} \simeq \ZZ^{18}$ defined as follows.
Fix an ordered basis $\{e_{01},e_{02},e_{03},e_{10},e_{13},e_{12},e_{20},e_{21},e_{23},e_{30},e_{32},e_{31},E_{01},E_{02},E_{03},E_{12},E_{13},E_{23}\}$ for $\Lambda_{tet}$, so 
that the product structure on $\WW_{tet}$ is determined by the skew symmetric matrix $\Omega_{tet}$:
\begin{equation}\label{eq:tet-communtation-relations}
\Omega_{tet} :=
\left(
\begin{gathered}
    \begin{tikzpicture}
    \draw[dashed]
        (.8,-.8) rectangle ( 0,0) node[pos=.5] {$\Omega_{sh}$}
        (2*.8,-2*.8) rectangle (1*.8,-1*.8) node[pos=.5] {$\Omega_{sh}$}
        (3*.8,-3*.8) rectangle (2*.8,-2*.8) node[pos=.5] {$\Omega_{sh}$}
        (4*.8,-4*.8) rectangle (3*.8,-3*.8) node[pos=.5] {$\Omega_{sh}$}
        (4.15*.8,0) rectangle (6*.8,-4*.8) node[pos=.5] {$\Omega_{sh,lg}$}
        (0,-4.15*.8) rectangle (4*.8,-6*.8) node[pos=.5] {$-\Omega_{sh,lg}^\perp$}
        ;
    \end{tikzpicture}
    \end{gathered}
    \right), %
    \Omega_{sh,lg} := \left(\begin{array}{rrrrrr}
    0 & 1 & 1 & 0 & 0 & 0 \\
    1 & 0 & 1 & 0 & 0 & 0 \\
    1 & 1 & 0 & 0 & 0 & 0 \\
    0 & 0 & 0 & 1 & 1 & 0 \\
    1 & 0 & 0 & 1 & 0 & 0 \\
    1 & 0 & 0 & 0 & 1 & 0 \\
    0 & 0 & 0 & 1 & 0 & 1 \\
    0 & 1 & 0 & 0 & 0 & 1 \\
    0 & 1 & 0 & 1 & 0 & 0 \\
    0 & 0 & 0 & 0 & 1 & 1 \\
    0 & 0 & 1 & 0 & 1 & 0 \\
    0 & 0 & 1 & 0 & 0 & 1
    \end{array}\right)
\end{equation}
where $\Omega_{sh} := \left(\begin{smallmatrix} 0 & 1 & -1 \\ -1 & 0 & 1 \\ 1 & -1 & 0\end{smallmatrix}\right)$.

\begin{lemma}{\cite[Prop 4.11]{JLSS2021},\cite[Thm 6.14]{Muller_2016}}\label{lemma:muller-applies}
    The map \begin{equation}
    \phi : \WW_{tet} \to \SkAlg^{int}_{T^{12}} (\TetSurf)[\triangle_{tet}^{-1}] 
\end{equation}
sending $X^{\pm e_{ij}} \mapsto a_{ij}$ and $X^{\pm E_{ij}} \mapsto A_{ij}^{\pm 1}$ is an isomorphism.
\end{lemma}
\begin{proof}
    The skew-symmetric matrix $\Omega_{tet}$ was constructed exactly so that $\phi$ would be an homomorphism.
    The relations are obtained from the commutation relations in Figure \ref{fig:T-region-commutation-relations} as applied to the elements of $\triangle_{tet}$ shown and given a total order in Figure \ref{fig:gen-ideal-tet}.
Namely, two edges commute if they have no shared vertex, otherwise their exact commutation relation depends on their weights at their shared gate.

 This lemma then reduces to \cite[Thm 6.14]{Muller_2016}, after noting that the stated skeins of \cite{Muller_2016} are isomorphic to our internal skein algebras by \cite{H2021}.
   The other main difference is that unlike in Muller's theorem, we are not working full triangulation of $\TetSurf$, since we could add additional non-parallel non-intersecting long edges.
   However, any two skeins incident to different gates on the same $T$-region puncture are related by a product of short edges.
    Therefore the missing edges are contained in $\mathbb{W}_{tet}$.
\end{proof}

Note that the short edges $a_{ij}$ are already invertible in $\SkAlg^{int}_{T^{12}}(\TetSurf)$, so that the localisation only introduces formal inverses of the long edges $A_{ij}$

\paragraph{The internal skein module for the ideal tetrahedron}
To power our computations in Section \ref{sec:quantum-A-polynomial}, we now write the skein module of $M_{tet}$ as a quotient of quantum torus $\mathbb{W}_{tet}$.
\begin{lemma}\label{lemma:ideal-tet-presentation}
   The localized submodule generated by the empty skein $\SkAlg^{int}(\TetSurf)[\triangle_{tet}^{-1}]\cdot \varnothing \subset \SkMod^{int}(M_{tet})[\triangle_{tet}^{-1}]$ is isomorphic to $\mathbb{W}_{tet}$ quotiented by the left ideal generated by the element
\begin{equation}\label{eq:crossing-relation}
    q^{5/2} (A_{03} a_{32} a_{01}) (A_{12} a_{23} a_{10}) + q^{3/2} (A_{01} a_{03}^{-1} a_{12}^{-1}) (A_{23} a_{21}^{-1} a_{30}^{-1}) + q^{3/2}A_{13}A_{02},
\end{equation}

Where the variables are as shown in Figure \ref{fig:gen-ideal-tet}.
\end{lemma}

Note that restricting to the submodule generated by the empty skein corresponds to the restriction to the image of eigenvalue map in the construction of the classical A-polynomial.

\begin{figure}
  \begin{center}
    \includegraphics[]{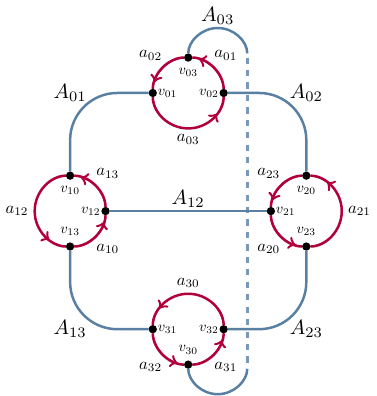}
    \caption{ The full set of short and long edges which generate $\mathbb{W}_{tet}$. The long (blue) edges have weight +1 at each gate and pass between $T$-regions. The short (red) edges are oriented to go from their weight +1 end to their weight -1 end and remain in a single $T$ region. When relevant for matrix computations, the generators are ordered as follows: $a_{01},a_{02},a_{03},a_{10},a_{13},a_{12},a_{20},a_{21},a_{23},a_{30},a_{32},a_{31},A_{01},A_{02},A_{03},A_{12},A_{13},A_{23}.$
    Labels for the twelve vertices/gates are also shown.}\label{fig:gen-ideal-tet}
  \end{center}
\end{figure}

\begin{proof}
By assumption the module is cyclic, and by Lemma \ref{lemma:muller-applies},
\begin{equation}
    \SkAlg^{int}_{T^{12}}(\TetSurf)[\triangle_{tet}^{-1}] \cong \mathbb{W}_{tet}.
\end{equation}
Hence $\SkAlg^{int}_{T^{12}}(\TetSurf)[\triangle_{tet}^{-1}]\cdot \varnothing \cong \mathbb{W}_{\triangle_{tet}}/ I$ for some left ideal $I$.

Let $\DD_4$ be the disk with four disjoint contractable $T$-regions along its boundary, so that
 $\DD_4\times [0,1] \cong M_{tet}$.
We consider a single gate in each $T$-region, labelled $0,1,2,3$ in clockwise order in both $\DD_4$ and $M_{tet}$, and note that 
\begin{equation}
    \Sk^{int}(M_{tet}) \cong \Sk^{int}(\DD_4\times I) \cong \SkAlg^{int}(\DD_4)\cdot \varnothing_{\DD_4\times I}^{} 
\end{equation}
as left $\SkAlg^{int}(\DD_4)$ modules\footnote{The action on $M_{tet}$ is induced by the embedding $\DD_4 \hookrightarrow \TetSurf \simeq \DD_4 \cup \DD_4$.}

Let $\triangle_{02} := \{B_{01},B_{02},B_{03},B_{12},B_{23}\} \subset \SkAlg^{int}_{T^4}(\DD_4)$ be a triangulation of the quadrilateral.
It follows from Lemma \ref{lemma:muller-applies}, $\SkAlg^{int}_{T^4}(\DD_4\times I)[\triangle_{02}^{-1}]$ is generated by $\triangle_{02}\cup \triangle_{02}^{-1}$, and from Lemma \ref{lemma:kauffman-presentation} that $B_{13} = B_{02}^{-1}\left(q^{1/2}B_{03}B_{12} +  q^{-1/2} B_{01}B_{23} \right)$ in the localization. 

To get the exact form of \eqref{eq:crossing-relation}, we identify the four-gate localized internal skein module of $M_{tet}$ with the sub-quantum torus of $\mathbb{W}_{{tet}}$ generated by 
\begin{equation}
    \begin{aligned}
        A'_{01} &:= q A_{01}a_{03}^{-1}a_{12}^{-1} & A'_{02} &:= A_{02} & A'_{03} &:= q A_{03}a_{32}a_{01} \\
        A'_{12} &:= q A_{12}a_{23}a_{10} & A'_{13} &:= A_{13} & A'_{23} &:= q A_{23} a_{21}^{-1}a_{30}^{-1}
    \end{aligned}
\end{equation}
and note that the isomorphism between the internal decorated skein modules for $M_{tet}$ and $\DD_4\times I$ is given by
\begin{equation}
    A'_{ij} \mapsto \begin{cases}
        -q^{3/2} B_{13} & ij = 13 \\
        B_{ij} & \text{otherwise.}
    \end{cases}
\end{equation}
where $A'_{13}$ picks up the factor of $-q^{3/2}$ due to the framing difference between $A_{13}'$ and $B_{13}.$

\end{proof}

\section[The localised quantum A-ideal]{The localised quantum $A$-ideal}\label{sec:quantum-A-polynomial}

In this section we explain how a triangulated knot complement $M_K:= S^3\setminus K$ determines a localization of the decorated skein module $\SkMod(\knotcomp)$ which defines the quantum A-polynomial.
Worked examples will be given in Section \ref{sec:worked-examples}.

\subsection[Quantum cluster charts]{Quantum cluster charts for $\SkAlg^{int}(\Sigma_\triangle)$ and $\SkMod^{int}(\Mbulk)$}\label{sec:quantum-cluster-charts}

Fix an ideal triangulation $\Delta$ of the knot complement $\knotcomp$, let $t$ denote the number of tetrahedra.
The union of the knot and the one-skeleton of the triangulation determines an is an embedded graph $K\cup\skel_1(\triangle) \hookrightarrow S^3$
Let $\Mbulk$ be the complement of a tubular neighbourhood of this graph.
In the introduction $\knotcomp$ is bi-partitioned by $B$-defect parallel to its boundary, separating a $T$-region tubular neighbourhood of its boundary from a $G$-region bulk.
We use the natural embedding $\Mbulk \hookrightarrow M_K$ to induce a bipartite structure on $\Mbulk$ and its boundary $\Sigma_\triangle$.

Note that $\Sigma_\triangle$ is a genus $t+1$ surface with the stratification $\Sigma_\triangle = \left(\Sigma_{2,2t}\right)_T \bigcup_{(S^1)^{\sqcup\, 2t}} \left( Ann^{\sqcup \,t}\right)_G$ and $4t$ punctures in its $T$-region (so that the $T$ region is $\Sigma_{2,6t}$.) Here $Ann\cong S^1\times[0,1]$. See Figure \ref{fig:the-big-surface}.

We construct a quantum cluster chart $\mathbb{W}_\triangle$ for the internal decorated skein algebra of $\Sigma_\triangle$.
Having done this, we give a compatible cluster chart for the internal skein module of $\Mbulk$.
Let $t$ be the number of tetrahedra in $\triangle$.

By Lemma \ref{lemma:ideal-tet-presentation}, each ideal tetrahedra has an associated rank 18 quantum torus $\mathbb{W}_{tet}$ with underlying lattice $\Lambda_{tet}$ and q-commutation relations given by the skew symmetric matrix $\Omega_{tet}$.
The generators of $\Lambda_{tet}$ consist of 12 short edges and 6 long edges.
Let $P_{sh} = \{ (e_1,f_1,) ,\ldots , (e_{6t},f_{6t})\}$ be the the pairs of coordinates whose corresponding short edges are identified by the ideal triangulation $\triangle$. 
Denote by $e_i^{+}$ and $e_i^{-}$ (resp. $f_i^+, f_i^i$) the weight +1 and -1 vertices $e_i$ (resp. $f_i$.)

We construct a rank $12t$ lattice $\Lambda_{th}$ of \emph{threads} with generators
\begin{equation}
     \theta_{e_1},\theta_{f_1},\ldots,\theta_{e_{6t}},\theta_{f_{6t}}.
\end{equation}
The thread $\theta_{e_i}$ (resp. $\theta_{f_i}$) has weight +1 at $e_i^-$ (resp. $f_i^-$) and -1 at $f_i^+$ (resp. $e_i^+$.)
Similarly, let $P_{lg}:= \{(E_1,F_1),\ldots,(E_{6t},F_{6t})\}$ denote pairs of associated long edges. For each pair $(E_i,F_i) \in P_{lg}$, let $(\theta_{E_i},\theta_{F_i})$ be the unique pair of threads such that $E_i + \theta_{E_i} - F_i - \theta_{F_i}$ has weight zero.
Such a pair exists because if $E_i$ and $F_i$ are identified, then so are two pairs of their adjacent short edges. It's unique because no two threads have the same weight at the same gate.

\begin{figure}
    \centering
    \includegraphics[width=0.7\linewidth]{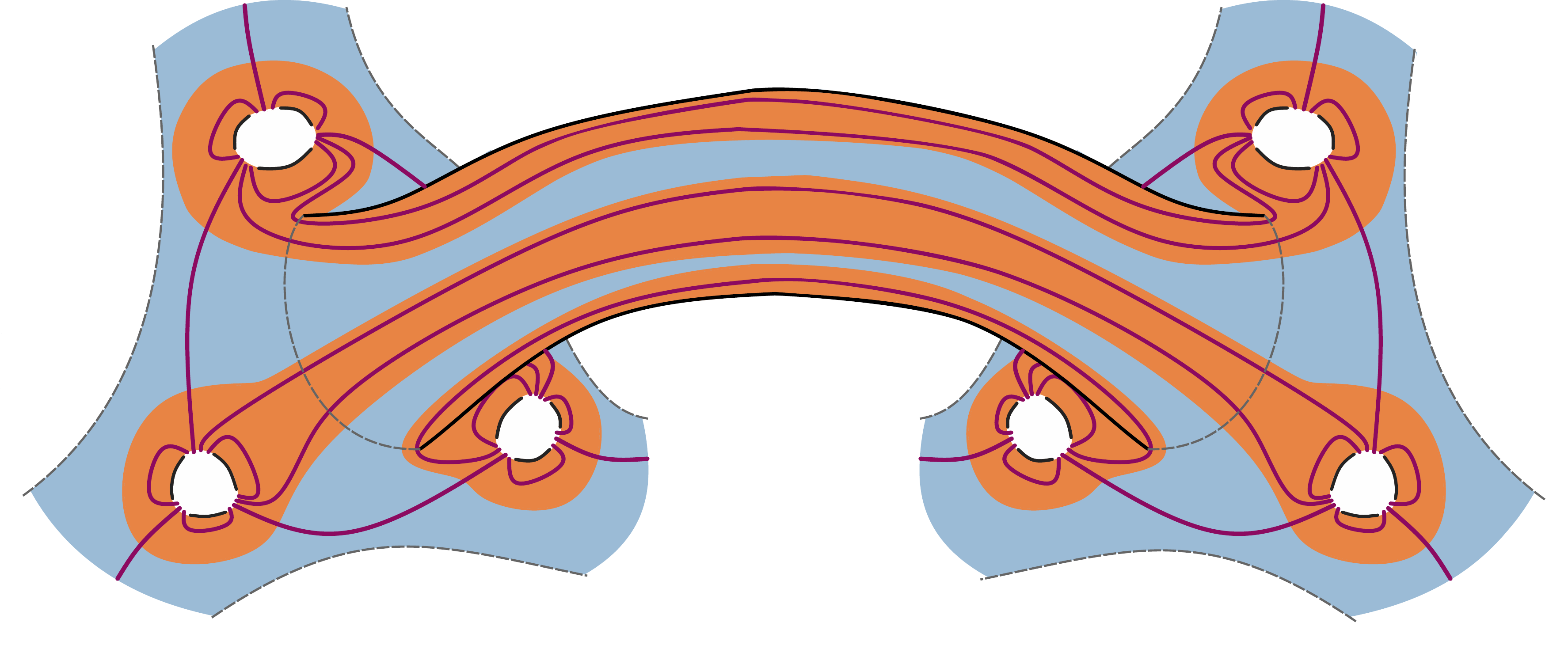}
    \caption{A portion of a surface with the gluing handle attachment shown. Long edges pass through $G$-regions (blue) to connect punctures, short edges surround punctures within a $T$-region (orange), and threads travel between tetrahedra.
    To reduce visual clutter we omit the $B$-defect separating the $G$-region and $T$-region. }
    \label{fig:threads}
\end{figure}

Let $\Lambda':=\Lambda_{th}\times \left( \Lambda_{tet}\right)^{\times t}$ be the lattice generated by all short edges, long edges, and threads.
We next define a $30t\times 30t$ matrix
\begin{equation}
\Omega' :=
\left(
\begin{gathered}
    \begin{tikzpicture}
    \draw[dashed]
        (.8,-.8) rectangle ( 0,0) node[pos=.5] {$\Omega_{tet}$}
        (3*.8,-3*.8) rectangle (2*.8,-2*.8) node[pos=.5] {$\Omega_{tet}$}
        (3.15*.8,0) rectangle (5*.8,-3*.8) node[pos=.5] {$\Omega_{tet,th}$}
        (0,-3.15*.8) rectangle (3*.8,-5*.8) node[pos=.5] {$-\Omega_{tet,th}^\perp$}
        (3.15*.8,-3.15*.8) rectangle (5*.8,-5*.8) node[pos=.5] {$\Omega_{th,th}$}
        ;
    \draw[very thick,loosely dotted,shorten >=4pt, shorten <=5pt] (.8,-.8) -- (2*.8,-2*.8);
    \end{tikzpicture}
    \end{gathered}
    \right) %
\end{equation}
which induces a skew-symmetric pairing on $\Lambda'$.
Here $\Omega_{tet,th}$ and $\Omega_{th,th}$ describe the commutation relations involving threads.

The matrix $\Omega_{tet,th}$ is determined by the value of the associated bilinear form $\langle -,\Omega'-\rangle:\Lambda'\times\Lambda' \to \ZZ$ on the sublattice  $\Lambda_{th}\times (\Lambda_{tet}^{\times t})$. These values are chosen precisely so that the algebra homomorphism $\WW_\triangle \to \SkAlg^{int}(\Sigma_\triangle)$ in the proof of Theorem \ref{thm:cluster-chart-for-sigma-tri} is well defined. 
For the long edges, we require that:
\begin{equation}
    \begin{aligned}
                \langle \theta_{E_j}, \Omega' E_j\rangle &=-1, &%
                \langle \theta_{E_j}, \Omega'F_j\rangle &=-1, &%
                \langle \theta_{F_j}, \Omega' F_j\rangle &=-1, &%
                \langle \theta_{F_j},\Omega' E_j\rangle &=-1 &
                 \qquad j=1,\ldots,6t.
    \end{aligned}
\end{equation}
All other long edges coordinates pair trivially with the threads.

For the short edges, we start with the values
\begin{equation}
    \begin{aligned}
        \langle \theta_{e_i},\Omega' e_i\rangle &= 1, &%
        \langle \theta_{e_i}, \Omega' f_i\rangle &= -1, &%
        \langle \theta_{f_i}, \Omega' f_i\rangle &= 1, &%
        \langle \theta_{f_i}, \Omega' e_i\rangle &= -1, & {i=1,\ldots,6t}.
    \end{aligned}
\end{equation}
Additionally, suppose that $e_\ell$ and $c$ are adjacent short edges, with $e_\ell^+ = c^-$. Since the ordering of each pair in $P_{sh}$ doesn't matter, we can assume that $c = f_k$ for some $k$. 
This implies that at $e^+_\ell$, the short edge $e_\ell$ is adjacent not only to $\theta_{e_\ell}$ but also $\theta_{e_k}$, imposing further q-commutation relations:
\begin{equation}\label{eq:rest-of-the-tet-th-q-commutation}
\begin{aligned}
    \langle \theta_{e_k},\Omega' e_{\ell}\rangle = -1,\qquad \text{for }\;\ell,k\; \text{ such that } e_\ell^+ = f_k^-.
\end{aligned}
\end{equation}
All other short edge coordinates pair trivially with the threads.

The matrix $\Omega_{th,th}$ is likewise constructed to facilitate the proof of Theorem \ref{thm:cluster-chart-for-sigma-tri} below.
Again consider the adjacent short edges $e_\ell$ and $f_k$. At their shared gate $e_\ell^+ = f_k^-$ there are exactly two incident threads\footnote{When $\ell=k$ this will be the two half-edges of a single thread, and we instead have $\langle \theta_{e_k},\Omega \theta_{e_\ell} \rangle = 0$.}, $\theta_{e_k}$ and $\theta_{e_\ell}$.
By construction $\theta_{e_k}$ has weight +1 at the common gate and $\theta_{e_\ell}$ has weight -1.
Their commutation relation will be determined by $\langle \theta_{e_k}, \Omega' \theta_{e_\ell}\rangle = 1$.
Since this is the only occasion in which threads meet at a gate, all other entries of $\Omega_{th,th}$ are zero.

\begin{figure}
    \centering
    \includegraphics[width=0.3\linewidth]{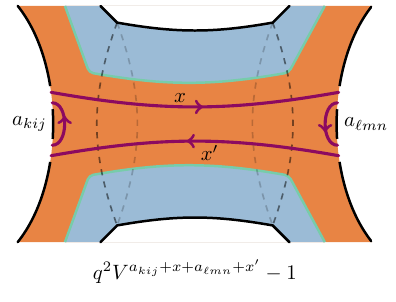}\hspace{1cm}
    \includegraphics[width=0.3\linewidth]{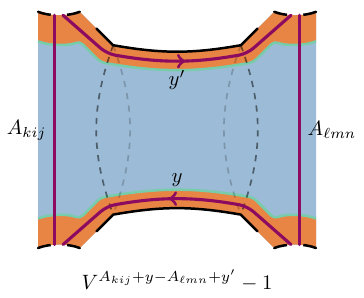}
    \caption{\textbf{Left:} The gluing relation identifying the pair $e_{kij},a_{\ell mn}$ of short edges. \textbf{Right:
    } The gluing relation identifying the pair $E_{kij},A_{\ell mn}$ of long edges. \\ Both types of relation rely on the addition of threads, labelled by $x,x',y,y'$ in the figure.
    }
    \label{fig:gluing-relations}
\end{figure}

\begin{lemma}\label{lemma:gluing-lattice-quotient}
    The sublattice $\Gamma_{gluing} \subset \Lambda'$ generated by 
    \begin{equation}
\begin{aligned}
      e_i + \theta_{e_i} + f_i + \theta_{f_i},&&  E_i + \theta_{E_i} - F_i - \theta_{F_i}, && i = 1,\ldots, 6t
\end{aligned}
\end{equation}
is unimodular and $\Gamma_{gluing} \subset \ker \Omega'$.
\end{lemma}
Both claims can be confirm by direct computation. 

We define a character $\chi : \Gamma_{gluing} \to \K$ by $\chi(e_i+\theta_{e_i} + f_i+\theta_{f_i}) = q^{-2}$ and $\chi(E_i + \theta_{E_i} - F_i - \theta_{F_i}) = 1$.
By Lemma \ref{lemma:gluing-lattice-quotient} and Lemma \ref{lemma:central-quotient}, we have $\WW_{\Lambda'} \otimes_{\WW_{\Gamma_{gluing}}}\K_\chi \cong \WW_{\Lambda'/\Gamma_{gluing}}.$
Let $\WW_\triangle := \WW_{\Lambda'/\Gamma_{gluing}}$.

\begin{theorem}\label{thm:cluster-chart-for-sigma-tri}
   The quantum torus $\WW_\triangle$ described immediately above is a quantum cluster chart for the internal decorated skein algebra of $\Sigma_\triangle$, i.e.
   \begin{equation}
       \SkAlg^{int}(\Sigma_\triangle)[S^{-1}_\triangle]\cong \mathbb{W}_\triangle.
   \end{equation}
   Where $S_\triangle \subset \SkAlg^{int}(\Sigma_\triangle)$ denotes the set of long edges determined by the triangulation.
\end{theorem}

\begin{proof}
Let $a_{kpe}, A_{kp_1p_2},\theta^{v_1}_{v_2}$ be the short edges, long edges, and threads, ordered lexicographically with respect to their indices, where $k=1,\ldots,t$ tracks the tetrahedra, $p, e,p_1,p_2 \in \{0,1,2,3\}$ correspond to various punctures, and $v_1,v_2$ run over the vertices.
The commutation relations in $\WW_{\Lambda'}$ were constructed exactly so that the map $\varphi: \WW_{\Lambda'} \to \SkAlg^{int}(\Sigma_\triangle)[S_\triangle^{-1}]$ which sends 
\begin{equation}
\begin{aligned}
        X^{e_i} \mapsto \text{$i$th short edge, }&& X^{E_i} \mapsto \text{$i$th long edge},&& X^{\theta_i} \mapsto \text{$i$th thread}
\end{aligned}
\end{equation}
is a linear homomorphism.
The only potential obstruction to surjectivity is the $G$-monodromy around the long edges.
A direct computation shows that after localising the long edges, this is equal to a two-term expression in the threads and therefore in the image.

Next we claim that the kernel of this map is generated by $X^v - \chi(v)$, for $v \in \Gamma_{gluing}$ and $\chi : \Gamma_{gluing} \to \K$ the character defined immediately before this theorem.
Inclusion $I_{gluing} := \langle X^{v} - \chi(v) \mid v \in \Gamma_{gluing}\rangle \subseteq \ker \varphi$ is immediate.
By linear independence of the gluing relations, $\WW_{\triangle} \simeq \WW_{\Lambda'} / I_{gluing}$ has rank $18t$.
On the other hand, $\Sigma_\triangle$ is a genus $t+1$ surface with $4t$ punctures.
If we choose a distinguished puncture to give a single gate, the corresponding localised decorated internal skein algebra will be a rank $6t+1$ quantum torus, see Figure \ref{fig:coordinates-for-sigma-tri}.
Opening a gate on each of the other $4t-1$ punctures introduces $4t-1$ new independent skeins, while opening an additional two gates per puncture introduces $8t$ final new skeins.
In total, we find that $\SkAlg^{int}_{\Rep_qT^{\boxtimes 12t}}(\Sigma_\triangle)[S^{-1}]$ is a rank $18t = (6t+1) + (4t-1) + 8t$ quantum torus.

\begin{figure}
    \centering
    \includegraphics{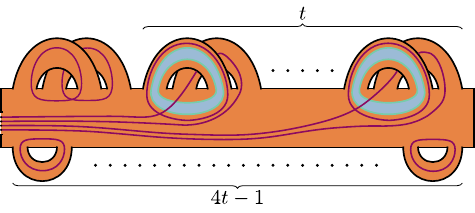}\hspace{7mm}\raisebox{.6\height}{\includegraphics[]{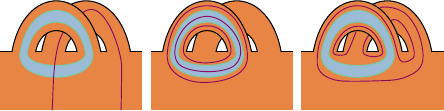}}
    \caption{\textbf{Left:} Generators of the quantum torus $\SkAlg^{int}_{\Rep_qT}(\Sigma_\triangle)[S^{-1}]$, where $S$ is the set of edges passing through the $G$-regions. There are $2t+2$ generators corresponding to meridians and longitudes, $t$ of which pass through $G$-regions. There are $4t-1$ generators corresponding to punctures (the final puncture's monodromy can be expressed in terms of the $T$-region meridian and longitude.)\\
    \textbf{Right:} Skeins around a 2-handle with $G$-annulus. Note that the $G$ monodromy (middle) and the more complicated $T$-monodromy (right) can be expressed in terms of either of the $T$-monodromies (middle) after localising the skein passing through the $G$-region (left).}
    \label{fig:coordinates-for-sigma-tri}
\end{figure}
We conclude that $\varphi$ induces the desired isomorphism $\SkAlg^{int}(\Sigma_\triangle)[S^{-1}_\triangle]\cong \mathbb{W}_\triangle.$
\end{proof}

We finish this section by describing $\SkMod^{int}(\knotcomp^{bulk})[S^{-1}_\triangle]$ as a quotient of $\SkAlg^{int}(\Sigma_{\triangle})[S^{-1}_\triangle]$.

\begin{lemma}
Let $I_{bulk} \subset \SkAlg^{int}(\Sigma_\triangle)$ be the left ideal generated by the bulk relations:
\begin{equation}
    q^{5/2} (A_{i03} a_{i32} a_{i01}) (A_{i12} a_{i23} a_{i10}) + q^{3/2} (A_{i01} a_{i03}^{-1} a_{i12}^{-1}) (A_{i23} a_{i21}^{-1} a_{i30}^{-1}) + q^{3/2}A_{i13}A_{i02},\qquad i = 1,\ldots,t.
\end{equation}
Where the labelling convention for short and long edge indices is as in Figure \ref{fig:gen-ideal-tet}, with the tetrahedra index pretended. 
Then
\begin{equation}
    \SkMod^{int}(\Mbulk)[S_\triangle^{-1}] \cong \SkAlg^{int}(\Sigma_{\triangle})[S_\triangle^{-1}]/I_{bulk}
\end{equation}
\end{lemma}
\begin{proof}
By construction $\Mbulk$ can be built by gluing together copies of the solid ideal tetrahedron $M_{tet}$.
The boundary $\Sigma_{\triangle}$ of $\Mbulk$ can be built out of $t$ copies of $\TetSurf\setminus \DD_3^{\sqcup 4}$.
Since $\SkAlg^{int}(\TetSurf)[S^{-1}_{tet}] \simeq \SkAlg^{int}(\TetSurf\setminus \DD_3^{\sqcup 4})[S^{-1}_{tet}] \simeq \WW_{tet}$, also acts on $M_{tet}$, the claim is that gluing is compatible with the quotient that produces the internal skein module and the localisation needed to get a quantum cluster chart.
Since quotients, localisation, and gluing are all colimits, these operations do in fact commute.
\end{proof}

\subsection[Computing the localised quantum A-ideal]{Computing the localised quantum $A$-ideal}\label{sec:localised-quantum-A-ideal}

This section concerns constructing a $\SkAlg_T(\TetSurf)$-module $\Sk(\knotcomp)^{loc} \simeq \Winv$ from the $\SkAlg^{int}(\Sigma_\triangle)$-module $\Sk^{int}(\Mbulk)\cong \SkAlg^{int}(\Sigma_\triangle)[S^{-1}_\triangle]/I_{bulk}$.

The bipartite surface $\Sigma_\triangle$ has $4t$ punctures corresponding to the vertices of the ideal triangulation, each equipped with 3 gates. These must be closed to pass from $\SkAlg^{int}(\Sigma_\triangle)$ to $\SkAlg(\overline{\Sigma}_\triangle)$. 

Recall that the term ``internal'' in the name ``internal skein algebra'' refers to the fact that $\SkAlg^{int}(\Sigma_\triangle)$ is an algebra object in $\Rep_q T^{\otimes 12t}$.
Each gate has an associated $T$ action, and by `closing a gate' we mean restricting to the subalgebra of elements which are invariant under this corresponding $T$-action.
Note that the action is compatible with localization, so that the quantum cluster chart $\WW_\triangle$ is likewise an element in $\Rep_qT^{\otimes 12t}.$
Let $\Winv \subset \WW_\triangle\cong\SkAlg^{int}(\Sigma_\triangle)[S_\triangle^{-1}]$ be the subalgebra invariant under the $T^{\otimes 12t}$-action.

Recall that $\Sigma_\triangle$ has a $G$-region handle (shown in Figure \ref{fig:the-big-surface}) corresponding to each long edge of the ideal triangulation $\triangle$.
These handles are what allowed us to localise the long edges on our way to a quantum cluster chart, but must be removed in order to obtain the $\SkAlg_{\Rep_qT}(T^2)$-module defining the quantum A-polynomial. 
Their removal consists of a series of 1-handle attachments, which translate to specializing certain monodromies defined using threads.

We start by giving a presentation of the invariants of the quantum cluster chart $\WW_\triangle$ from Theorem \ref{thm:cluster-chart-for-sigma-tri}.

\begin{proposition}\label{prop:invariant-subtorus}
The quantum cluster chart $\WW_\triangle \simeq \SkAlg^{int}(\Sigma_\triangle)[S^{-1}_\triangle]$ restricts to an isomorphism of the $T$-invariant subalgebra $\Winv \simeq \left(\SkAlg^{int}(\Sigma_\triangle)[S^{-1}_\triangle]\right)^{inv}$, which is likewise a quantum torus generated by
\begin{itemize}
    \item \textbf{Longitude $L$ and meridian $M$} - these are determined by the knot $K$ and satisfy $LM = q ML$.
    \item \textbf{Puncture monodromies $\alpha_1,\ldots, \alpha_{4t}$} - these central elements encircle the $4t$ boundary components and generate a unimodular sublattice.
    \item \textbf{thread monodromies $r_1,r_2,\ldots,r_{2t}$} - these correspond to curves in the $T$-region running parallel to the annular $G$ regions and commute with the longitude, the meridian, and each other.
    \item \textbf{Skeletal variables $Y_1,\ldots,Y_{t-1}$} - these are ratios of non-equivalent long edges, scaled by threads/short edges to have weight zero.
\end{itemize}
\end{proposition}

We've given more generators than necessary in Proposition \ref{prop:invariant-subtorus}. The thread monodromies come in equivalent pairs. When $t=2$ and there's an additional relation which allows us to write one of the thread monodromies as a product of the others together with the longitude.

We establish a few relevant lemmas before proving Proposition \ref{prop:invariant-subtorus}.

\begin{lemma}\label{lemma:thread-monodromies}
    Let $r_1,\ldots,r_{2t} \in \SkAlg^{int}(\Sigma_\triangle)$ be the \textbf{thread monodromies}, i.e. skeins represented by the essential curves in the $T$-region which run parallel to the $B$-defects. Then
    \begin{enumerate}
        \item Each $r_i$ is represented by a unique minimal-degree monomial in the threads.
        \item Thread monodromies come in parallel pairs $r_{2j-1},r_{2j}$ of equal degree, which are equivalent upon localising any of the long edge skeins that they cross.
    \end{enumerate}
    \end{lemma}
\begin{figure}
    \centering
    \includegraphics[width=.8\linewidth]{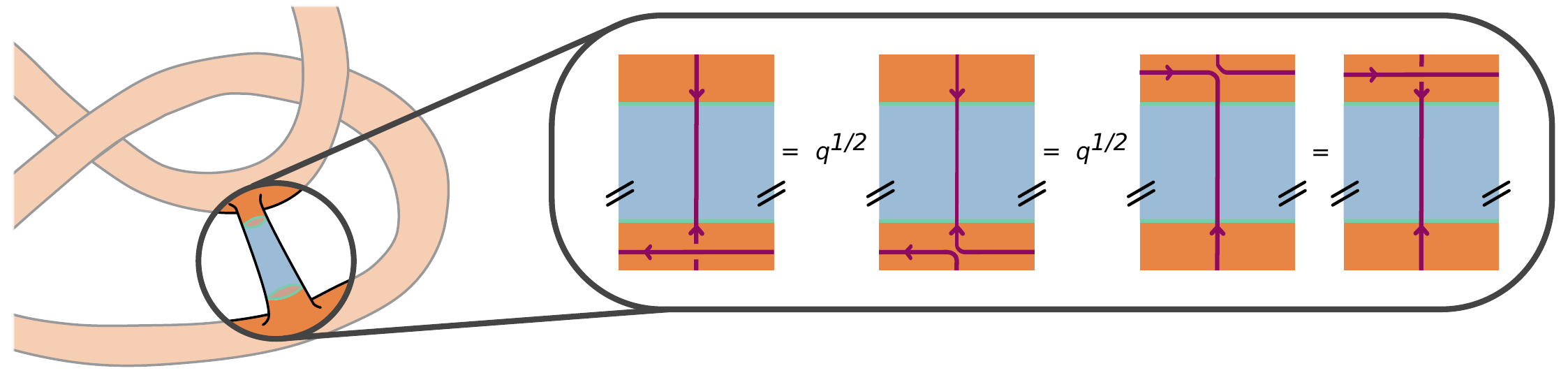}
    \caption{Each $G$-region handle has an associated pair of thread monodromies crossed by the long edge skeins passing through that region. Multiplying by the long edge skein allows us to pass the monodromy from one side of the $G$-region to the other.}
    \label{fig:pairs-of-thread-monodromies}
\end{figure}
\begin{proof}
Recall the pairs of identified long edges $P_{lg} = \{(E_1,F_1),\ldots,(E_{6t},F_{6t})\}$.
Each long edge appears in two faces and is therefore listed twice in $P_{lg}.$
Suppose without loss of generality that $F_{1} = E_{2}, F_{2} = E_{3},\ldots, F_{m}=E_{1}$ is a cycle of identified long edges, then $\theta_{E_{1}}+\cdots+\theta_{E_{m}}$ and $\theta_{F_{1}}+\cdots+\theta_{F_{m}}$ will both have weight zero. Hence
$V^{\theta_{E_{1}}+\cdots + \theta_{E_{m}}}$ and $ V^{\theta_{F_{1}}+\cdots + \theta_{F_{m}}}$
are a pair of equal degree monomials in the threads represented by thread monodromies. They are by construction parallel to the two $B$-defects enclosing a single annular $G$-region, see Figure \ref{fig:pairs-of-thread-monodromies}.

These are minimal and unique because each gate has exactly one long edge, which has weight +1, and exactly two adjacent threads\footnote{If the long edge is not identified with any others in the triangulation, then this will be the two half-edges of a single thread}, one with weight +1 and the other with weight -1.

Next, the equivalence between pairs of skeins can be computed directly using the $T$-relations of Figure \ref{fig:kauffman-crossings}.
We've done this computation in Figure \ref{fig:pairs-of-thread-monodromies}.

\end{proof}

\begin{lemma}
Let $\Gamma_{\alpha} \subset \Gamma_{inv}$ be the sublattice generated by the \textbf{puncture monodromies:}
\begin{equation}
    \begin{aligned}
        \varepsilon_{i0} := e_{i01} + e_{i02} + e_{i03}&& \varepsilon_{i1} := e_{i10} + e_{i12} + e_{i13} \\ \varepsilon_{i2} := e_{i20} + e_{i21} + e_{i23} && \varepsilon_{i3} := e_{i30} + e_{i31} + e_{i32}
    \end{aligned}
\end{equation}
for $i =1,\ldots,t$, using the indexing conventions of Figure \ref{fig:gen-ideal-tet} with the tetrahedra index pretended.
We claim this is a unimodular sublattice whose associated quantum torus is central in $\WW^{inv}_\triangle$.
\end{lemma}
The skein $\alpha_{ij} := V^{\varepsilon_{ij}}$ is represented by a $T$-region curve parallel to the $j$th puncture on the $i$th tetrahedron, hence the name.
\begin{proof}
    Unimodularity follows from the fact that the generators are linearly independent and not multiples of other elements in $\Lambda_{inv}$.
    
    Their centrality follows from the the fact that $T$-invariant skeins can be isotoped away from the punctures and hence can be made disjoint from the puncture monodromies. Disjoint skeins always commute. 
\end{proof}

We're now ready to prove Proposition \ref{prop:invariant-subtorus}.

\begin{proof}[Proof of Proposition \ref{prop:invariant-subtorus}]
The quantum cluster chart restricts to the invariant subalgebras by construction - we define the $T$ weights on the abstract quantum torus to match those of the skeins they model.

Excising the $G$-region annuli for a moment, we arrive at a $T$-region torus with $6t$ punctures. This will have a skein algebra generated the meridian, longitude, $4t$ puncture monodromies, and $2t$ thread monodromies.

With the $G$-regions, we must consider long edge skeins. 
Consider again the presentation of the decorated internal skein algebra of $\Sigma_\triangle$ with just one gate in Figure \ref{fig:coordinates-for-sigma-tri}.
The skeins $A_1,\ldots,A_{t}$ passing through the $t$ annular $G$-regions have weight +2 at this gate. Their formal inverses $A_i^{-1}$ are not skeins, but do have weight $-2$ for the $T$-action associated with the gate.
In fact, these are the only elements with weight $-2$ at a single gate -- any skein in the $T$-region is either invariant or has opposite weights at a pair of gates.

Only invariant non-trivial elements passing through the $G$-regions are therefore monomials in $A_i$'s and their formal inverses, which has with total degree zero.
One can verify by direct computation that these monomials are all products of the expressions $Y_{j} := A_tA_j^{-1}$ for $j = 1,\ldots,t$.
\end{proof}

Now we specialize the thread and puncture monodromies. 
    Let $\Gamma_{mon}\subset \Lambda_{inv}$ be the sublattice generated by the puncture and thread monodromies, and $\WW_{mon}\subset \WW_\triangle^{inv}$ be the associated quantum subtorus. Define a character $\chi_{mon} : \Gamma_{mon} \to \K$ by $\varepsilon_{ij} \mapsto q^{3/2}$, $r_{2k-1} \mapsto q^{1/2}q^{|r_{2k-1}|/2}$, $r_{2k} \mapsto q^{-1/2}q^{|r_{2k}|/2}$, where $|\cdot|$ denotes the number of threads in the thread monodromy.
    The appearance of $\pm 1/2$ in the value on thread monodromies comes from their relationship with the $G$-monodromy.

    By Lemma \ref{lemma:central-quotient}, the relative tensor product $\Winv\otimes_{\WW_{mon}} \K_{\chi_{mon}}$ is generated by the remaining variables, $L,M, Y_1,\ldots,Y_{t-1}$.

    We are now prepared to state our main result, which follows by construction and from excision of internal skein modules.
\begin{theorem}\label{thm:its-the-quantum-A-polynomial}
We have an equality of left ideals
\begin{equation}
\Iloc = \C_q[M^{\pm1},L^{\pm1}] \cap \left(\mathcal{I}_{thr} + \mathcal{I}_{bulk}\right),
\end{equation}
between the quantum A-ideal and the joint elimination ideal of the tetrahedron relations and thread relations. 
\end{theorem}
Note that the image of $\SkAlg_{\Rep_qT}(T^2) \hookrightarrow \WW^{inv}_\triangle \otimes_{\WW_{mon}} \K_{\chi_{mon}}$ is the subalgebra generated by the longitude and meridian.

\subsection{Worked examples}\label{sec:worked-examples}

The software package SnapPy \cite{SnapPy} provides triangulations of knot complements.
For our purposes, it is useful to describe the triangulation as a family of set maps  $r_k : \{0,1,2,3\} \to \{0,\ldots,t-1\}$ and $s_{kj}:\{0,1,2,3\} \to \{0,1,2,3\}$ for $k = 0,\ldots, t-1$ and $j=0,1,2,3$.
The set $\{0,1,2,3\}$ indexes the vertices/faces of a tetrahedron, while $\{0,\ldots,t-1\}$ indexes the tetrahedra themselves.
Each $r_k$ describes which tetrahedra each face of the $k$th tetrahedra is glued, so that $r_k(i)$ is the destination tetrahedra for the $i$th face of the $k$ tetrahedra.
Each $s_{kj}$ describes in more detail the gluing map of the $j$th face of the $k$th tetrahedra.
More precisely, $s_{kj}(j)$ specifies the face on the distant tetrahedra  while the other three coordinates describe the the orientation of the gluing.

\subsubsection[The 4-1 knot]{The $4_1$ knot}\label{sec:41-example}
The gluing data for a two-tetrahedra triangulation of the $4_1$ knot complement is:
\begin{equation}\label{eq:triangulation-data}
    \begin{aligned}
        \rho_0 &= (1,1,1,1); &s_{00} &= (0,1,3,2), & s_{01} &= (1,3,0,2),& s_{02} &= (1,0,2,3),& s_{03} &= (2,0,3,1) \\
        \rho_1 &= (0,0,0,0); & s_{10} &= (0,1,3,2), &s_{11} &= (1,3,0,2),& s_{12} &= (1,0,2,3),& s_{13} &= (2,0,3,1)  
    \end{aligned}
\end{equation}
Where the $i$th component denotes the value of the map on $i$. 
Together $\rho_0$ and $\rho_1$ tell us that every face of the index 0 (resp. 1) tetrahedron is glued to some face in the index 1 (resp. 0) tetrahedron.
From $s_{00}(0) = 0, s_{01}(1) = 3, s_{02}(2) = 2,$ and $s_{03}(3) = 1$, we learn where each face in the index 0 tetrahedron is glued, e.g. the index zero faces in the two tetrahedra are glued together.

Two gates are connected by a thread precisely when they would be identified in the truncated triangulation. We will therefore name the threads according to which gates they connect, i.e. $\theta_{g_0,g_1}$ has weight +1 at gate $g_0$ and weight -1 at gate $g_1$.

Gates are specified by three indices $t,p,e$: the index $t$ specifies the tetrahedra, the index $p$ specifies the puncture, and the index $e$ specifies which puncture the incident long edge connects.
Let $f_{ti}$ and $f_{tj}$ be two faces of the index $t$ tetrahedron's and $i,j$ are the indices of their unique non-adjacent punctures.
Note that $i\neq j \neq p$.
Then a gate $v_{tpe}$ shares a thread with the vertices of index $\rho_t(i),s_{ti}(p), s_{ti}(e)$ and $\rho_t(j),s_{tj}(p), s_{tj}(e)$.

For the gluing data \eqref{eq:triangulation-data}, we find the following 24 threads:
\begin{equation}\label{eq:ex_41_threads}
    \begin{aligned}
    \theta^{012}_{113}, &&\theta^{113}_{003}, &&\theta^{003}_{112}, &&\theta^{112}_{013}, &&\theta^{013}_{103}, &&\theta^{103}_{012},
    \theta^{002}_{123}, &&\theta^{123}_{032}, &&\theta^{032}_{120}, &&\theta^{120}_{001}, &&\theta^{001}_{110}, &&\theta^{110}_{002},\\
    \theta^{030}_{131}, &&\theta^{131}_{021}, &&\theta^{021}_{130}, &&\theta^{130}_{031}, &&\theta^{031}_{121}, &&\theta^{121}_{030},
    \theta^{102}_{023}, &&\theta^{023}_{132}, &&\theta^{132}_{020}, &&\theta^{020}_{101}, &&\theta^{101}_{010}, &&\theta^{010}_{102},
    \end{aligned}
\end{equation}
together with the short and long edge generators:
\begin{equation}
    \begin{aligned}
e_{001}, e_{002}, e_{003}, e_{010}, e_{013}, e_{012}, e_{020}, e_{021}, e_{023}, e_{030}, e_{032}, e_{031},E_{001}, E_{002}, E_{003}, E_{012}, E_{013}, E_{023}; \\
e_{101}, e_{102}, e_{103}, e_{110}, e_{113}, e_{112}, e_{120}, e_{121}, e_{123}, e_{130}, e_{132}, e_{131}, E_{101}, E_{102}, E_{103}, E_{112}, E_{113}, E_{123}.
    \end{aligned}
\end{equation}

\begin{figure}
    \centering
    \includegraphics{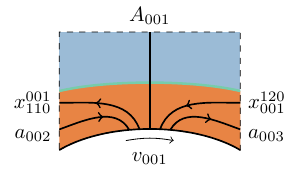}
    \caption{An example of the various generators of $\mathbb{W}_{\Lambda'}$ incident to a single vertex. Shown here for the gate $v_{001}$ in the triangulated $4_1$ knot complement considered in Section \ref{sec:41-example}, but the configuration of short edges, threads, and long edges is the same at each gate regardless of the triangulated knot complement.}
    \label{fig:threads-at-a-gate}
\end{figure}

Next, the gluing sublattice $\Gamma_{gluing}$ is generated by the following terms:
\begin{equation}\label{eq:ex_41_gluing}
\begin{aligned}
    e_{110} + \theta^{112}_{013} + e_{010} + \theta^{012}_{113},&& e_{120} + \theta^{123}_{032} + e_{030} + \theta^{031}_{121},\\
    e_{130} + \theta^{131}_{021} + e_{020} + \theta^{023}_{132},&& e_{101} + \theta^{103}_{012} + e_{013} + \theta^{010}_{102},\\
    e_{121} + \theta^{120}_{001} + e_{003} + \theta^{002}_{123},&& e_{131} + \theta^{132}_{020} + e_{023} + \theta^{021}_{130},\\
    e_{102} + \theta^{101}_{010} + e_{012} + \theta^{013}_{103},&& e_{112} + \theta^{113}_{003} + e_{002} + \theta^{001}_{110},\\
    e_{132} + \theta^{130}_{031} + e_{032} + \theta^{030}_{131},&& e_{103} + \theta^{102}_{023} + e_{021} + \theta^{020}_{101},\\
    e_{113} + \theta^{110}_{002} + e_{001} + \theta^{003}_{112},&& e_{123} + \theta^{121}_{030} + e_{031} + \theta^{032}_{120},\\
    &&\\
    \theta^{012}_{113} -\theta^{131}_{021} - E_{012} + E_{113}, && \theta^{020}_{101} -\theta^{110}_{002} - E_{002} + E_{101}, \\
    \theta^{113}_{003} -\theta^{030}_{131} - E_{113} + E_{003}, && \theta^{123}_{032} -\theta^{023}_{132} - E_{123} + E_{023}, \\
    \theta^{002}_{123} -\theta^{132}_{020} - E_{002} + E_{123}, && \theta^{112}_{013} -\theta^{031}_{121} - E_{112} + E_{013}, \\
    \theta^{103}_{012} -\theta^{021}_{130} - E_{103} + E_{012}, && x^{121}_{030} -x^{003}_{112} - E_{112} + E_{003}, \\
    \theta^{001}_{110} -x^{101}_{010} - E_{001} + E_{101}, && \theta^{013}_{103} -\theta^{130}_{031} - E_{013} + E_{103}, \\
    \theta^{010}_{102} -\theta^{120}_{001} - E_{001} + E_{102}, && \theta^{032}_{120} -\theta^{102}_{023} - E_{023} + E_{102}.
\end{aligned}
\end{equation}

We quotient by the gluing relations, restrict to the $T$-invariant subalgebra, then quotient by the puncture monodromies to get a presentation of $\SkAlg(\overline{\Sigma}_\triangle)[S^{-1}_\triangle]$. This is generated by the following elements, where we've used the notation $x := X^{\theta}, a := X^{e}, A := X^{E}$ when translating from lattice coordinates.
\begin{description}
  \item[Meridian and Longitude:]
\begin{equation}
  \begin{aligned}
    M &= q^{-1/2} a_{103}^{-1} x^{010}_{102} x^{101}_{010} \\
    L &= q^{-3/2} a_{013}^{-1} a_{103}^{-1} a_{112}^{-1} a_{120}^{-1} a_{130}^{-1} x^{012}_{113} x^{020}_{101} x^{030}_{131} x^{002}_{123} x^{010}_{102} x^{132}_{020} (x^{101}_{010})^{2} x^{110}_{002} x^{121}_{030}
  \end{aligned}
\end{equation}
\item[Thread monodromy:] 
\begin{equation}
    r = x^{012}_{113} x^{003}_{112} x^{013}_{103} x^{112}_{013} x^{103}_{012} x^{113}_{003}
\end{equation}
\item[Skeletal variable:] 
\begin{equation}
    Y= a_{001} a_{131}^{-1} A_{113} A_{123}^{-1} (x^{002}_{123})^{-1} (x^{021}_{130})^{-1} (x^{131}_{021})^{-1} (x^{113}_{003})^{-1}
\end{equation}
  \end{description}

These are subject to the following commutation relations:
\begin{equation}\begin{aligned}
  L M = q M L,\hspace{1cm}&&
  Y L = q^2 L Y,\hspace{1cm}&&
 Y r = q^{-1} r Y.
 \end{aligned}
 \end{equation}

 The bulk relations are 
 \begin{equation}\begin{aligned}
   B_1 =  q^{23/2} M^{-2} L^{-1} r^{-3} Y^{-1}+ q M^2 r Y  + q^{3/2},\hspace{1cm}&&
   B_2 = q^{1/2} L r Y + q r^{-1} Y^{-1} + q^{3/2}
 \end{aligned}\end{equation}

\subsubsection[The 4-1 knot after a Pachner move]{The $4_1$ knot after a Pachner move}
We consider a different triangulation of the $4_1$ knot complement with three tetrahedra.
This new triangulation was obtained by performing a $(2,3)$-Pachner move on face $\#3$ of tetrahedron $\#0$ in the triangulation considered just above.
As explained in Section \ref{sec:41-example}, we have the following gluing maps:
\begin{equation}
    \begin{aligned}
        \rho_0 &= (1, 2, 1, 2) & s_{00} &= (1, 3, 0, 2) & s_{01} &= (2, 0, 3, 1) & s_{02} &= (0, 1, 3, 2) & s_{03} &= (0, 1, 3, 2) \\
        \rho_1 &= (2, 0, 2, 0) & s_{10} &= (1, 2, 3, 0) & s_{11} &= (2, 0, 3, 1) & s_{12} &= (0, 1, 3, 2) & s_{13} &= (0, 1, 3, 2)\\
        \rho_2 &= (0, 1, 0, 1) & s_{20} &= (1, 3, 0, 2) & s_{21} &= (3, 0, 1, 2) & s_{22} &= (0, 1, 3, 2) & s_{23} &= (0, 1, 3, 2)
    \end{aligned}
\end{equation}
We have the following generators for $\SkAlg(\overline{\Sigma}_\triangle)[S^{-1}_\triangle]$:

We have the following generators for $\SkAlg(\overline{\Sigma}_\triangle)[S^{-1}_\triangle]$:
\begin{description}
  \item[Meridian and Longitude:]
  \begin{equation}
    \begin{aligned}
      M &= q^{-1/2} a_{231}^{-1} x^{020}_{232} x^{230}_{020} \\
      L &= q^{-3/2} a_{021}^{-1} a_{103}^{-1} a_{121}^{-1} a_{131}^{-1} a_{202}^{-1} a_{231}^{-2} x^{023}_{102} (x^{020}_{232})^{2} x^{030}_{120} x^{123}_{230} x^{101}_{201} x^{113}_{212} x^{130}_{220} x^{212}_{030} x^{203}_{132} x^{220}_{113} (x^{230}_{020})^{3}
    \end{aligned}
  \end{equation}
\item[Thread monodromies:]
  \begin{equation}
    \begin{aligned}
      r_1 &= q^{1/2} x^{012}_{130} x^{032}_{213} x^{030}_{120} x^{120}_{032} x^{113}_{212} x^{130}_{220} x^{212}_{030} x^{220}_{113} x^{213}_{012}
\\
      r_2 &= q^{-1/2} x^{010}_{210} x^{110}_{010} x^{210}_{110}
    \end{aligned}
  \end{equation}

\item[Skeletal variable:]
  \begin{equation}
    \begin{aligned}
      Y_1 &= q^{-3/2} a_{001} a_{223}^{-1} a_{232}^{-1} A_{213} A_{223}^{-1} x^{003}_{221} (x^{020}_{232})^{-1} x^{032}_{213} x^{030}_{120} x^{120}_{032} x^{113}_{212} x^{212}_{030} x^{223}_{002} x^{220}_{113} (x^{230}_{020})^{-1} \\
      Y_2 &= q^{-1/2} a_{002} a_{212}^{-1} A_{201} A_{213}^{-1} (x^{003}_{221})^{-1} (x^{021}_{231})^{-1} (x^{131}_{202})^{-1} (x^{103}_{021})^{-1} (x^{201}_{001})^{-1} (x^{202}_{103})^{-1} (x^{221}_{131})^{-1}
    \end{aligned}
  \end{equation}

\end{description}

These are subject to the following commutation relations:
\begin{equation}\begin{aligned}
  L  M &= q  M  L, \\
  Y_1  M = q  M  Y_1 &\qquad
  Y_1  L = q^2  L  Y_1 \\
  Y_1  r_1 = q^{-1}  r_1  Y_1 \qquad
  Y_2  r_1 &= q  r_1  Y_2 \qquad
  Y_2  r_2 = q^{-1}  r_2  Y_2 \\
 \end{aligned}
 \end{equation}

The scaled localised bulk relations, i.e. the generators of $I_{bulk}[S^{-1}_\triangle]\otimes_{\WW_{mon}} \K_{\chi_{mon}}$, are
\begin{equation}
  \begin{aligned}
    B_1 &=q^{3}  M^2  L^{-1}  r_{1}^{-2}  r_{2}^{-1}  Y_1^2  Y_2 + q^{5/2}  M^2  L^{-1}  r_{1}^{-2}  Y_1^2 + q^{3/2}\\
    B_2 &= q^{3}  M^{-2}  L  r_1^2  r_2  Y_{1}^{-1}  Y_2 + q^{1/2}  M^{-4}  L  Y_{1}^{-1} + q^{3/2} \\
    B_3 &= q^{3/2}  M^2  r_1^2  r_2^2  Y_{1}^{-1} + q^{4}  r_1^2  r_2  Y_{1}^{-1}  Y_2 + q^{3/2} 
  \end{aligned}
\end{equation}

\subsubsection[The 3-1 knot]{The $3_1$ knot}
The trefoil is notable for being non-hyperbolic, meaning one could worry that the hyperbolic geometry based methods introduced in \cite{Dimofte2011} on which we build might break down.
This is fortunately not an issue.

Using the gluing data notation introduced in Section \ref{sec:41-example}, we consider the following triangulation of the trefoil knot complement:
\begin{equation}
    \begin{aligned}
        \rho_0 &= (1, 1, 1, 1) & s_{00} &= (0, 1, 3, 2) & s_{01} &= (3, 2, 0, 1) & s_{02} &= (2, 0, 3, 1) & s_{03} &= (2, 0, 3, 1) \\
        \rho_1 &= (0, 0, 0, 0) & s_{10} &= (0, 1, 3, 2) & s_{11} &= (1, 3, 0, 2) & s_{12} &= (2, 3, 1, 0) & s_{13} &= (1, 3, 0, 2)
    \end{aligned}
\end{equation}
We have the following generators for $\SkAlg(\overline{\Sigma}_\triangle)[S^{-1}_\triangle]$:
\begin{description}
  \item[Meridian and Longitude:]
  \begin{equation}
    \begin{aligned}
      M &= q^{1/2} a_{030}^{-1} x^{032}_{110} x^{110}_{031}\\
      L &= q^{9/2} a_{012}^{-1} a_{020}^{-1} a_{030}^{-3} a_{103}^{-1} a_{113}^{-1} (x^{032}_{110})^{4} x^{002}_{123} x^{010}_{102} x^{021}_{130} x^{112}_{013} x^{123}_{032} x^{101}_{023} x^{130}_{002} (x^{110}_{031})^{3}
    \end{aligned}
  \end{equation}
\item[Thread monodromy:]
  \begin{equation}
    \begin{aligned}
      r &= x^{012}_{113} x^{023}_{132} x^{020}_{103} x^{013}_{101} x^{030}_{112} x^{112}_{013} x^{103}_{012} x^{132}_{020} x^{101}_{023} x^{113}_{030}
    \end{aligned}
  \end{equation}

\item[Skeletal variable:]
  \begin{equation}
    \begin{aligned}
      Y &= a_{013} a_{123}^{-1} A_{102} A_{123}^{-1} (x^{003}_{131})^{-1} x^{020}_{103} (x^{002}_{123})^{-1} (x^{021}_{130})^{-1} (x^{131}_{021})^{-1} x^{103}_{012} x^{132}_{020} (x^{130}_{002})^{-1} (x^{102}_{010})^{-1} (x^{121}_{003})^{-1}
    \end{aligned}
  \end{equation}

\end{description}

These are subject to the following commutation relations:
\begin{equation}\begin{aligned}
  L  M = q  M  L,\qquad
 Y  L = q^{-1}  L  Y,\qquad
 Y  r = q  r  Y,
 \end{aligned}
 \end{equation}

The scaled localised bulk relations, i.e. the generators of $I_{bulk}[S^{-1}_\triangle]\otimes_{\WW_{mon}} \K_{\chi_{mon}}$, are
\begin{equation}
  \begin{aligned}
    B_1 &= q^{5} M^{-4} L r^{-1} + q^{7} r^{-1} Y + q^{3/2} \\
    B_2 &= q^{-9/2} M^6 L^{-1} Y^{-1} + q^{-5} M^2 r Y^{-1} + q^{3/2}
  \end{aligned}
\end{equation}

\subsubsection[The 5-1 knot]{The $5_1$ knot}

Using the gluing data notation introduced in Section \ref{sec:41-example}, we consider the following triangulation of the $5_1$ knot complement:
\begin{equation}
    \begin{aligned}
\rho_{0} &= (1,1,2,2) & s_{00} &= (0,1,3,2) & s_{01} &= (1,3,0,2) & s_{02} &= (0,1,3,2) & s_{03} &= (0,1,3,2)\\ \rho_{1} &= (0,2,2,0) & s_{10} &= (0,1,3,2) & s_{11} &= (2,0,3,1) & s_{12} &= (0,2,1,3) & s_{13} &= (2,0,3,1)\\ \rho_{2} &= (1,1,0,0) & s_{20} &= (1,3,0,2) & s_{21} &= (0,2,1,3) & s_{22} &= (0,1,3,2) & s_{23} &= (0,1,3,2)
    \end{aligned}
\end{equation}
We have the following generators for $\SkAlg(\overline{\Sigma}_\triangle)[S^{-1}_\triangle]$:
\begin{description}
  \item[Meridian and Longitude:]
  \begin{equation}
    \begin{aligned}
      M &= q^{1/2} a_{021}^{-1} a_{131}^{-1} x^{023}_{132} x^{130}_{230} x^{230}_{020}\\
      L &= q^{31} a_{021}^{-8} a_{101}^{-1} a_{120}^{-1} a_{131}^{-9} a_{232}^{-1} (x^{023}_{132})^{9} x^{030}_{220} x^{002}_{203} (x^{130}_{230})^{9} x^{102}_{023} x^{110}_{002} x^{121}_{030} x^{231}_{123} x^{203}_{103} x^{220}_{110} (x^{230}_{020})^{8}
    \end{aligned}
  \end{equation}
\item[Thread monodromies:]
  \begin{equation}
    \begin{aligned}
      r_1 &= q^{1/2} x^{012}_{113} x^{023}_{132} x^{132}_{213} x^{113}_{223} x^{102}_{023} x^{223}_{102} x^{213}_{012} \\
      r_2 &= q^{-1/2} x^{031}_{121} x^{030}_{220} x^{002}_{203} x^{103}_{221} x^{110}_{002} x^{121}_{030} x^{203}_{103} x^{220}_{110} x^{221}_{031}
    \end{aligned}
  \end{equation}

\item[Skeletal variables:]
  \begin{equation}
    \begin{aligned}
      Y_1 &= q^{-1/2} a_{001} a_{213} a_{212} a_{231} A_{212}^{-1} A_{223}\\
          &\times (x^{023}_{132})^{-1} (x^{020}_{101})^{-1} (x^{002}_{203})^{-1} (x^{103}_{221})^{-1} (x^{132}_{213})^{-1} (x^{101}_{202})^{-1} (x^{102}_{023})^{-1} (x^{223}_{102})^{-1} (x^{203}_{103})^{-1} (x^{230}_{020})^{-1} (x^{202}_{003})^{-1}\\
      Y_2 &= q^{1} a_{002} a_{212}^{-1} a_{231}^{-1} A_{201} A_{223}^{-1} x^{023}_{132} x^{020}_{101} x^{132}_{213} x^{101}_{202} x^{102}_{023} x^{223}_{102} (x^{201}_{001})^{-1} x^{230}_{020} x^{202}_{003}
    \end{aligned}
  \end{equation}

\end{description}

These are subject to the following commutation relations:
\begin{equation}\begin{aligned}
  L  M = q  M  L,\qquad
  Y_1  M &= q  M  Y_1,\qquad
 Y_2  M = q^{-1}  M  Y_2,\\
  Y_1  L = q^{10}  L  Y_1,&\qquad
 Y_2  L = q^{-9}  L  Y_2,\\
  Y_1  r_1 = q^{-1}  r_1  Y_1,&\qquad
  Y_2  r_1 = q  r_1  Y_2,\\
  Y_1  r_2 &= q  r_2  Y_1,
 \end{aligned}
 \end{equation}

The scaled localised bulk relations, i.e. the generators of $I_{bulk}[S^{-1}_\triangle]\otimes_{\WW_{mon}} \K_{\chi_{mon}}$, are
\begin{equation}
  \begin{aligned}
    B_1 &= q^{-5/2}  M^{11}  L^{-1}  r_1  r_2  Y_1^2  Y_2 + q^{2}  M^2  r_1  Y_1 + q^{3/2}\\
    B_2 &= q^{-7/2}  M^{-10} L  r_1^{-2}  r_2^{-2}  Y_1^{-2} + q^{3}  r_1  Y_1^{-1} + q^{3/2}\\
    B_3 &= q^{2}  M^{-2}  r_1^{-3}  r_2^{-2}  Y_1^{-1} + q^{2}  M^{-1}  r_1^{-2}  r_2^{-1}  Y_1  Y_2 + q^{3/2}
  \end{aligned}
\end{equation}

\subsubsection[The 5-2 knot]{The $5_2$ knot}

Using the gluing data notation introduced in Section \ref{sec:41-example}, we consider the following triangulation of the $5_2$ knot complement:
\begin{equation}
    \begin{aligned}
\rho_{0} &= (1,1,2,2) & s_{00} &= (0,1,3,2) & s_{01} &= (1,3,0,2) & s_{02} &= (0,1,3,2) & s_{03} &= (1,3,0,2)\\ \rho_{1} &= (0,2,2,0) & s_{10} &= (0,1,3,2) & s_{11} &= (0,1,3,2) & s_{12} &= (1,3,0,2) & s_{13} &= (2,0,3,1)\\ \rho_{2} &= (1,1,0,0) & s_{20} &= (2,0,3,1) & s_{21} &= (0,1,3,2) & s_{22} &= (2,0,3,1) & s_{23} &= (0,1,3,2)
    \end{aligned}
\end{equation}
We have the following generators for $\SkAlg(\overline{\Sigma}_\triangle)[S^{-1}_\triangle]$:
\begin{description}
  \item[Meridian and Longitude:]
  \begin{equation}
    \begin{aligned}
      M &= q^{1/2} a_{021}^{-1} a_{131}^{-1} x^{023}_{132} x^{130}_{230} x^{230}_{020} \\
      L &= q^{31} a_{021}^{-8} a_{101}^{-1} a_{120}^{-1} a_{131}^{-9} a_{232}^{-1} (x^{023}_{132})^{9} x^{030}_{220} x^{002}_{203} (x^{130}_{230})^{9} x^{102}_{023} x^{110}_{002} x^{121}_{030} x^{231}_{123} x^{203}_{103} x^{220}_{110} (x^{230}_{020})^{8}
    \end{aligned}
  \end{equation}
\item[Thread monodromies:]
  \begin{equation}
    \begin{aligned}
      r_1 &= q^{1/2} x^{012}_{113} x^{023}_{132} x^{132}_{213} x^{113}_{223} x^{102}_{023} x^{223}_{102} x^{213}_{012}\\
      r_2 &= q^{-1/2} x^{031}_{121} x^{030}_{220} x^{002}_{203} x^{103}_{221} x^{110}_{002} x^{121}_{030} x^{203}_{103} x^{220}_{110} x^{221}_{031}
    \end{aligned}
  \end{equation}

\item[Skeletal variables:]
  \begin{equation}
    \begin{aligned}
      Y_1 &= q^{-1/2} a_{001} a_{213} a_{212} a_{231} A_{212}^{-1} A_{223} \\
          &\times (x^{023}_{132})^{-1} (x^{020}_{101})^{-1} (x^{002}_{203})^{-1} (x^{103}_{221})^{-1} (x^{132}_{213})^{-1} (x^{101}_{202})^{-1} (x^{102}_{023})^{-1} (x^{223}_{102})^{-1} (x^{203}_{103})^{-1} (x^{230}_{020})^{-1} (x^{202}_{003})^{-1}\\
      Y_2 &= q^{1} a_{002} a_{212}^{-1} a_{231}^{-1} A_{201} A_{223}^{-1} x^{023}_{132} x^{020}_{101} x^{132}_{213} x^{101}_{202} x^{102}_{023} x^{223}_{102} (x^{201}_{001})^{-1} x^{230}_{020} x^{202}_{003}
    \end{aligned}
  \end{equation}

\end{description}

These are subject to the following commutation relations:
\begin{equation}\begin{aligned}
  L  M = q  M  L,\qquad
  Y_1  M &= q  M  Y_1,\qquad
 Y_2  M = q^{-1}  M  Y_2,\\
         Y_1  L = q^{10}  L  Y_1,&\qquad
 Y_2  L = q^{-9}  L  Y_2,\\
         Y_1  r_1 = q^{-1}  r_1  Y_1,&\qquad
 Y_2  r_1 = q  r_1  Y_2,\\
         Y_1  r_2 &= q  r_2  Y_1
 \end{aligned}
 \end{equation}

The scaled localised bulk relations, i.e. the generators of $I_{bulk}[S^{-1}_\triangle]\otimes_{\WW_{mon}} \K_{\chi_{mon}}$, are
\begin{equation}
  \begin{aligned}
    B_1 &= q^{-5/2}  M^{11}  L^{-1}  r_1  r_2  Y_1^2  Y_2 + q^{2}  M^2  r_1  Y_1 + q^{3/2}\\
    B_2 &= q^{-7/2}  M^{-10}  L  r_1^{-2}  r_2^{-2}  Y_1^{-2} + q^{3}  r_1  Y_1^{-1} + q^{3/2}\\
    B_3 &= q^{2}  M^{-2} r_1^{-3}  r_2^{-2}  Y_1^{-1} + q^{2}  M^{-1}  r_1^{-2}  r_2^{-1}  Y_1  Y_2 + q^{3/2}
  \end{aligned}
\end{equation}

\subsubsection[The 8-9 knot]{The $8_9$ knot}
Here we consider a knot with eight crossings. 
Using the gluing data notation introduced in Section \ref{sec:41-example}, we consider the following triangulation of the $8_9$ knot complement:
\begin{equation}
    \begin{aligned}
\rho_{0} &= (1,2,3,3) & s_{00} &= (0,1,3,2) & s_{01} &= (0,1,3,2) & s_{02} &= (0,1,3,2) & s_{03} &= (1,3,0,2)\\ \rho_{1} &= (0,4,5,2) & s_{10} &= (0,1,3,2) & s_{11} &= (0,1,3,2) & s_{12} &= (0,1,3,2) & s_{13} &= (1,0,2,3)\\ \rho_{2} &= (6,0,3,1) & s_{20} &= (0,1,3,2) & s_{21} &= (0,1,3,2) & s_{22} &= (3,2,0,1) & s_{23} &= (1,0,2,3)\\ \rho_{3} &= (2,6,0,0) & s_{30} &= (2,3,1,0) & s_{31} &= (2,3,1,0) & s_{32} &= (2,0,3,1) & s_{33} &= (0,1,3,2)\\ \rho_{4} &= (7,1,5,7) & s_{40} &= (0,1,3,2) & s_{41} &= (0,1,3,2) & s_{42} &= (0,2,1,3) & s_{43} &= (0,2,1,3)\\ \rho_{5} &= (7,4,6,1) & s_{50} &= (2,0,3,1) & s_{51} &= (0,2,1,3) & s_{52} &= (1,0,2,3) & s_{53} &= (0,1,3,2)\\ \rho_{6} &= (2,7,5,3) & s_{60} &= (0,1,3,2) & s_{61} &= (0,1,3,2) & s_{62} &= (1,0,2,3) & s_{63} &= (3,2,0,1)\\ \rho_{7} &= (4,6,5,4) & s_{70} &= (0,1,3,2) & s_{71} &= (0,1,3,2) & s_{72} &= (1,3,0,2) & s_{73} &= (0,2,1,3)
    \end{aligned}
\end{equation}
We have the following generators for $\SkAlg(\overline{\Sigma}_\triangle)[S^{-1}_\triangle]$:
\begin{description}
  \item[Meridian and Longitude:]
  \begin{equation}
    \begin{aligned}
      M &= q^{-1/2} a_{430}^{-1} a_{530}^{-1} a_{631}^{-1} x^{032}_{223} x^{123}_{032} x^{223}_{632} x^{432}_{123} x^{532}_{431} x^{630}_{531} \\
      L &= q^{-9/2} a_{003}^{-1} a_{031}^{-1} a_{102}^{-1} a_{302}^{-1} a_{321}^{-1} a_{332}^{-1} a_{430}^{-2} a_{530}^{-1} a_{631}^{-1} a_{702}^{-1} a_{710}^{-1} a_{732}^{-1} \\
      &\times x^{032}_{223} x^{001}_{301} x^{030}_{320} (x^{123}_{032})^{2} x^{103}_{402} x^{223}_{632} x^{203}_{002} x^{210}_{101} x^{323}_{210} x^{331}_{203} x^{303}_{620} (x^{432}_{123})^{2} x^{402}_{701} x^{421}_{712} (x^{532}_{431})^{2} x^{620}_{730} x^{630}_{531} x^{602}_{330} x^{731}_{421} x^{703}_{602} x^{713}_{532}
    \end{aligned}
  \end{equation}
\item[Thread monodromies:]
  \begin{equation}
    \begin{aligned}
      r_1 &= x^{012}_{113} x^{113}_{512} x^{330}_{012} x^{512}_{703} x^{602}_{330} x^{703}_{602} \\
r_2 &= q^{-1} x^{003}_{202} x^{013}_{312} x^{112}_{013} x^{231}_{621} \\
r_3 &= x^{202}_{112} x^{312}_{231} x^{302}_{003} x^{621}_{302} \\
r_4 &= x^{032}_{223} x^{123}_{032} x^{223}_{632} x^{432}_{123} x^{632}_{723} x^{723}_{432} \\
r_5 &= q^{-1} x^{130}_{520} x^{420}_{130} x^{410}_{720} x^{531}_{710} x^{520}_{410} x^{630}_{531} x^{720}_{630} x^{710}_{420} \\
r_6 &= q^{-1/2} x^{412}_{713} x^{431}_{721} x^{532}_{431} x^{713}_{532} x^{721}_{412} \\
r_7 &= q^{-1/2} x^{102}_{212} x^{212}_{613} x^{403}_{102} x^{503}_{403} x^{613}_{503} \\
r_8 &= q^{-1/2} x^{101}_{501} x^{210}_{101} x^{323}_{210} x^{501}_{610} x^{610}_{323}
    \end{aligned}
  \end{equation}

\item[Skeletal variables:]
  \begin{equation}
    \begin{aligned}
      Y_1 &= q^{-3} a_{001} a_{332}^{-1} a_{623}^{-1} a_{713} a_{723} a_{732}^{-2} A_{703}^{-1} A_{713} \\
      &\times x^{203}_{002} x^{331}_{203} (x^{302}_{003})^{-1} (x^{431}_{721})^{-1} x^{421}_{712} (x^{531}_{710})^{-1} (x^{532}_{431})^{-1} x^{620}_{730} (x^{630}_{531})^{-1} x^{602}_{330} (x^{621}_{302})^{-1} x^{731}_{421} x^{703}_{602} (x^{720}_{630})^{-1} (x^{713}_{532})^{-1}
      \\
Y_2 &= q^{-1/2} a_{002} a_{332}^{-1} A_{313} a_{623} A_{703}^{-1} (x^{313}_{001})^{-1} x^{302}_{003} (x^{620}_{730})^{-1} x^{602}_{330} x^{621}_{302} x^{703}_{602} \\
Y_3 &= a_{030} a_{623} A_{612} a_{731}^{-1} A_{723}^{-1} x^{123}_{032} (x^{213}_{321})^{-1} (x^{321}_{031})^{-1} x^{432}_{123} (x^{612}_{213})^{-1} (x^{620}_{730})^{-1} x^{723}_{432} \\
Y_4 &= q^{1/2} a_{101} a_{632} A_{613}^{-1} A_{702} x^{403}_{102} (x^{401}_{502})^{-1} x^{503}_{403} (x^{502}_{103})^{-1} x^{613}_{503} (x^{720}_{630})^{-1} (x^{702}_{401})^{-1} \\
Y_5 &= q^{-3/2} a_{113} a_{612}^{-1} a_{623} a_{632} A_{601} A_{613}^{-1} a_{713}^{-1} a_{732}\\
&\times  (x^{013}_{312})^{-1} (x^{112}_{013})^{-1} (x^{231}_{621})^{-1} (x^{312}_{231})^{-1} (x^{421}_{712})^{-1} x^{531}_{710} (x^{510}_{110})^{-1} (x^{620}_{730})^{-1} (x^{601}_{510})^{-1} x^{630}_{531} (x^{731}_{421})^{-1} \\
Y_6 &= q^{-3/2} a_{121} a_{632} a_{723} a_{732}^{-1} a_{731}^{-1} A_{713} A_{723}^{-1} (x^{221}_{120})^{-1} (x^{431}_{721})^{-1} x^{432}_{123} (x^{532}_{431})^{-1} (x^{631}_{221})^{-1} x^{723}_{432} (x^{720}_{630})^{-1} (x^{713}_{532})^{-1} \\
Y_7 &= q^{-1} a_{210} a_{623}^{-1} a_{632}^{-1} A_{612}^{-1} A_{613} a_{713} a_{732}^{-1} x^{212}_{613} x^{421}_{712} (x^{531}_{710})^{-1} x^{612}_{213} x^{620}_{730} (x^{630}_{531})^{-1} x^{731}_{421}
    \end{aligned}
  \end{equation}

\end{description}

These are subject to the following commutation relations:
\begin{equation}\begin{aligned}
L M = q M L, &&
Y_3 M = q M Y_3,&&
Y_4 M = q^{-1} M Y_4,&&
Y_5 M = q^{-1} M Y_5,&&
Y_7 M = q M Y_7, \\
Y_1 L = q L Y_1, &&
Y_3 L = q^2 L Y_3, &&
Y_4 L = q^{-2} L Y_4, &&
Y_5 L = q^{-1} L Y_5, &&
Y_7 L = q L Y_7, && \\
Y_1 r_1 = q r_1 Y_1, &&
Y_2 r_1 = q r_1 Y_2, &&
Y_3 r_2 = q^{-1} r_2 Y_3, &&
 Y_7 r_2 = q r_2 Y_7, &&\\
 Y_3 r_3 = q r_3 Y_3, &&
 Y_6 r_3 = q r_3 Y_6, &&
 Y_4 r_4 = q^{-1} r_4 Y_4, &&
 Y_1 r_5 = q^{-1} r_5 Y_1, &&
 Y_6 r_5 = q^{-1} r_5 Y_6, &&\\
 Y_4 r_6 = q r_6 Y_4, &&
 Y_5 r_6 = q r_6 Y_5, &&
 Y_7 r_6 = q^{-1} r_6 Y_7, &&
 Y_5 r_7 = q^{-1} r_7 Y_5, &&\\
 Y_4 Y_1 = q Y_1 Y_4, &&
 Y_5 Y_1 = q Y_1 Y_5, &&
 Y_6 Y_1 = q^{-1} Y_1 Y_6, &&
 Y_7 Y_1 = q^{-1} Y_1 Y_7, && \\
 Y_5 Y_3 = q^{-1} Y_3 Y_5, &&
 Y_6 Y_3 = q Y_3 Y_6, &&
 Y_5 Y_4 = q Y_4 Y_5, &&
 Y_7 Y_4 = q^{-1} Y_4 Y_7, && \\
 Y_6 Y_5 = q Y_5 Y_6, &&
 Y_7 Y_5 = q^{-1} Y_5 Y_7, &&
 Y_7 Y_6 = q Y_6 Y_7, &&
 \end{aligned}
 \end{equation}

The scaled localised bulk relations, i.e. the generators of $I_{bulk}[S^{-1}_\triangle]\otimes_{\WW_{mon}} \K_{\chi_{mon}}$, are
\begin{equation}
  \begin{aligned}
    B_1 &= q^{2} r_1^{-1} r_2^{-1} r_4^{-1} r_5^{-1} r_6^{-1} r_7^{-1} Y_3^{-1} + q^{3/2} r_3 r_5 r_7^{-1} Y_2^{-1} + q^{3/2} \\
    B_2&= q^{2} M^3 r_2^3 r_3^{-1} r_5^{-2} r_6^3 r_7^3 Y_1 Y_5 Y_6^{-1} +  q  M^2 r_2^2 r_3^{-1} r_4 r_5^{-2} r_6 r_7^2 Y_1 Y_3 Y_4 Y_6^{-1} + q^{3/2} \\
    B_3 &=  q  M^{-1} r_2^{-2} r_3 r_5 r_6^{-1} r_7^{-1} Y_1^{-1} Y_2 Y_3^{-1} Y_6 Y_7 + q^{5/2} r_2 r_3 r_7 Y_3^{-1} Y_5 Y_7 + q^{3/2} \\
    B_4 &= q^{2} r_1^{-1} r_3^{-1} r_4^{-1} r_5^{-1} r_6^{-1} Y_5 Y_7 + q^{3/2} r_3^{-1} r_5^{-1} r_7 Y_2^{-1} + q^{3/2} \\
    B_5 &= q^{3/2} M^{-1} r_2^{-1} r_4^{-1} r_5^{-2} r_6^{-3} r_7^{-1} Y_6^{-1} +  q  r_4^{-1} r_5^{-1} Y_4^{-1} + q^{3/2} \\
    B_6 &= q^{-3/2} M L r_2^3 r_5^2 r_6^6 r_7^3 Y_3^{-1} Y_4^{-2} Y_5 Y_6 Y_7^{-1} + q^{-3} M^{-1} L r_2 r_3 r_4^{-1} r_5^2 r_6^4 r_7 Y_1^{-1} Y_3^{-1} Y_4^{-2} Y_6 Y_7^{-1} + q^{3/2} \\
    B_7 &= L^{-1} r_2^{-1} r_6^{-4} r_7^{-1} Y_1 Y_3 Y_4 Y_6^{-1} + q^{2} M^{-1} r_2 r_3 r_6^{-1} r_7 Y_1 Y_5 Y_6^{-1} + q^{3/2} \\
    B_8 &= q^{-1} M^{-1} L r_2 r_3 r_5 r_6^4 r_7 Y_1^{-1} Y_3^{-1} Y_4^{-1} Y_6 Y_7^{-1} + q^{3/2} M^{-1} r_2 r_3^2 r_4 r_6 r_7 Y_6^{-1} + q^{3/2}
  \end{aligned}
\end{equation}

\bibliography{main.bib}

\end{document}